\documentclass[10pt]{article}
\usepackage{geometry}
\usepackage[utf8]{inputenc}
\usepackage[english]{babel}

\usepackage{dsfont} 
\usepackage{amsmath,amssymb,amsthm}
\usepackage[justification=centering]{caption}
\usepackage[colorlinks=true, citecolor=red, linkcolor=blue]{hyperref}
\usepackage[nameinlink, capitalize]{cleveref}

\usepackage[colorinlistoftodos]{todonotes}
\usepackage{varwidth}

\usepackage{color}
\usepackage{float}
\usepackage[super]{nth}
\usepackage{booktabs}
\usepackage{rotating}
\usepackage{verbatim}
\usepackage{multirow}
\usepackage{physics}
\usepackage{mathtools}
\usepackage{subcaption}
\usepackage{bm}
\usepackage[toc,page]{appendix}
\usepackage[shortlabels]{enumitem}
\usepackage{nomencl}

\newcounter{mnotecount}[section]

\theoremstyle{plain}
\newtheorem{theorem}{Theorem}
\newtheorem{corollary}{Corollary}
\newtheorem{lemma}{Lemma}
\newtheorem{proposition}{Proposition}

\newtheorem*{claim*}{Claim}

\theoremstyle{definition}

\newtheorem{remark}{Remark}

\numberwithin{equation}{section}

\title{The Dirac Equation in the Classical Limit: A Duality Approach to a Convergence Rate}

\author{François Golse, Jakob M\"oller}

\begin{document}
\maketitle
\begin{abstract}
We study the semiclassical limit of the Dirac equation with external electromagnetic fields using a dual formulation based on a Wasserstein-1-type pseudo-distance. This pseudo-distance is tailored to matrix-valued density operators and matrix-valued Wigner measures and extends earlier metrics comparing density matrices with classical phase-space distributions in the scalar case. We prove a quantitative convergence rate in the semiclassical parameter, improving previous results for the Dirac equation. Our approach relies on a Lagrange-multiplier formulation of the semiclassical limit for the full matrix-valued Wigner equation (without projection onto energy eigenspaces), introduced in~[F.~Golse, N.~Leopold, N.~Mauser, J.~M\"oller, C.~Saffirio, \textit{J.~Math.~Pures Appl.}, (2026)].
\end{abstract}

\setcounter{tocdepth}{1}
\tableofcontents

\section{Introduction}
We study the semiclassical limit of the Dirac equation with external electromagnetic potentials $A\colon \mathbb{R}_t \times \mathbb{R}^3_x \rightarrow \mathbb{R}^3$ (magnetic potential) and $V\colon \mathbb{R}_t \times \mathbb{R}^3_x \rightarrow \mathbb{R}$ (electric potential). The Dirac equation is given by
\begin{equation}
    i\hbar\partial_t \Psi = \left(\alpha\cdot(-i\hbar \nabla -A)+\beta +V\right)\Psi, \quad \Psi|_{t=0} = \Psi^{\mathrm{in}}.
    \label{eq:DiracEquationWaveFunction}
\end{equation}
Here, $\hbar$ is semiclassical parameter, i.e. the scaled reduced Planck constant.
The unknown $$\Psi(t) \in \mathfrak{H} := L^2(\mathbb{R}^3,\mathbb{C}^4)$$ is a 4-spinor with values in $\mathbb{C}^4$. Here, $\alpha_k$ and $\beta$ are the Dirac (or Gamma) matrices
\begin{align*}
        \alpha_k := \begin{pmatrix}
            0 & \sigma_k \\ \sigma_k & 0 
        \end{pmatrix}, \quad k=1,2,3, && \alpha_4 \equiv \beta := \begin{pmatrix}
            I_{2\times 2} & 0 \\ 0 & -I_{2\times 2}
        \end{pmatrix},
    \end{align*}
    where the $2\times 2$ Pauli matrices $\sigma_k$ are given by
    \begin{align*}
\sigma_1 &:= 
\begin{pmatrix}
0 & 1 \\
1 & 0
\end{pmatrix} ,
\quad 
\sigma_2 := 
\begin{pmatrix}
0 & -i  \\
i & 0  
\end{pmatrix} 
\quad \text{and} \quad 
\sigma_3 := 
\begin{pmatrix}
1 & 0 \\
0 & -1 
\end{pmatrix} .
\end{align*}The Dirac equation is equivalent to the von Neumann equation 
\begin{equation}
\label{eq:vonNeumannDirac}
    i\hbar\partial_t R(t) = [\alpha\cdot(-i\hbar \nabla -A)+\beta +V,R(t)], \quad R(0) = R^{\mathrm{in}}
\end{equation}
for the density operator $$ R(t) \in \mathcal{L}(L^2(\mathbb{R}^3,\mathbb{M}_4(\mathbb{C}))),$$
where $\mathcal{L}(E)$ denotes the algebra of bounded operators on a Hilbert space $E$ and $\mathbb{M}_4(\mathbb{C})$ denotes the algebra of $4\times 4$ matrices with entries in $\mathbb{C}$. 

In general, one can define matrix-valued Hamiltonians as follows. Let $P(x,\xi)\in \mathbb{M}_m(\mathbb{C})$ be an $m\times m$ matrix-valued function on the phase space $\mathbb{R}^d_x \times \mathbb{R}^d_{\xi}$ and let $R$ be the kernel of a matrix-valued density operator\footnote{More generally one can consider $P$ to be a smooth section on a vector bundle on the phase space manifold. Then, via local trivializations and local charts, $P$ can be represented as a $\mathbb{C}^{n\times m}$-valued function. However, we restrict ourselves to flat spacetimes. For a recent reference on the definition of Weyl operators and Wigner functions on curved spacetime see \cite{andersson_pseudodifferential_2026}.}. Then the matrix-valued Weyl operator $H$  with matrix-valued symbol $P$ is defined as
\begin{align*}
(H R) (x,y)= \mathrm{Op}^W_{\hbar}[P](R)(x,y) := \frac{1}{(2\pi\hbar)^d} \iint P(\tfrac{x+z}{2},\xi) e^{i\frac{(x-z)}{\hbar}\cdot \xi} R(z,y) d\xi dz.
\end{align*}
Note that since they are matrix-valued, the order of $P$ and $R$ is important. For the Dirac equation the symbol $P$ is given by
\begin{equation}
\label{eq:DiracHamiltonFunction}
    P(t,x,\xi) = \alpha\cdot( \xi -A(t,x))+\beta +V(t,x)I_4.
\end{equation} Suppose that $R(t)$ is a solution of the von Neumann equation
\begin{equation}
\label{eq:vonNeumannIntro}
i\hbar\partial_tR(t)=[H,R(t)]\,,\qquad R(0)=R^{\mathrm{in}}.
\end{equation}
Then the matrix-valued Wigner transform $W_{\hbar}[R]$ (defined in \eqref{eq:WignerTransform}) of $R$ solves the \emph{Wigner equation}
\begin{equation}
\label{eq:WignerIntro}
    \partial_t W_{\hbar} = \tfrac{1}{i\hbar}[P,W_{\hbar}]+ \tfrac{1}{2}(\{P,W_{\hbar}\}-\{W_{\hbar},P\}) + r_{\hbar}\,,\qquad W_{\hbar}(0)=W_{\hbar}[R^{\mathrm{in}}].
\end{equation}
Here, $r_{\hbar}$ is an $O(\hbar)$ remainder term (the remainder term of the Dirac equation is given explicitly in Section \ref{sec:Remainder}).
Note that in the scalar case the oscillating term $\tfrac{1}{i\hbar}[P,W_{\hbar}]$ vanishes. The Wigner transform converges to the matrix-valued  Wigner measure $F$ in the limit $\hbar\rightarrow 0$, cf. \cite{gerard_homogenization_1997, lions_sur_1993}. As shown in \cite{golse_semi-classical_2026}, the limit of the Wigner equation \eqref{eq:WignerIntro} in the sense of distributions is given by
\begin{align}
\begin{cases}
\partial_t F = \tfrac{1}{i}[P,\mathcal{Y}] + \tfrac{1}{2}  \left( \{P, F \} - \{ F, P \} \right), \qquad F|_{t=0}  = F^{\mathrm{in}}, \label{eq:matrix liouville equation introduction}\\
[P,F] = 0.\end{cases}
\end{align}
where $\mathcal{Y}$ is the Lagrange multiplier maintaining the constraint $[P,F]=0$. This equation can be seen as a \emph{matrix Liouville equation}. Indeed, if $f$ and $P$ are scalar functions, the commutators in \eqref{eq:matrix liouville equation introduction} vanish and by the antisymmetry of the Poisson bracket for scalar functions we obtain the well-known scalar Liouville equation
\begin{align}
\partial_t f =  \{P, f \}, \qquad f|_{t=0}  = f^{\mathrm{in}}. 
\end{align}
The matrix Liouville equation \eqref{eq:matrix liouville equation introduction} contains the whole information about the classical system. In particular, it contains the information about the projected Wigner measures $F_j = \Pi_j F \Pi_j$ where $\Pi_j$ is the projection on the $j$th eigenspace of $P$ (this demands some assumptions on the regularity of the eigenvalues, cf. Section \ref{section:liouville and test equation}) and the scalar Wigner measures $f_j = \tr(\Pi_j F \Pi_j)$. The projected Wigner measures $F_j$ obey the matrix Liouville equation including the Berry phase and no-name term/Poissonian curvature \cite{spohn_semiclassical_2000, chai_multiband_2015, emmrich_geometry_1996, littlejohn_geometric_1991}, while the scalar Wigner measure obeys the classical Liouville equation where the Hamilton function is given by the $j$th eigenvalue $\lambda_j$ of the matrix-valued Hamiltonian $P$. For the Dirac equation in particular, the symbol $P$ has two distinct eigenvalues $\lambda_{\pm}$ and its matrix Liouville equations contains the \emph{relativistic transport equations} for the electron and positron phase space distribution functions $f_{\pm}(t,x,v) := \tr((\Pi_{\pm}F)(t,x,\xi-A(t,x)))$, i.e.
\begin{equation}
    \partial_t f_{\pm} \pm \tfrac{v}{\langle v \rangle} \cdot \nabla_x f_{\pm} + (E \pm \tfrac{v}{\langle v \rangle} \times B)\cdot \nabla_v f_{\pm} = 0, \quad \langle v \rangle := \sqrt{1+|v|^2},
\end{equation}
where $E = -\partial_t A -\nabla V$ and $B=\nabla \times A$ is the electric and magnetic field, respectively.

Inspired by the framework introduced in \cite{golse_semi-classical_2026}, we introduce the following pseudo-distance $\mathfrak{D}(t)$ between the matrix-valued density operator $R$ and the matrix-valued Wigner measure $F$ and show that this pseudo-distance converges to zero with a certain rate of convergence in $\hbar$. The pseudo-distance $\mathfrak{D}(t)$ is given by
\begin{align}
\label{eq:MatrixDistanceIntroduction}
\mathfrak{D}((R(t),F(t)) := \sup_{\mathcal{A}} \left|\Tr(R(t) \mathrm{Op}^T_{\hbar}[\Phi(t,\cdot,\cdot)]) - \iint_{\mathbb{R}^d\times \mathbb{R}^d} \tr(F(t,q,p)\Phi(t,q,p)) dq dp\right|,
\end{align}
where $\mathrm{Op}^T[\cdot]$ denotes the Toeplitz quantization, cf. Section \ref{sec:Notation}. Moreover, $\Tr(\cdot)$ denotes the trace in $L^2(\mathbb{R}^d,\mathbb{M}_{m}(\mathbb{C}))$, where $\mathbb{M}_{m}(\mathbb{C})$ denotes the vector space of complex $m\times m$ matrices and $\tr(\cdot)$ denotes the trace in $\mathbb{M}_{m}(\mathbb{C})$. Here, $\mathcal{A}$ (compare definition \eqref{eq:AlgebraA}) denotes the set of final data  $\Phi^{\mathrm{fin}}$ of Sobolev norm $H^k$ less than one for the solutions of the dual matrix Liouville equation
\begin{align}
\begin{cases}
\label{eq:test equation introduction}
\partial_t \Phi = \tfrac{1}{i}[P,\mathcal{Z}]+ \tfrac{1}{2} \left( \{P, \Phi \} - \{ \Phi, P \} \right), \qquad \Phi|_{t=T}  = \Phi^{\mathrm{fin}}, \\
[P,\Phi] = 0,
\end{cases}
\end{align}
where $\mathcal{Z} \in C([0,T],H^{k-1})$ is the Lagrange multiplier maintaining the constraint $[P,\Phi]=0$. By the wellposedness of the matrix Liouville equation \eqref{eq:test equation introduction} (Theorem \ref{thm:LWPLiouville}), proved in Section \ref{sec:LWPMatrixLiouville}, the solution $\Phi$ is uniquely determined by its final datum $\Phi^{\mathrm{fin}}$. For a detailed discussion of the definition of $\mathfrak{D}(t)$ we refer to Section \ref{sec:Distance}.

We have the following main result.
\begin{theorem}
\label{thm:main}
    Let $A,V \in H^s(\mathbb{R}_t\times \mathbb{R}^d_x)$ with $s$ large enough and let $P=\alpha\cdot( \xi -A(t,x))+\beta +V(t,x)I_4$ be the Dirac Hamilton function \eqref{eq:DiracHamiltonFunction}. Let $R(t)$ be a solution of the von Neumann equation \eqref{eq:vonNeumannIntro} with initial data $R^{\mathrm{in}}$ satisfying
    $$\Tr\left( \sqrt{1 - \hbar^2\Delta} R^{\mathrm{in}} \sqrt{1 - \hbar^2\Delta} \right) < + \infty,$$
    and
    \begin{align}
\label{eq:vonNeumannInitialDataMainThm}
\begin{cases}
    \|R^{\mathrm{in}}\|_{L^2} \leq C(2\pi\hbar)^{d/2}, \\
   \|  \tfrac{i}{\hbar} [-x_j, R^{\mathrm{in}} ] \|_{L^2}  \leq C (2\pi\hbar)^{d/2}, \\
     \| \tfrac{i}{\hbar}[-i\hbar \partial_j, R^{\mathrm{in}} ] \|_{L^2}  \leq C (2\pi\hbar)^{d/2}, \\
    \left\| \tfrac{i}{\hbar}\left[\alpha \cdot (-i\hbar\nabla -A(0)) + \beta + V(0),R^{\mathrm{in}}\right]\right\|_{L^2} \leq C(2\pi\hbar)^{d/2}.
    \end{cases}
\end{align}
\begin{enumerate}
\item Let $\mathfrak{D}(t)$ be the pseudo-distance defined in \eqref{eq:MatrixDistanceIntroduction}, where $\Phi\in C([0,T],H^k)\cap C^1([0,T],H^{k-2})$, $k\geq 2$, $k\in \mathbb{N}$ is a strong solution of the matrix Liouville equation \eqref{eq:test equation introduction} with final datum $\Phi^{\mathrm{fin}}\in H^k$, provided by Theorem \ref{thm:LWPLiouville}. Then the following inequality holds for $\mathfrak{D}(t)$,
\begin{align}
\label{eq:MainGronwallEstimate}
    \mathfrak{D}(t) &\leq \mathfrak{D}^{\mathrm{in}} + \sqrt{\hbar}\, \Xi_2 e^{\Xi_2 t} + \hbar\,  \Xi_4 e^{\Xi_4 t}, \quad t\in [0,T], 
\end{align}
    where the constant $\Xi_2$ depends on $T$, $\|A\|_{W^{2,\infty}_{t,x}}$ and $\|V\|_{W^{2,\infty}_{t,x}}$ while $\Xi_4$ depends on $T$, $\|A\|_{W^{4,\infty}_{t,x}}$ and $\|V\|_{W^{3,\infty}_{t,x}}$.
\item Let $W_{\hbar}[R]$ be the Wigner function of $R$. Then the commutator $[P,W_{\hbar}]$ obeys the following bound:
\begin{align}
    \|[P(t),W_{\hbar}(t)]\|_{L^2} &\leq \hbar\left[C(1+F(\|A\|_{W^{2,\infty}_{t,q}},\|V\|_{W^{2,\infty}_{t,q}})) \exp(Ct(1+F(\|A\|_{W^{2,\infty}_{t,x}},\|V\|_{W^{2,\infty}_{t,x}})))\right]. \label{eq:PropagationCommutator}
\end{align}
where $F$ is a continuous function vanishing at the origin.
\end{enumerate}
\end{theorem}

The proof of the Gronwall estimate \eqref{eq:MainGronwallEstimate} is contained in Section \ref{sec:GronwallDirac}. The proof of the $L^2$ estimate \eqref{eq:PropagationCommutator} of the commutator $[P,W_{\hbar}]$ is contained in Proposition \ref{lemma:GronwallLagrangeWigner} in Section \ref{sec:Lagrange multiplier}.

\begin{remark}
    The fact that $\mathfrak{D}(t)$ is a pseudo-distance is due to the fact that $\mathfrak{D}(R(t),F(t)) = 0$ does not necessarily imply that $\tilde{W}_{\hbar}[R](t) = F(t)$, where $\tilde{W}_{\hbar}[R]$ is the Husimi function. Naturally, since in the pseudo-distance $\mathfrak{D}(t)$ one tests against block-diagonal test functions $\Phi(t)$ (as $\Phi(t)$ commutes with $P(t)$) the pseudo-distance does not ``see'' the off-diagonal blocks which are themselves captured by $[P,W_{\hbar}]$. Theorem \ref{thm:main} provides an explicit convergence rate for the commutator $[P,W_{\hbar}]$ which completes the picture of the convergence of the quantum dynamics generated by the Dirac Hamiltonian towards the matrix Liouville equation. 
\end{remark}

\begin{remark}
\label{RemarkCommutator}
    Since the commutator $[P,W_{\hbar}]$ measures, in some sense, the ``mixing" of the electron and positron components, its vanishing in the semiclassical limit $\hbar \rightarrow 0$ indicates that in the classical equation the electron and positron modes are decoupled (at least for the linear Dirac equation). Estimate \eqref{eq:PropagationCommutator} then gives a rate for this decoupling of the eigenmodes. In the limit, the equation for the projected matrix-valued Wigner measure $F_j := \Pi_j F \Pi_j$ reads
\begin{equation*}
    \partial_t F_j = \{\lambda_j,F_j\} - \left[F_j,[\Pi_j,\{\lambda_j,\Pi_j\}-\partial_t \Pi_j] + \tfrac{1}{2}\textstyle\sum_{k=1}^m (\lambda_j-\lambda_k)\Pi_j \{\Pi_k,\Pi_k\}\Pi_j\right],
\end{equation*}
where the second and third terms on the right-hand side are the \emph{Berry term} and \emph{Poisson curvature} cf. \cite{gerard_homogenization_1997, spohn_semiclassical_2000, golse_semi-classical_2026, emmrich_geometry_1996, littlejohn_geometric_1991}. For the Dirac equation, where $m=2$ and $\lambda_{\pm}(t,x,\xi) = \pm \sqrt{1+|\xi-A(t,x)|^2}+V(t,x)$, this reduces to
\begin{align*}
\begin{cases}
    \partial_t F_{\pm} = \{\lambda_{\pm},F_{\pm}\} - \left[F_{\pm},[\Pi_{\pm},\{\lambda_{\pm},\Pi_{\pm}\}-\partial_t \Pi_{\pm}] \pm\tfrac{1}{2} {\sigma}\Pi_{\pm} \{\Pi_{\pm},\Pi_{\pm}\}\Pi_{\pm}\right], \\
    \sigma \equiv \sigma(t,x,\xi):= |\lambda_+(t,x,\xi)-\lambda_-(t,x,\xi)| = 2\sqrt{1+|\xi-A(t,x)|^2},
\end{cases}
\end{align*}
since $\{\Pi_{\pm},\Pi_{\pm}\} = \{\Pi_{\mp},\Pi_{\mp}\}$ in view of $\Pi_+ + \Pi_- = I_4$. Therefore, at least for the Dirac equation, the limit of the electron mode $F_+$ ``sees" only the electron mode, similarly for the positron mode $F_-$. Note however that is only true for the linear case, i.e. for external potentials $A,V$. In the self-consistent case where $A,V$ are determined by the total densities 
\begin{align*}
    \rho := \int (\tr(F_+)+\tr(F_-)) d\xi, && J := \int \frac{\xi-A}{\langle \xi-A\rangle}(\tr(F_+)-\tr(F_-)) d\xi
\end{align*}
the eigenmodes are coupled (as is the case for the two-species relativistic Vlasov-Maxwell equation for electrons and positrons).
\end{remark}

    
\begin{remark}
    The conditions \eqref{eq:vonNeumannInitialDataMainThm} on the initial density operator $R^{\mathrm{in}}$ essentially correspond to $\Psi$ being in the domain of the Dirac operator $H$, cf. Proposition \ref{thm:existence dirac}. 

Moreover, the conditions the initial density matrix $R^{\mathrm{in}}$ translate as follows to the Wigner transform. One has
\begin{align*}
\partial_tW_{\hbar}[R] = W_{\hbar}[\partial_t R], &&
\nabla_q W_{\hbar}[R] = W_{\hbar}\left[\tfrac{i}{\hbar}[-i\hbar \nabla,R]\right], && \nabla_p W_{\hbar} = W_{\hbar}\left[\tfrac{i}{\hbar}[-x,R]\right],
\end{align*}
We have
\begin{align}
        \begin{cases}
    \|W^{\mathrm{in}}\|_{L^2} \leq C, \\
   \|\partial_{p_j}W^{\mathrm{in}}\|_{L^2}  \leq C, \\
     \| \partial_{q_j}W^{\mathrm{in}}\|_{L^2}  \leq C.
    \end{cases} \quad \Leftrightarrow \quad \begin{cases}
    \|R^{\mathrm{in}}\|_{\mathrm{HS}} \leq C(2\pi\hbar)^{d/2}, \\
   \|\tfrac{i}{\hbar}[x_j,R^{\mathrm{in}}]\|_{\mathrm{HS}}  \leq C(2\pi\hbar)^{d/2}, \\
     \|\tfrac{i}{\hbar}[-i\hbar \partial_j,R^{\mathrm{in}}]\|_{\mathrm{HS}}  \leq C(2\pi\hbar)^{d/2},
    \end{cases}
    \end{align}
while
\begin{align}
    \|\partial_tW_{\hbar}|_{t=0}\|_{L^2} \leq C \quad &\Leftrightarrow \quad \|\partial_t R|_{t=0}\|_{\mathrm{HS}} \leq C(2\pi\hbar)^{d/2} 
    \nonumber \\ \quad &\Leftrightarrow \quad \|\tfrac{i}{\hbar}[H(0),R^{\mathrm{in}}]\|_{\mathrm{HS}} \leq C(2\pi\hbar)^{d/2} \nonumber \\ \quad &\Leftrightarrow \quad \|\tfrac{i}{\hbar}[\alpha \cdot (-i\hbar \nabla -A(0))+ \beta +V(0),R^{\mathrm{in}}]\|_{\mathrm{HS}} \leq C(2\pi\hbar)^{d/2}.
\end{align}
\end{remark}
\begin{remark}  
\label{remark:DilutedMixedStates}
By a change of variables the condition
    \begin{equation*}
        \|R^{\mathrm{in}}\|_{L^2} \leq C (2\pi\hbar)^{d/2}
    \end{equation*}
    is equivalent to
    \begin{equation}
        \label{eq:MixedStateCondition}\|\tilde{R}^{\mathrm{in}}\|_{L^2} \leq C \quad \Leftrightarrow \quad \|W^{\mathrm{in}}\|_{L^2} \leq C.
    \end{equation}
    By the conservation of the $L^2$-norm (Corollary \ref{thm:existence wigner}), this implies
    \begin{equation*}
        \|W_{\hbar}[R(t)]\|_{L^2} \leq C,
    \end{equation*}
    uniformly in $\hbar > 0$ and $t \in \mathbb{R}$.
    Note that \eqref{eq:MixedStateCondition} excludes pure states since it implies that, for a family $R_{\hbar}$ indexed by $\hbar$,
    \begin{equation*}
        \mathrm{rk}(R_{\hbar}) \geq \frac{1}{C(2\pi \hbar)^d}\rightarrow \infty \quad \text{as } \hbar \rightarrow 0.
    \end{equation*}
    For a detailed discussion, see the discussion after Theorem 3.1 in \cite{golse_velocity_2025}.
    In the context of the mean-field limit of $N$-body fermionic systems, the typical mean field scaling is
    \begin{equation}
    \label{eq:MeanFieldScaling}
        \hbar = N^{-1/d}.
    \end{equation}
    Moreover, the Planck constant is small compared to the typical action of the system in a semiclassical scaling. Then, as a consequence of the Pauli exclusion principle, the typical action of the system grows very large as the number $N$ of fermions tends to infinity.
    
    On the other hand, the one-particle reduced density matrix of a $N$-body Slater determinant is a rank-$N$ mixed state. Combined with the mean-field scaling \eqref{eq:MeanFieldScaling} this means that the rank of $R_{\hbar}$ grows as
    \begin{equation}
        \mathrm{rk}(R_{\hbar}) \geq \frac{N}{C(2\pi)^d} .
    \end{equation}
    Consequently, since the Dirac equation is a fermionic equation, condition \eqref{eq:MixedStateCondition} is completely natural. 
\end{remark}

A survey of matrix-valued Wigner measures that includes the semiclassical limit of the Dirac equation for smooth time-independent potentials can be found in \cite{gerard_homogenization_1997}. In \cite{spohn_semiclassical_2000}, the semiclassical limit of spin precession was treated and the classical BMT equation was derived. In both of these references the semiclassical limit of the Dirac equation was studied by projection on the electron and positron eigenspaces. In \cite{golse_semi-classical_2026} the semiclassical limit of the matrix-valued Wigner transform of $R$ was studied without projection on the eigenspaces by introducing a Lagrange multiplier maintaining the commutatitvity constraint of the matrix-valued Hamiltonian function and the matrix-valued Wigner measure. These results hold in the distributional sense for time-dependent potentials with finite regularity, in particular no convergence rate was obtained. In Theorem \ref{thm:main} we are able to give a  rate of convergence for the weak convergence of the Wigner transform $W_{\hbar}$ to the Wigner measure $F$, thereby extending the results of \cite{golse_semi-classical_2026}.

The problem of the semiclassical limit of the Dirac equation has also been studied in \cite{bolte_semiclassical_2004}.

The pseudo-distance introduced in \eqref{eq:MatrixDistanceIntroduction} generalizes earlier notions \cite{golse_random_2019,golse_convergence_2021} between the scalar density operator and the scalar Wigner measure, for instance for the Schrödinger equation with potential, where the Hamiltonian operator and Hamilton function are given by
\begin{equation}
\label{eq:SchrödingerHamiltonian}
    H = -\tfrac{\hbar^2}{2m}\Delta +V, \qquad P(x,\xi) = \tfrac{1}{2m}|\xi|^2 + V(x).
\end{equation}
In that case, the commutator between $P$ and $W_{\hbar}$ vanishes and one does not have to deal with a Lagrange multiplier since the  commutativity constraint is always satisfied.

\section{Proof strategy}
Our starting point is   the pseudo-distance $\mathfrak{D}(t)$, cf. \eqref{eq:MatrixDistanceIntroduction}, which compares the quantum state $R(t)$ with the matrix-valued phase-space measure $F(t)$ by testing against Toeplitz-quantized matrix-valued symbols $\Phi(t,q,p)$. The test symbols are chosen dynamically: for a fixed terminal datum $\Phi^{\mathrm{fin}}$ we solve a backward-in-time matrix Liouville equation \eqref{eq:test equation introduction}, so that the classical pairing $\iint \tr(F(t)\Phi(t))$ is conserved. In the matrix-valued setting the dual equation includes a Lagrange multiplier $\mathcal{L}=\tfrac{1}{i}[P,\mathcal{Z}]$ that maintains the commutativity constraint $[P,\Phi]=0$, cf. \cite{golse_semi-classical_2026}. This allows us to treat the full Wigner dynamics without projecting onto energy eigenspaces. This is done in Section \ref{sec:preliminaries}. 

With this choice of test functions, differentiating $\Tr(R(t)\mathrm{Op}^T_{\hbar}[\Phi(t)])$ in time and using the von Neumann equation for $R(t)$  as well as the dual matrix Liouville equation for the test function $\Phi$ yields an inequality for $\mathfrak{D}(t)$ in which all contributions are explicit commutator/quantization errors. For the Dirac Hamiltonian we then rewrite these errors by splitting $H$ into the free kinetic, magnetic, and electric parts; this isolates first- and second-order finite-difference terms, as well as a ``Husimi extra term'' that can be recast as a commutator with $P$ and therefore absorbed into the Lagrange-multiplier structure. This is done in Section \ref{sec:DistanceDirac}. 

A key step is to bound the Lagrange-multiplier contribution quantitatively, which requires a convergence rate for the commutator $[P,W_{\hbar}]$ of the symbol of the Hamiltonian and the Wigner transform $W_{\hbar}[R]$ of $R$. This is proved in Proposition \ref{lemma:GronwallLagrangeWigner}. Here, we have to introduce covariant derivatives to close the argument, cf. Remark \ref{Remark:Preconditioning}. Collecting all estimates leads to a Gronwall-type inequality for $\mathfrak{D}(t)$ (Section \ref{sec:GronwallDirac}), and wellposedness of the dual matrix Liouville equation (Section \ref{sec:LWPMatrixLiouville}) provides the needed Sobolev bounds on $\Phi$ to close the argument and obtain the stated semiclassical convergence rate.

\section{Preliminaries}
\label{sec:preliminaries}

\subsection{Matrix Liouville equation and dual equation}
\label{section:liouville and test equation}

In this section we suppose that $P$ is some sufficiently regular matrix-valued symbol. The precise conditions on $P$ vary with the choice of quantization, i.e. whether one defines the Hamilton operator $H$ as the Weyl quantization $\mathrm{Op}_{\hbar}^W[P]$, the Toeplitz quantization $\mathrm{Op}_{\hbar}^T[P]$ or any other quantization rule.

The Wigner transform $W_{\hbar}[R]$, defined in \eqref{eq:WignerTransform}, converges in $\mathcal{S}'(\mathbb{R}_t\times \mathbb{R}^d_x \times \mathbb{R}^d_{\xi})$ to a positive definite matrix-valued Radon measure $F$ on $\mathbb{R}^d_x \times \mathbb{R}^d_{\xi}$, called \emph{Wigner measure}, cf. \cite{gerard_homogenization_1997, lions_sur_1993}. If, in addition, $W_{\hbar}[R]$ is bounded in $t$ with values in $L^2(\mathbb{R}^d_x\times \mathbb{R}^d_{\xi})$ uniformly in $\hbar$, i.e.
\begin{equation}
    \sup_{0<\hbar \leq 1}\sup_{t} \|W_{\hbar}[R(t)]\|_{L^2(\mathbb{R}^d_x\times \mathbb{R}^d_{\xi})} \leq C,
\end{equation}
we obtain that $$F\in L^{\infty}(\mathbb{R}_t,L^2(\mathbb{R}^d_x\times \mathbb{R}^d_{\xi})).$$ 
We recall here the theory of the limit of the Wigner equation \eqref{eq:WignerIntro} towards the matrix Liouville equation \eqref{eq:matrix liouville equation introduction} from \cite{golse_semi-classical_2026}. The Wigner measure $F$ is a solution in $H^{-1}_{\mathrm{loc}}(\mathbb{R}_t\times \mathbb{R}^d_x \times \mathbb{R}^d_{\xi},\mathbb{M}_m(\mathbb{C}))$ of the matrix-valued Liouville equation
\begin{align}
\begin{cases}
\partial_t F = \mathfrak{X}+ \tfrac{1}{2}  \left( \{P, F \} - \{ F, P \} \right), \qquad F|_{t=0}  = F^{\mathrm{in}}, \label{eq:matrix liouville equation}\\
[P,F] = 0,\end{cases}
\end{align}
cf. \cite{golse_semi-classical_2026}. Here, $\mathfrak{X} \in H^{-1}_{\mathrm{loc}}(\mathbb{R}_t\times \mathbb{R}^d_x \times \mathbb{R}^d_{\xi},\mathbb{M}_m(\mathbb{C}))$ is a matrix-valued distribution and the Lagrange multiplier associated to the constraint $[P,F]=0$. It is a priori not obvious that it makes sense to prescribe initial data $F^{\mathrm{in}}$ since up to now, the Wigner measure $F$ is a distribution in $t$. However, recall from \cite{golse_semi-classical_2026}  that $\mathfrak{X}$ has the following property. For all $\Phi \in C^{\infty}_0(\mathbb{R}_t\times \mathbb{R}^d_x \times \mathbb{R}^d_{\xi},\mathbb{M}_m(\mathbb{C}))$ (in the case of the Dirac equation one can also take $H^s(\mathbb{R}_t\times \mathbb{R}^d_x \times \mathbb{R}^d_{\xi},\mathbb{M}_m(\mathbb{C}))$ for $s$ sufficiently large),
\begin{align}
\label{eq:PropertyX}
[P,\Phi] = 0 \qquad  \text{implies that} \qquad  \tr_{\mathbb{C}^d}(\Phi \mathfrak{X})  = 0 \qquad \text{in } H^{-1}_{\mathrm{loc}}(\mathbb{R}_t\times \mathbb{R}^d_x \times \mathbb{R}^d_{\xi},\mathbb{M}_m(\mathbb{C})).
\end{align}
If $P$ is diagonalizable with eigenvalues $\lambda_j \in C^{\infty}(\mathbb{R}_t\times \mathbb{R}^d_x\times \mathbb{R}^d_{\xi})$ such that $$\lambda_1(t,x,\xi)< \dots < \lambda_n(t,x,\xi) \quad {\text{for all }} (t,x,\xi),$$ 
with associated projections $\Pi_j\in C^{\infty}(\mathbb{R}_t\times \mathbb{R}^d_x\times \mathbb{R}^d_{\xi},\mathbb{M}_m(\mathbb{C})))$ on the $j$th eigenspace, then \eqref{eq:PropertyX} implies that 
\begin{align}
\label{eq:conditionXer}
\Pi_j\mathfrak{X}\Pi_j = 0 \qquad \text{in } H^{-1}_{\mathrm{loc}}(\mathbb{R}_t\times \mathbb{R}^d_x \times \mathbb{R}^d_{\xi},\mathbb{M}_m(\mathbb{C})).
\end{align}
As Lemma \ref{lemma:CharacterizationSubspace} shows, this is equivalent to the fact that there exists another matrix-valued distribution $$\mathcal{Y}\in H^{-1}_{\mathrm{loc}}(\mathbb{R}_t\times \mathbb{R}^d_x\times \mathbb{R}^d_{\xi},\mathbb{M}_m(\mathbb{C}))),$$ such that
\begin{equation}
\label{eq:formula Xer Liouville}
    \mathfrak{X} = \tfrac{1}{i}[P,\mathcal{Y}].
\end{equation}
Moreover,
if there are only two distinct eigenvalues $\lambda_+>\lambda_-$ (such as for the Dirac equation), the matrix-valued function $\mathcal{Y}$ satisfies
\begin{equation}
    \mathcal{Y} =\tfrac{1}{\sigma^2}[P,[P,\mathcal{Y}]],
\end{equation}
where $$\sigma :=|\lambda_+-\lambda_-|$$
is the distance of the two eigenvalues.
In that case, it follows from $[P,F]=0$ that
\begin{equation}
\label{eq:FormulaY}
    \mathcal{Y} = \tfrac{i}{\sigma^2}(-[\partial_tP,F]-[P,\tfrac{1}{2}\left( \{P, F \} - \{ F, P \} \right)]).
\end{equation}
Therefore, as long as the derivatives of $P$ are sufficiently regular,
\begin{align*}
F \in L^{\infty}(\mathbb{R}_t,L^2( \mathbb{R}^d_x \times \mathbb{R}^d_{\xi},\mathbb{M}_m(\mathbb{C})) )\quad \text{implies that} \quad \mathcal{Y}  \in L^{\infty}(\mathbb{R}_t,H^{-1}_{\mathrm{loc}}( \mathbb{R}^d_x \times \mathbb{R}^d_{\xi},\mathbb{M}_m(\mathbb{C}))).
\end{align*}
Using the equation, this implies that
\begin{equation*}
    \partial_t F \in L^{\infty}(\mathbb{R}_t,H^{-1}_{\mathrm{loc}}(\mathbb{R}^d_x\times\mathbb{R}^d_{\xi},\mathbb{M}_m(\mathbb{C}))).
\end{equation*}
Therefore
\begin{equation*}
    F \in W^{1,\infty}(\mathbb{R}_t,H^{-1}_{\mathrm{loc}}(\mathbb{R}^d_x\times\mathbb{R}^d_{\xi},\mathbb{M}_m(\mathbb{C}))) \cap L^{\infty}(\mathbb{R}_t,L^2(\mathbb{R}^d_x\times \mathbb{R}^d_{\xi},\mathbb{M}_m(\mathbb{C}))), 
\end{equation*}
and so, since $F$ is continuous in $t$, the prescription of the initial data $F^{\mathrm{in}}$ in \eqref{eq:matrix liouville equation} is well-defined. 

Next we define the dual equation. Let $$\Phi
\in C([0,T],H^k(\mathbb{R}^d_x \times \mathbb{R}^d_{\xi},\mathbb{M}_m(\mathbb{C})))\cap C^1([0,T],H^{k-2}(\mathbb{R}^d_x \times \mathbb{R}^d_{\xi},\mathbb{M}_m(\mathbb{C})))$$ with $k \in \mathbb{N}$ sufficiently large be the unique backwards-in-time solution to
\begin{align}
\begin{cases}
\label{eq:test equation_prelim}
\partial_t \Phi = \mathcal{L} +  \tfrac{1}{2} \left( \{P, \Phi \} - \{ \Phi, P \} \right), \qquad \Phi|_{t=T}  = \Phi^{\mathrm{fin}}, \\
[P,\Phi] = 0,
\end{cases}
\end{align}
provided by Theorem \ref{thm:LWPLiouville},
where 
\begin{equation*}
    \mathcal{L} \in C([0,T],H^{k-1}(\mathbb{R}^d_x\times \mathbb{R}^d_{\xi},\mathbb{M}_m(\mathbb{C}))),
\end{equation*}
is the Lagrange multiplier associated to the constraint $[P,\Phi] = 0$. The reason for the backward evolution in time will become clear below.

We have to impose the analogue of condition \eqref{eq:conditionXer} for the Lagrange multiplier $\mathfrak{X}$ of the Wigner equation, since the dual equation \eqref{eq:test equation_prelim} does not come from a semiclassical limit. Therefore we complement \eqref{eq:test equation_prelim} by the condition
\begin{align}
\Pi_{j}\mathcal{L}\Pi_{j} = 0,
\end{align}
which is equivalent to $\mathcal{L}$ being of the form 
\begin{equation}
\label{eq:formula Xer test}
\mathcal{L} = \tfrac{1}{i}[P,\mathcal{Z}],
\end{equation}
for some $\mathcal{Z} \in C([0,T],H^{k-1}(\mathbb{R}^d_x\times \mathbb{R}^d_{\xi},\mathbb{M}_m(\mathbb{C})))$, cf. Lemma \ref{lemma:CharacterizationSubspace}. Thus, we replace \eqref{eq:test equation_prelim} by
\begin{align}
\begin{cases}
\label{eq:test equation}
\partial_t \Phi =  \tfrac{1}{i}[P,\mathcal{Z}]+ \tfrac{1}{2} \left( \{P, \Phi \} - \{ \Phi, P \} \right), \qquad \Phi|_{t=T}  = \Phi^{\mathrm{fin}}, \\
[P,\Phi] = 0.
\end{cases}
\end{align}
Let now $$P(t,q,p) = \alpha \cdot (p-A(t,q)) + \beta + V(t,q)I_4$$
be the symbol of the Dirac Hamiltonian. We have the following wellposedness theory for the dual Liouville equation:
\begin{theorem}
\label{thm:LWPLiouville}
    Let $P$ be the symbol of the Dirac Hamiltonian and let $A,V \in H^s(\mathbb{R}_t\times \mathbb{R}^d_x)$ for $s$ large enough. Suppose that $\Phi^{\mathrm{in}} \in H^k(\mathbb{R}^d_q\times \mathbb{R}^d_p,\mathbb{M}_m(\mathbb{C}))$ for $k\geq 2$. Then the matrix Liouville equation     \begin{align}
    \begin{cases}
\partial_t \Phi = \tfrac{1}{i}[P,\mathcal{Z}] + \tfrac{1}{2} \left( \{P, \Phi \} - \{ \Phi, P \} \right), \qquad \Phi|_{t=0}  = \Phi_{\mathrm{in}} = \Phi_{\mathrm{in}}^{\ast}, \\
[P,\Phi] = 0,
\end{cases}
\label{eq:MatrixLiouvilleTest}
\end{align}
where
\begin{align}
    P(t,q,p) = \alpha \cdot(p-A(t,q))+ \beta +V(t,q),
\end{align}
has a unique strong solution $\Phi = \Phi^{\ast} \in C([0,T],H^k(\mathbb{R}^d_q\times\mathbb{R}^d_p,\mathbb{M}_4(\mathbb{C})))\cap C^1([0,T],H^{k-2}(\mathbb{R}^d_q\times\mathbb{R}^d_p,\mathbb{M}_4(\mathbb{C})))$ with initial data $\Phi^{\mathrm{in}}$. Moreover, the solution depends continuously on the initial data in the $H^k$-topology.

    In particular we have the following growth bound on the $H^k$-norm of $\Phi$, where $k \in \mathbb{N}$:
    \begin{equation}
        \label{eq:GronwallTestForward}
        \|\Phi(t)\|_{H^k}^2 \leq \|\Phi^{\mathrm{in}}\|_{H^k}^2 e^{C_{k} t}, \quad \text{for } t\in [0,T],
    \end{equation}
    where 
    \begin{equation*}
        C_{k} := (1+F\left(\|\nabla A\|_{{W}^{k,\infty}},\|\partial_t A\|_{{W}^{k-1,\infty}},\|\nabla V\|_{W^{k,\infty}}\right),
    \end{equation*}
    and where $F$ is a continuous function vanishing at the origin.
\end{theorem}
\begin{remark}
    Note that as stated, the Theorem is written for the forward-in-time resolution of the matrix Liouville equation. However, since the matrix Liouville equation is time-reversible, we obtain the reverse estimate
    \begin{equation}
    \label{eq:GronwallTestBackward}
        \|\Phi(t)\|_{H^k}^2 \leq \|\Phi^{\mathrm{fin}}\|_{H^k}^2 e^{C_{k} (T-t)}, \quad \text{for } t\in [0,T],
    \end{equation}
    for the final datum $\Phi^{\mathrm{fin}} \in H^{k}$.
\end{remark}

\begin{remark}
By Lemma \ref{lemma:CharacterizationSubspace} it follows that if $P$ has two distinct eigenvalues $\lambda_+>\lambda_-$ one can write
\begin{equation*}
    \mathcal{Z} =\tfrac{1}{\sigma^2}[P,[P,\mathcal{Z}]], \qquad \sigma := |\lambda_+-\lambda_-|
\end{equation*}
and thus, by using the equation and $[P,\Phi]=0$, one has the explicit expression, analogous to \eqref{eq:FormulaY},
\begin{equation}
        \mathcal{Z} = \tfrac{i}{\sigma^2}\left(-[\partial_tP, \Phi] - [P,\tfrac{1}{2}(\{P,\Phi\}-\{\Phi,P\})]\right).
    \end{equation}
    This expression is then used to derive the a priori estimate for $\Phi$. Note that $\mathcal{Z}$ is of the order of one derivative of $\Phi$ such that one expects an estimate of the form
    \begin{equation*}
        \|\mathcal{Z}\|_{H^k} \leq \text{const.} \|\Phi\|_{H^{k+1}}.
    \end{equation*}
    For details see Section \ref{sec:LWPMatrixLiouville}.
    \end{remark}

\subsection{Pseudo-distance for matrix-valued density operators}
\label{sec:Distance}
Now we consider the quantity
\begin{align*}
\mathfrak{D}(t) \equiv \mathfrak{D}(R(t),F(t)) = \sup_{\mathcal{A}}\left|\Tr(R(t) \mathrm{Op}^T_{\hbar}[\Phi(t,\cdot,\cdot)]) - \iint_{\mathbb{R}^d\times \mathbb{R}^d} \tr(F(t,q,p)\Phi(t,q,p)) dq dp\right|,
\end{align*}
where $\Phi(t)$ is a solution of the dual matrix Liouville equation \eqref{eq:test equation} in the sense of Theorem \ref{thm:LWPLiouville}. We also recall that $\Tr(\,\cdot \,)$ denotes the trace in $L^2(\mathbb{R}^d,\mathbb{M}_m(\mathbb{C}))$ and $\tr(\,\cdot\,)$ denotes the trace in $\mathbb{M}_m(\mathbb{C})$. 

The supremum is taken over the set $\mathcal{A}$ given by
\begin{equation}
\label{eq:AlgebraA}
    \mathcal{A} := \left\{\|\Phi^{\mathrm{fin}}\|_{H^k} \leq 1, \quad \Phi^{\mathrm{fin}} \text{ is final datum for } \eqref{eq:test equation}\right\}.
\end{equation}
The Sobolev norms of the solution $\Phi(t)$ are then propagated in time by the a priori estimate \eqref{eq:AprioriMatrixLiouville} of Proposition \ref{thm:AprioriMatrixLiouville}.
Since (as we shall see below) 
\begin{align*}
    \partial_t \iint \tr(F(t,q,p)\Phi(t,q,p))dqdp = 0, 
\end{align*}
it follows that
\begin{align*}
      \iint \tr(F(t,q,p)\Phi(t,q,p))dqdp = \iint \tr(F^{\mathrm{in}}(q,p)\Phi^{\mathrm{in}}(q,p))dqdp.
\end{align*}
Here, $$\Phi^{\mathrm{in}}(q,p) := \Phi(0,q,p).$$
Thus, $\mathfrak{D}(t)$ can be written as
\begin{align}
\label{eq:MatrixDistance}
\mathfrak{D}(t) = \sup_{\mathcal{A}}\left|\Tr(R(t) \mathrm{Op}^T_{\hbar}[\Phi(t,\cdot,\cdot)]) - \iint_{\mathbb{R}^d\times \mathbb{R}^d} \tr(F^{\mathrm{in}}(q,p)\Phi^{\mathrm{in}}(q,p)) dq dp\right|.
\end{align}
Since the supremum is taken over the test function $\Phi(t)$ at a time $t>0$, we must express the initial datum $\Phi^{\mathrm{in}}=\Phi(0)$ in terms of $\Phi(t)$. This amounts to solving the (dual) evolution equation backward in time with terminal datum $\Phi(t)$ (and, in particular, with $\Phi^{\mathrm{fin}}=\Phi(T)$ when $t=T$).

\begin{proposition}
    The quantity $\mathfrak{D}(t)$, defined in \eqref{eq:MatrixDistance} is a pseudo-distance, i.e. it is symmetric, it satisfies the triangle inequality and $\mathfrak{D}(F(t),F(t)) = 0$.
    \begin{proof}
        First note that there is a bijective correspondence between the density matrices $R$ and the Husimi functions $\tilde{W}_{\hbar}[R]$ and we have the following identity:
        \begin{equation}
    \Tr(R(t)\mathrm{Op}^T_{\hbar}[\Phi(t)]) = \iint \tr(\tilde{W}_{\hbar}(t,q,p)\Phi(t,q,p))dqdp,
\end{equation}
where 
\begin{equation}
    \tilde{W}_{\hbar}(q,p) := \iint G_{\hbar}(q-q')G_{\hbar}(p-p')W_{\hbar}(q',p')dq'dp',
\end{equation}
is the Husimi function of $W_{\hbar}$ with $G_{\hbar}$ given by \eqref{eq:Gaussian}. We can rewrite the pseudo-distance $\mathfrak{D}(t)$ as
\begin{equation}
\label{eq:Husimi distance}
    \mathfrak{D}(\tilde{W}_{\hbar}(t),F(t)) = \sup_{\mathcal{A}}\left|\iint (\tilde{W}_{\hbar}(t,q,p)-F(t,q,p))\Phi(t,q,p)dqdp \right|.
\end{equation}
Now $$\mathfrak{D}(\tilde{W}_{\hbar}(t),F(t)) = \mathfrak{D}(F(t),\tilde{W}_{\hbar}(t)).$$ Moreover, $\tilde{W}_{\hbar}(t)=F(t)$ implies $$\mathfrak{D}(\tilde{W}_{\hbar}(t),F(t)) = 0.$$ Clearly, $\mathfrak{D}(t) \geq 0$. Finally, the triangle inequality follows from
\begin{align*}
    \mathfrak{D}(F(t),G(t)) &= \sup_{\mathcal{A}} \left|\iint (F(t,q,p)-G(t,q,p))\Phi(t,q,p)dqdp \right| \\ &= \sup_{\mathcal{A}} \left|\iint (F(t,q,p)-H(t,q,p) + H(t,q,p)-G(t,q,p))\Phi(t,q,p)dqdp \right| \\
    &= \sup_{\mathcal{A}} \left|\iint (F(t,q,p)-H(t,q,p))\Phi(t,q,p)dqdp \right. \\ &\qquad \qquad + \left. \iint (H(t,q,p)-G(t,q,p))\Phi(t,q,p)dqdp \right| \\
    &\leq \sup_{\mathcal{A}} \left|\iint (F(t,q,p)-H(t,q,p))\Phi(t,q,p)dqdp\right| \\ &\qquad \qquad + \left| \iint (H(t,q,p)-G(t,q,p))\Phi(t,q,p)dqdp \right| \\
    &= \mathfrak{D}(F(t),H(t)) + \mathfrak{D}(H(t),G(t)).
\end{align*}
Thus, $\mathfrak{D}(t)$ is a pseudo-distance.
    \end{proof}
\end{proposition}
Next differentiate the $\mathfrak{D}(t)$ with respect to $t$ and integrate over $t$:
\begin{align*}
    \mathfrak{D}(t)-\mathfrak{D}^{\mathrm{in}} &\leq \int_0^t \sup_{\mathcal{A}} \left|\Tr((\partial_t R(s)) \mathrm{Op}^T_{\hbar}[\Phi(s,\cdot,\cdot)]) + \Tr( R(s) \mathrm{Op}^T_{\hbar}[\partial_t\Phi(s,\cdot,\cdot)]) \right|ds.\\
    &\qquad + \int_0^t \sup_{\mathcal{A}}\left|\partial_t\iint_{\mathbb{R}^{2d}} \tr( F(s,q,p)\Phi(s,q,p))dqdp\right|ds,
\end{align*}
where
\begin{equation}
    \mathfrak{D}^{\mathrm{in}} := \mathfrak{D}(0) = \sup_{\mathcal{A}}\left|\Tr(R^{\mathrm{in}} \mathrm{Op}^T_{\hbar}[\Phi^{\mathrm{in}}]) - \iint_{\mathbb{R}^d\times \mathbb{R}^d} \tr(F^{\mathrm{in}}(q,p)\Phi^{\mathrm{in}}(q,p)) dq dp\right|.
\end{equation}
To proceed we need the following Lemma:
\begin{lemma}
    Let $F\in L^{\infty}([0,T],H^1(\mathbb{R}^d_q \times \mathbb{R}^d_p))$ be a solution in the sense of distributions of the matrix Liouville equation \eqref{eq:matrix liouville equation} with initial datum $0 \leq F^{\mathrm{in}} = (F^{\mathrm{in}})^{\ast}\in H^1$ such that $[P,F^{\mathrm{in}}] = 0$ and let $\Phi \in C([0,T],H^s(\mathbb{R}^d_q \times \mathbb{R}^d_p))$, $s\geq 2$ be a solution of the dual matrix Liouville equation \eqref{eq:test equation} with final datum $\Phi^{\mathrm{fin}}$ with $\Phi^{\mathrm{in}} := \Phi(0)$ and $[P,\Phi^{\mathrm{fin}}]=0$. Then
    \begin{equation}
        \partial_t \iint_{\mathbb{R}^{2d}} \tr(F(t,q,p)\Phi(t,q,p))dqdp = 0.
    \end{equation}
    \begin{proof}
    By direct computation,
        \begin{align*}
    \partial_t \tr(F\Phi) &= \tfrac{1}{2}\tr((\{P,F\}-\{F,P\})\Phi + F(\{P,\Phi\}-\{\Phi,P\})) + \tr(\mathfrak{X}\Phi+F\mathcal{L}) \\
    &= \tfrac{1}{2}\tr(\{P,F\Phi\}-\{F\Phi,P\}) + \tr(\mathfrak{X}\Phi+F\mathcal{L}),
\end{align*}
where we used that $[F,P] = [\Phi,P] = 0$.
Therefore
\begin{align}
\begin{split}
    \partial_t \iint \tr(F\Phi) &= \iint \tr(\tfrac{1}{2}\{P,F\Phi\}-\{F\Phi,P\}) + \iint \tr(\mathfrak{X}\Phi+F\mathcal{L})  \\ &= \iint \tr(\mathfrak{X}\Phi+F\mathcal{L})  \\
    &= 0. 
    \end{split}\label{eq:PartialFPIsZero}
\end{align}
The last step needs justification since we need to show that
\begin{equation}
\label{eq:Pairing_X_Phi}
    \sup_{t\in[0,T]}\iint \tr(\mathfrak{X}(t)\Phi(t))
\end{equation}
is well-defined. Indeed, by the bound \eqref{eq:SpaceTimeGradientWigner} from Propostion \ref{lemma:GronwallLagrangeWigner} we obtain that $W_{\hbar} \in L^{\infty}([0,T],H^1_{q,p})$. By weak convergence we obtain that $F \in L^{\infty}([0,T],H^1_{q,p})$. By \eqref{eq:formula Xer Liouville} and \eqref{eq:FormulaY} we obtain
\begin{align*}
    \|\mathfrak{X}\|_{L^2_{q,p}} &\leq \left\|\tfrac{1}{\sigma^2} [P,[\partial_tP,F]] + \tfrac{1}{\sigma^2}[P,[P,[\tfrac{1}{2}(\{P,F\}-\{F,P\})]]] \right\|_{L^2_{q,p}} \\
    &\leq C \|\tfrac{1}{\sigma}P\|_{L^{\infty}_{q,p}}\|\tfrac{1}{\sigma}\partial_t P\|_{L^{\infty}_{q,p}}\|F\|_{L^2_{q,p}} + \|\tfrac{1}{\sigma}P\|_{L^{\infty}_{q,p}}^2 \|\nabla_{q,p}P\|_{L^{\infty}_{q,p}} \|F\|_{\dot{H}^1_{q,p}} \\ 
    &\leq C\|F\|_{{H}^1_{q,p}}.
\end{align*}
Since the Dirac symbol is linear in $p$ we can bound $\|\nabla_{q,p}P\|_{L^{\infty}_{q,p}}$. Therefore
$$\mathfrak{X} \in L^{\infty}([0,T],L^2_{q,p})$$ and since
$\Phi \in C([0,T],H^{s}_{q,p})$ the pairing \eqref{eq:Pairing_X_Phi} is well-defined.  
Moreover, $F\in L^{\infty}([0,T],L^2)$ and $\mathcal{L} =\tfrac{1}{i}[P,\mathcal{Z}]\in C([0,T],H^{s-1})$ by Lemma \ref{eq:PropertiesEstimatesZdelta}. Therefore,
$$\sup_{t\in[0,T]}\|\tr(F(t)\mathcal{L}(t))\|_{L^1_{q,p}}\leq C.$$
Then, using \eqref{eq:formula Xer Liouville} and \eqref{eq:formula Xer test} and the fact that $[F(t),P(t)] = [\Phi(t),P(t)] = 0$ for all $t\in [0,T]$ we obtain
\begin{align*}
    \iint \tr(\mathfrak{X}(t) \Phi(t) + F (t)\mathcal{L}(t)) dq dp &= \iint \tr(\tfrac{1}{i}[P(t),\mathcal{Y}(t)] \Phi + F(t) \tfrac{1}{i}[P(t),\mathcal{Z}(t)]) dq dp \\
    &= -\iint\tr(\mathcal{Y}(t)\tfrac{1}{i}[P(t),\Phi(t)]  + \tfrac{1}{i}[P(t),F(t)]\mathcal{Z}(t)) dq dp \\
    &= 0.
\end{align*}
This completes the proof.
    \end{proof}
\end{lemma}
Using the Lemma and the dual equation \eqref{eq:test equation}, we obtain
\begin{align}
\begin{split}
\mathfrak{D}(t) \leq \mathfrak{D}^{\mathrm{in}} &+ \sup_{\mathcal{A}} \int_0^t \left|\Tr(R(s)\left(\tfrac{i}{\hbar}[H(s),\mathrm{Op}^T_{\hbar}[\Phi(s)]]+\mathrm{Op}^T_{\hbar}\left[\tfrac{1}{2}\left(\{P(s),\Phi(s)\} - \{\Phi(s),P(s)\}\right)\right]\right))\right| ds \\ &\qquad + \sup_{\mathcal{A}}\int_0^t\left|\Tr(R(s)\mathrm{Op}^T_{\hbar}[\tfrac{1}{i}[P(s),\mathcal{Z}(s)]) \right|ds .
\end{split}
\label{eq:GrönwallDistance}
\end{align}
We will study the right hand side of \eqref{eq:GrönwallDistance} in Section \ref{sec:DistanceDirac}.

\begin{remark}Compare this to the scalar equation. In that case the scalar distance  (which now is a bona fide distance as there are no off-diagonal terms) $D(t)$ satisfies
\begin{align*}
D(t) \leq D^{\mathrm{in}} &+\int_0^t \sup_{a^{\mathrm{fin}} \in \mathcal{A}'}\left| \Tr(R(s)\left(\tfrac{i}{\hbar}\left[-\tfrac{\hbar^2}{2}\Delta +V(s),\mathrm{Op}^T_{\hbar}[a(s)]\right]+\mathrm{Op}^T_{\hbar}\left[\{\tfrac{1}{2}|p|^2+V(q),a(s)\}\right]\right))\right|ds .
\end{align*}
where $\mathcal{A}'$ is the analogous algebra to $\mathcal{A}$ but with the matrix Liouville equation replaced by the scalar Liouville equation. Naturally, since $P$ and $W_{\hbar}$ are scalar, they commute and therefore $\mathcal{L}=\frac{1}{i}[P,\mathcal{Z}]=0$.
\end{remark}

\begin{remark}
Recall the pseudo-distance $\mathfrak{D}$ written in terms of the Husimi function $\tilde{W}_{\hbar}$, i.e. \eqref{eq:Husimi distance}. Naturally, since $\tilde{W}_{\hbar}$ converges to the Wigner measure $F$ (in fact, this is one method to show that the Wigner measure is positive definite), the pseudo-distance between $\tilde{W}_{\hbar}$ and $F$ should be small.
However, dealing with the equation for the Husimi function seems problematic since there are issues with its wellposedness \cite{emamirad_existence_2001} and the propagation of its derivatives. Moreover, the Husimi equation involves a number of convolutions with Gaussians, which might lead to less traceable computations. In the development of quantum optimal transport there have been notions to compare the positive (in the scalar case) or positive definite (in the matrix case) Husimi function with positive (or positive definite) classical measures such as in \eqref{eq:Husimi distance}, albeit for the scalar case. For a discussion of this ansatz see \cite{zyczkowski_monge_1998}. Recently, in \cite{cotler_egorov_2025}, the authors prove an Egorov theorem for the scalar quantum Wasserstein distance.
\end{remark}

\begin{remark}
    In recent years, quantum optimal transport has been developed using operator-valued couplings of density matrices or density matrices and classical phase space distributions, cf. \cite{golse_optimal_2022}. The extension of optimal transport (classical or quantum) to matrix-valued densities is an area of ongoing research: \cite{brenier_optimal_2020,chen_matrix_2018, ning_metrics_2014, ning_matrix-valued_2015}. It is not straightforward to extend the notion of couplings (either classical or quantum) to the matrix-valued case, for instance consider the discussion in Section II.C of \cite{ning_matrix-valued_2015} for classical matrix-valued measures. The same is true for the quantum case, cf. the counterexample in Appendix \ref{sec:Counterexample}.  In \cite{ning_metrics_2014} a dual formulation is proposed for the classical matrix-valued optimal transport and which resembles our approach: Take two matrix-valued measures $\mu_1,\mu_2$.  Let $f$ be a  Hermitian Lipschitz matrix-valued function and let $I$ be an interval. Then the distance proposed in \cite{ning_metrics_2014} is defined as
    \begin{equation}
        \sup_{f}\left\{\int_I \tr(f(d\mu_1-d\mu_2)) \colon \mathrm{Lip}(f) \leq 1, \, \|f\|_{\infty} \leq C  \right\},\label{eq:ClassicalMatrixDistance}
    \end{equation}
    which goes back to the work by Connes \cite{connes_compact_1989}. Our ansatz provides a new approach to the quantum case that is similar to the duality approach for classical matrix-valued densities given by \eqref{eq:ClassicalMatrixDistance} by using the Lagrange multiplier formulation of \cite{golse_semi-classical_2026}.
\end{remark}

\subsection{Solutions of the Dirac equation}
Under suitable assumptions on the time-dependent potentials $A,V$ one can construct a unitary two-parameter group, generated by
\begin{equation}\label{eq: dirac hamiltonian}H(t) = \alpha\cdot(-i\hbar \nabla -A(t,x))+\beta+V(t,x).\end{equation}
We quote the following theorem from \cite{golse_semi-classical_2026}.

\begin{proposition}
\label{thm:existence dirac}
Let $r >3$ and $A,V  \in C^1(\mathbb{R}, L^r+L^{\infty})$. Then for all $t \in \mathbb{R}$ the Dirac operator \eqref{eq: dirac hamiltonian} with domain $\mathcal{D}(H(t)) = H^1(\mathbb{R}^3,\mathbb{C}^4)$ is self-adjoint. Moreover, there exists a  jointly strongly continuous two-parameter family  of unitary propagators $U_{H}(t;s)$ satisfying 
\begin{align}
i \hbar \partial_t U_H(t;s) \psi = H(t) U_H(t;s) \psi  
\quad \text{for all} \;  \psi \in \mathcal{D}(H(t)),
\end{align}
and $U_H(s;s) = 1$ for all $s \in \mathbb{R}$.
\end{proposition}

\begin{corollary}
\label{thm:corollary existence}
Let $A,V$ and $U_{H}$ be as in Proposition~\ref{thm:existence dirac}.  Moreover,  let  $ R^{\mathrm{in}}$ be a density operator such that $$\Tr\left( \sqrt{1 - \hbar^2\Delta} R^{\mathrm{in}} \sqrt{1 - \hbar^2\Delta} \right) < + \infty.$$ Then  \begin{equation}
\label{eq:evolution R}R(t) := U_H(t;0) R^{\mathrm{in}} U_H(0;t),\end{equation} is the unique solution to the von Neumann equation \eqref{eq:vonNeumannIntro} with initial data $R^{\mathrm{in}}$.
\end{corollary}

By Plancherel's theorem we deduce the following Corollary.
\begin{corollary}
\label{thm:existence wigner}
Let $A,V$ be as in Proposition \ref{thm:existence dirac} and $R_{\mathrm{in}}$ as in Corollary \ref{thm:corollary existence}. Then the Wigner transform $W_{\hbar}(t)= W_{\hbar}[R(t)]$, where $R(t)$ is given by \eqref{eq:evolution R}, satisfies the Wigner equation \eqref{eq:WignerIntro} in the sense of distributions, where $W_{\hbar}(0) = W_{\hbar}^{\mathrm{in}} = W_{\hbar}[R^{\mathrm{in}}]$. The evolution of $W_{\hbar}(t)$ conserves mass, i.e.
\begin{equation}
    \|W_{\hbar}(t)\|_{L^2} = \|W_{\hbar}^{\mathrm{in}}\|_{L^2}.
\end{equation}
for all $t$.
\end{corollary}

\subsection{The remainder term of the Wigner equation}
\label{sec:Remainder}
Recall the Wigner equation \eqref{eq:WignerIntro}
\begin{equation}
    \partial_t W_{\hbar} = \tfrac{1}{i\hbar}[P,W_{\hbar}]+ \tfrac{1}{2}(\{P,W_{\hbar}\}-\{W_{\hbar},P\}) + r_{\hbar}[W_{\hbar}] \,,\qquad W_{\hbar}(0)=W_{\hbar}[R^{\mathrm{in}}].
\end{equation}
In order to prove the convergence rate of the commutator later on we need the explicit form of the remainder $r_{\hbar}$. For the derivation we refer to the Appendix in \cite{golse_semi-classical_2026}. The remainder $r_{\hbar}$ is given by
\begin{equation}
    r_{\hbar}[W_{\hbar}] := \hbar \Theta[A_k][\alpha_k,W^{\hbar}] +\frac{1}{2}(-\nabla_x A_k\cdot \nabla_{\xi}-\theta[A_k])[\alpha_k, W^{\hbar}]_+   +(-\nabla V \cdot \nabla_{\xi} -\theta [V])W^{\hbar}.
    \label{eq:Remainder}
\end{equation}
The pseudo-differential operators $\Theta$ and $\theta$ are given by (cf. Section \ref{sec:Notation}, equations \eqref{eq:FirstOrderPseudoDiff} and \eqref{eq:SecondOrderFinDiff})
\begin{align*}
    (\theta[F]f)(x,\xi) &= \frac{1}{(2\pi)^d}\int e^{i(\xi-\eta)\cdot y} \delta[F]_j(x,y)\partial_{\eta_j}f(x,\eta)d y d\eta,
\end{align*}
where
\begin{equation*}
    \delta[F]_j(x,y) = \frac{F\left(x+\tfrac{\hbar y}{2}\right)-F\left(x-\tfrac{\hbar y}{2}\right)}{\hbar|y|}\frac{y_j}{|y|},
\end{equation*}
and
\begin{align*}
    (\Theta[F]f)(x,\xi) &= \frac{1}{(2\pi)^d}\int e^{i(\xi-\eta)\cdot y} \Delta[F](x,y)\Delta_{\eta}f(x,\eta)d y d\eta,
\end{align*}
where
\begin{equation*}
     \Delta[F](x,y) = i\frac{F\left(x+\frac{\hbar y}{2}\right)-2F(x)+F\left(x-\frac{\hbar y}{2}\right)}{2\hbar^2|y|^2}. 
\end{equation*}


\section{Notations and definitions}\label{sec:Notation}

\subsection{Density matrices and Wigner transforms}
Let 
\begin{equation}
    \mathfrak{H}^{d,m} := L^2(\mathbb{R}^d, \mathbb{C}^m).
\end{equation}
We denote the Hilbert space of Dirac 4-spinors on $\mathbb{R}^3$ by
\begin{equation}
    \mathfrak{H} \equiv \mathfrak{H}^{3,4} = L^2(\mathbb{R}^3, \mathbb{C}^4).
\end{equation}
Let $\mathbb{M}_m(\mathbb{C})$ be the space of $m\times m$ matrices with entries in $\mathbb{C}$ and let
\begin{equation}
    \langle v , w \rangle_{\mathbb{C}^m} := v^{\ast} w = \sum_{j=1}^m \overline{v_j}w_j, \quad v,w \in \mathbb{C}^m
\end{equation}
be the inner product on $\mathbb{C}^m$, and $v^* := \overline{v}^\top$ be the conjugate transpose of $v$.
The set of bounded operators on $\mathfrak{H}^{d,m}$ is denoted by $\mathcal{L}(\mathfrak{H}^{d,m})$. The set of trace-class operators on $\mathfrak{H}$ is denoted by $\mathcal{L}^1( \mathfrak{H})$. 

In particular we denote the trace in $\mathcal{L}(\mathfrak{H}^{d,m})$ by $\Tr(\cdot) \equiv \Tr_{\mathfrak{H}^{d,m}}(\cdot)$ and the trace in $\mathbb{M}_m(\mathbb{C})$ by $\tr(\cdot) \equiv \tr_{\mathbb{C}^m}(\cdot)$, In particular, for a trace-class operator $A \in \mathcal{L}^1(\mathfrak{H}^{d,m})$ with kernel $K(x,y)\in L^2(\mathbb{R}^d_x\times\mathbb{R}^d_y,\mathbb{M}_m(\mathbb{C}))$ we have
\begin{equation}
    \Tr(A) = \int_{\mathbb{R}^d_x} \tr_{\mathbb{C}^m}(K(x,x)) dx = \int_{\mathbb{R}^d_x} \sum_{j=1}^m K_{jj}(x,x) dx.
\end{equation}

A \emph{density operator} $R\in \mathcal{L}^1( \mathfrak{H}^{d,m})$ is a self-adjoint ($R=R^*$), nonnegative ($\langle R\psi,\psi\rangle \geq 0$ for all $\psi \in \mathfrak{H}^{d,m}$) operator 
with trace $$\Tr(R) = \int_{\mathbb{R}^d_x} \tr_{\mathbb{C}^m}(R(x,x)) dx = \int_{\mathbb{R}^d_x} \sum_{j=1}^m R_{jj}(x,x) dx = 1.$$ A rank-one density operator is a \emph{pure state}.  On $\mathfrak{H}^{d,m}$, the operator $R$ is an integral operator 
with integral kernel $(x,y)\mapsto R(x,y) \in L^2(\mathbb{R}_x^d\times \mathbb{R}_y^d,\mathbb{M}_m(\mathbb{C}))$, which we abusively denote by the same letter. Since $R=R^*$ is trace-class, it is compact and by the spectral theorem 
there is a sequence of real eigenvalues $(\lambda_j)_{j\ge 1}$ and a Hilbert basis (i.e. a complete orthonormal system) $\{\psi_j\,:\,j\ge 1\} \subset \mathfrak{H}^{d,m}$ such that
\begin{equation}\label{SpecDecR}
    R(x,y) = \sum_{j=1}^{\infty} \lambda_j \psi_j(x) \overline{\psi_j(y)}^\top.
\end{equation}
Since $R$ is a density operator,
\begin{equation}\label{CondLambda}
\lambda_j \geq 0\quad\text{since }R=R^*\ge 0,\quad\text{ and }\sum_{j=1}^{\infty}\lambda_j 
=\Tr(R)=1,
\end{equation}
since $\{\psi_j\}$ is an orthonormal family. In particular, the restriction of the integral kernel to the diagonal, 
i.e. $R(x,x)$, is a well-defined element of $L^1(\mathbb{R}^d_x, \mathbb{M}_m(\mathbb{C}))$ and we can define the density function $\rho := \tr(R(x,x))\in L^1(\mathbb{R}_x^d,\mathbb{R})$.

Let $R$ be a density operator. Then the Wigner transform $W_{\hbar}[R]$ is defined as
\begin{equation}
\label{eq:WignerTransform}
   W_{\hbar}[R](x,\xi) := \frac{1}{(2\pi)^d}\int e^{-i\xi\cdot y}R\left(x+\tfrac{\hbar y}{2}, x-\tfrac{\hbar y}{2}\right) dy.
\end{equation}

\subsection{Brackets}
For two matrices $A,B$, the commutator is defined as $[A,B] := AB-BA$ and the anticommutator as $[A,B]_+ := AB+BA$.  For two (not necessarily commuting) matrix-valued functions $F,G$ on phase space $\mathbb{R}^d_x \times \mathbb{R}^d_{\xi}$ we define the Poisson bracket as
\begin{equation}
    \{F,G\} := \sum_{k=1}^d \partial_{x_k} F \partial_{\xi_k} G-\partial_{\xi_k} F \partial_{x_k} G,
\end{equation}
where the derivatives are taken in the sense of distributions. Note that for matrix-valued functions the order is important and in general, $\{F,G\}\neq-\{G,F\}$. Depending on the context, $\{\,\cdot\,,\cdot\,\}$ may also denote the standard Poisson bracket for scalar functions, in which case $\{f,g\} = -\{g,f\}$.
On the space $\mathcal{S}'(\mathbb{R}_x^d\times \mathbb{R}^d_{\xi},\mathbb{M}_m(\mathbb{C}))$ we define the duality bracket
\begin{equation}
    \langle A,B\rangle := \int \tr(A^{\ast} B) dx d \xi, \quad A \in \mathcal{S}'(\mathbb{R}_x^d\times \mathbb{R}^d_{\xi},\mathbb{M}_m(\mathbb{C})), B\in \mathcal{S}(\mathbb{R}_x^d\times \mathbb{R}^d_{\xi},\mathbb{M}_m(\mathbb{C})).
\end{equation}

\subsection{Pseudo-differential operators and Toeplitz operators}
We use the following notation for the first order finite difference
\begin{equation}
    \label{eq:FirstOrderFinDiff}\delta[F]_j(x,y):= \frac{F\left(x+\tfrac{\hbar y}{2}\right)-F\left(x-\tfrac{\hbar y}{2}\right)}{\hbar|y|}\frac{y_j}{|y|}
\end{equation}
We denote the associated pseudo-differential operator by (with summation over repeated indices)
\begin{align}
\label{eq:FirstOrderPseudoDiff}
    (\theta[F]f)(x,\xi) &:= \frac{1}{(2\pi)^d}\int e^{i(\xi-\eta)\cdot y} \delta[F]_j(x,y)\partial_{\eta_j}f(x,\eta)d y d\eta \\
    &= -\frac{1}{(2\pi)^d}\int e^{i(\xi-\eta)\cdot y} \frac{F\left(x+\tfrac{\hbar y}{2}\right)-F\left(x-\tfrac{\hbar y}{2}\right)}{i\hbar}f(x,\eta)d y d\eta. \nonumber
\end{align}
The second line is the definition used in \cite{lions_sur_1993, markowich_classical_1993, moller_pauli-poisson_2025, golse_semi-classical_2026}.

For the second-order centered finite difference we write
 \begin{equation}
 \label{eq:SecondOrderFinDiff}
     \Delta[F](x,y) :=i \frac{F\left(x+\frac{\hbar y}{2}\right)-2F(x)+F\left(x-\frac{\hbar y}{2}\right)}{2\hbar^2|y|^2} .
 \end{equation}
 We denote the associated pseudo-differential operator by
 \begin{align}
 \label{eq:SecondOrderPseudoDiff}
    (\Theta[F]f)(x,\xi) &:= \frac{1}{(2\pi)^d}\int e^{i(\xi-\eta)\cdot y} \Delta[F](x,y)\Delta_{\eta}f(x,\eta)d y d\eta \\
    &= -\frac{1}{(2\pi)^d}\int e^{i(\xi-\eta)\cdot y} \Delta[F](x,y)|y|^2 f(x,\eta)d y d\eta. \nonumber
\end{align}
 Coherent states are denotes by $\ket{q,p}$. The wavefunction of $\ket{q,p}$ is given by
 \begin{equation}
     \psi_{q,p}(x) := \frac{1}{(\pi\hbar)^{d/4}}e^{-{\tfrac{|x-q|^2}{2\hbar}}}e^{\tfrac{i}{\hbar}p\cdot x}.
 \end{equation}
 The Toeplitz operator with symbol $\Phi$ is defined as
 \begin{align}
     \mathrm{Op}^T_{\hbar}[\Phi(t,\cdot,\cdot)] := \iint \frac{dqdp}{(2\pi\hbar)^d} \ket{q,p}\bra{q,p}\Phi(t,q,p).
 \end{align}
We denote the kernel of the operator $\ket{q,p}\bra{q,p}$ by $K_{q,p}$. That is,
\begin{equation}
    K_{q,p}(X,Y) := \psi_{q,p}(X)\psi_{q,p}(Y)^* = \tfrac{1}{(\pi\hbar)^{d/2}}e^{-\tfrac{1}{2{\hbar}}(|X-q|^2+|Y-q|^2)}e^{\tfrac{i}{\hbar}(X-Y)\cdot p}.
\end{equation}

\subsection{Miscellaneous}
The Gaussian is defined as
\begin{equation}
\label{eq:Gaussian}
 G_{\hbar}(z) :=  \tfrac{1}{(\pi\hbar)^{d/2}} e^{-\frac{|z|^2}{\hbar}}.
\end{equation}
The space of tempered distributions is denoted by $\mathcal{S}'$ where $\mathcal{S}$ is the Schwartz space.

We use the following notation for the so-called ``Japanese bracket":
\begin{equation}
    \langle x \rangle := \sqrt{1 +|x|^2}.
\end{equation} 

The letter $C$ denotes a general positive constant that may vary from one line to another. 

\section{Pseudo-distance for the Dirac equation}
\label{sec:DistanceDirac}
From now on let
\begin{align*}
H(t) = \alpha \cdot (-i\hbar \nabla - A(t)) + \beta + V(t), && P(t,q,p) = \alpha \cdot (p-A(t,q)) + \beta +V(t,q)I_4.
\end{align*}
Following \eqref{eq:GrönwallDistance}, we have to control
\begin{align}
&\Tr(R(t)\tfrac{i}{\hbar}[H(t),\mathrm{Op}^T_{\hbar}[\Phi(t)]]) + \Tr(R(t)\mathrm{Op}^T_{\hbar}\left[\tfrac{1}{2}(\{P(t),\Phi(t)\}-\{\Phi(t), P(t)\})+\tfrac{1}{i}[P(t),\mathcal{Z}(t)]\right]).
\end{align}
In this section we rewrite this expression in order to be able to prove the Gronwall estimate. We split the Dirac Hamiltonian into three parts:
\begin{align*}
H(t) 
&= \underbrace{\alpha \cdot(-i\hbar \nabla)+\beta}_{\text{free kinetic part}} \underbrace{- \alpha\cdot A(t)}_{\text{magnetic part}} + \underbrace{V(t)}_{\text{electric part}} \\
&=: H_{\mathrm{kin}} + H_{\mathrm{mag}}(t) + H_{\mathrm{el}}(t).
\end{align*}
Similarly we decompose the Dirac symbol $P$ as follows
\begin{align*}
P(t,q,p) 
&= \underbrace{\alpha \cdot p+\beta}_{\text{free kinetic part}} \underbrace{- \alpha\cdot A(t,q)}_{\text{magnetic part}} + \underbrace{V(t,q)I_4}_{\text{electric part}} \\
&=: P_{\mathrm{kin}}(p) + P_{\mathrm{mag}}(t,q) + P_{\mathrm{el}}(t,q).
\end{align*}
The treatment of the free kinetic part and the electric part are rather straightforward. For the magnetic part one has to take into account the commutator of the Toeplitz quantization with  the multiplication operator $A$ and the matrix $\alpha$ \emph{at the same time}, i.e. if $\Phi(q,p)\ket{q,p}\bra{q,p}$ is the Toeplitz quantization with matrix-valued symbol $\Phi$ one has to treat
  \begin{align*}
\tfrac{i}{\hbar}[H_{\mathrm{mag}},\Phi(q,p)\ket{q,p}\bra{q,p}] = \tfrac{i}{\hbar}[-\alpha \cdot A,\Phi(q,p)\ket{q,p}\bra{q,p}],
\end{align*}
where the commutator is both a commutator of the multiplication operator $A$ with the Toeplitz quantization and a commutator of the matrix $\alpha$ with the symbol matrix $\Phi$.

Finally we have to treat the Lagrange multiplier term
\begin{align*}
    \Tr(R(t)\mathrm{Op}^T_{\hbar}[\mathcal{L}(t)]) = \Tr(R(t)\mathrm{Op}^T_{\hbar}[\tfrac{1}{i}[P(t),\mathcal{Z}(t)]]).
\end{align*}

In the following it is always understood that $\Phi(t) = \Phi(t,q,p)$.

\subsection{Free kinetic part}
We consider first the free kinetic part, i.e. $$H_{\mathrm{kin}} =\alpha \cdot (-i\hbar \nabla) + \beta, \quad P_{\mathrm{kin}}(p) = \alpha \cdot p + \beta.$$ 
\begin{lemma}
\label{lemma:KineticPart}We have
\begin{align*}
&\Tr(R(t)\tfrac{i}{\hbar}[H_{\mathrm{kin}},\mathrm{Op}^T_{\hbar}[\Phi(t)]]) + \Tr(R(t)\mathrm{Op}^T_{\hbar}[\tfrac{1}{2}(\{P_{\mathrm{kin}},\Phi(t)\}-\{\Phi(t), P_{\mathrm{kin}}\})]) \\ & \quad  =\Tr(R(t)\mathrm{Op}^T_{\hbar}\left[\tfrac{i}{\hbar}[P_{\mathrm{kin}},\Phi(t)]\right] ) + \Tr(R(t)\mathrm{Op}^T_{\hbar}\left[\tfrac{1}{2i}\partial_{p_k}[\alpha_k,\Phi(t)]\right]).
\end{align*}
\begin{proof}
First, notice that
\begin{align*}
\tfrac{i}{\hbar}[\alpha \cdot (-i\hbar \nabla)+\beta,\mathrm{Op}^T_{\hbar}[\Phi]] =
[\alpha_k\partial_k,\mathrm{Op}^T_{\hbar}[\Phi]]+ \tfrac{i}{\hbar}[\beta,\mathrm{Op}^T_{\hbar}[\Phi]].
\end{align*}
Since 
\begin{align*}
\mathrm{Op}^T_{\hbar}[\Phi(t,\cdot,\cdot)] = \iint \frac{dqdp}{(2\pi\hbar)^d} \ket{q,p}\bra{q,p}\Phi(t,q,p),
\end{align*}
we study the action of $[\alpha_k\partial_k,\cdot]$ on $\Phi(q,p)\ket{q,p}\bra{q,p}$. We can decompose the commutator as follows:
\begin{align*}[\alpha_k\partial_k,\Phi(q,p)\ket{q,p}\bra{q,p}] = \tfrac{1}{2}[\alpha_k,\Phi(q,p)]_+ [\partial_k,\ket{q,p}\bra{q,p}] + \tfrac{1}{2}[\alpha_k,\Phi(q,p)][\partial_k,\ket{q,p}\bra{q,p}]_+,
\end{align*}
where $[A,B]_+ = AB+BA$ denotes the anticommutator. Denote by $K_{q,p}(X,Y)$ the kernel of $\ket{q,p}\bra{q,p}$. We have
\begin{align*}
[\partial_k,\ket{q,p}\bra{q,p}](X,Y) &= (\partial_{X_k}+\partial_{Y_k})K_{q,p}(X,Y) \\
&=-\left(\frac{X_k-q_k}{\hbar} + \frac{Y_k-q_k}{\hbar}\right) K_{q,p}(X,Y) \\
&=-\partial_{q_k}K_{q,p}(X,Y).
\end{align*}
Moreover,
\begin{align*}
[\partial_k,\ket{q,p}\bra{q,p}]_+(X,Y) &= (\partial_{X_k}-\partial_{Y_k})K_{q,p}(X,Y) \\
&=\tfrac{2i}{\hbar}p_kK_{q,p}(X,Y)-\left(\frac{X_k-q_k}{\hbar} - \frac{Y_k-q_k}{\hbar}\right) K_{q,p}(X,Y) \\
&=\tfrac{2i}{\hbar}p_k K_{q,p}(X,Y)-\left(\frac{X_k-Y_k}{\hbar} \right) K_{q,p}(X,Y) \\
&=\tfrac{2i}{\hbar}p_kK_{q,p}(X,Y)-\tfrac{1}{i}\partial_{p_k} K_{q,p}(X,Y).
\end{align*}
We conclude that
\begin{align*}
\tfrac{i}{\hbar}[\alpha_k(-i\hbar\partial_k) + \beta,\Phi(q,p)\ket{q,p}\bra{q,p}] &= -\tfrac{1}{2}[\alpha_k,\Phi(q,p)]_+ \partial_{q_k} \ket{q,p}\bra{q,p} \\ &\qquad  + \tfrac{1}{2}[\alpha_k,\Phi(q,p)]\left(\tfrac{2i}{\hbar}p_k-\tfrac{1}{i}\partial_{p_k}\right)\ket{q,p}\bra{q,p} \\ &\qquad + \tfrac{i}{\hbar}[\beta,\Phi(q,p)]\ket{q,p}\bra{q,p} \\
&= -\tfrac{1}{2}[\alpha_k,\Phi(q,p)]_+ \partial_{q_k} \ket{q,p}\bra{q,p} \\ &\qquad  + \tfrac{i}{\hbar}[\alpha_k p_k + \beta,\Phi(q,p)]-\tfrac{1}{2i}[\alpha_k,\Phi(q,p)]\partial_{p_k}\ket{q,p}\bra{q,p}.
\end{align*}
Consequently,
\begin{align}
\tfrac{i}{\hbar}[\alpha \cdot (-i\hbar \nabla)+\beta,\mathrm{Op}^T_{\hbar}[\Phi(t)]] &= \mathrm{Op}^T_{\hbar}\left[\tfrac{1}{2}\partial_{q_k}\left(\alpha_k \Phi(t) + \Phi(t)\alpha_k\right)\right] \nonumber\\
&\qquad +  \mathrm{Op}^T_{\hbar}\left[\tfrac{i}{\hbar}[\alpha\cdot p + \beta,\Phi(t)]\right] \nonumber \\
&\qquad + \mathrm{Op}^T_{\hbar}\left[\tfrac{1}{2i}\partial_{p_k}[\alpha_k,\Phi(t)]\right]. \label{eq:free kinetic part quantum}
\end{align}
For the Poisson bracket we have
\begin{align}
\tfrac{1}{2}\left(\{\alpha \cdot p + \beta,\Phi(q,p)\}-\{\Phi(q,p),\alpha \cdot p + \beta\}\right) = -\tfrac{1}{2} \partial_{q_k}(\alpha_k \Phi(q,p) + \Phi(q,p) \alpha_k).
\label{eq:free kinetic part classical}
\end{align}
We note that the Poisson brackets \eqref{eq:free kinetic part classical} cancel with the first term on the right hand side of \eqref{eq:free kinetic part quantum}. Therefore,
\begin{align*}
\tfrac{i}{\hbar}[\alpha \cdot (-i\hbar \nabla)+\beta,\mathrm{Op}^T_{\hbar}[\Phi(t,\cdot,\cdot)]] + \mathrm{Op}^T_{\hbar}\left[\tfrac{1}{2}\left(\{\alpha \cdot p + \beta,\Phi(q,p)\}-\{\Phi(q,p),\alpha \cdot p + \beta\}\right)\right] \\
= \mathrm{Op}^T_{\hbar}\left[\tfrac{i}{\hbar}[\alpha\cdot p + \beta,\Phi(q,p)]\right] + \mathrm{Op}^T_{\hbar}\left[\tfrac{1}{2i}\partial_{p_k}[\alpha_k,\Phi(q,p)]\right].
\end{align*}
This proves the claim.
\end{proof}
\end{lemma}

\subsection{Magnetic part}\label{sec:MagneticPart}
Now we deal with the magnetic part, i.e. $$H_{\mathrm{mag}}(t) =-\alpha\cdot A(t), \quad P_{\mathrm{mag}}(t,q) =-\alpha\cdot A(t,q).$$

\begin{lemma}
\label{lemma:MagneticPart}
We have
\begin{align*}
&\Tr(R(t)\tfrac{i}{\hbar}\left[H_{\mathrm{mag}}(t),\mathrm{Op}^T_{\hbar}[\Phi(t)]\right]) -\Tr(R(t)\mathrm{Op}^T_{\hbar}\left[\tfrac{1}{2}(\{P_{\mathrm{mag}}(t),\Phi(t)\}-\{\Phi(t), P_{\mathrm{mag}}(t)\})\right]) \\ &\qquad=-\Tr(R(t)\delta[A_k]_j\mathrm{Op}^T_{\hbar}\left[\tfrac{1}{2}[-\alpha_k,\partial_{p_j}\Phi]_+\right]-\mathrm{Op}^T_{\hbar}\left[ \partial_{q_j}A_k(q)\tfrac{1}{2}[-\alpha_k,\partial_{p_j}\Phi]_+\right]) \nonumber\\
&\qquad \qquad \qquad + \hbar\Tr({R}(t)\Delta[A_k] \mathrm{Op}^T_{\hbar}\left[[\alpha_k,\Delta_p\Phi]\right])
 \nonumber\\ &\qquad \qquad \qquad +  \Tr({R}(t) \mathrm{Op}^T_{\hbar}\left[\tfrac{1}{2i} \partial_{q_j} \left(\partial_{q_j}A_k\left[-\alpha_k,\Phi\right]\right)\right])-\Tr({R}(t) \mathrm{Op}^T_{\hbar}\left[\tfrac{1}{2i} R^{jj}_k(x,q)[-\alpha_k,\Phi]\right]) \\
 &\qquad \qquad \qquad +  \tfrac{i\hbar}{4} \Tr({R}(t) \mathrm{Op}^T_{\hbar}\left[ \partial_{q_i}\partial_{q_j}\left(R^{ij}_k(x,q)\left[-\alpha_k,\Phi\right]\right)\right]) \\
&\qquad \qquad \qquad  + \Tr({R}(t) \mathrm{Op}^T_{\hbar}\left[ \tfrac{i}{\hbar} [P_{\mathrm{mag}},\Phi]\right]).\nonumber
\end{align*}
Here, $R^{ij}_k(x,q)$ denotes the remainder of the Taylor expansion of $A_k$, i.e.
\begin{align*}
    R^{ij}_k(x,q) = \int_0^1 (1-\theta)(\partial_{i}\partial_{j}A_k)(q+\theta(x-q))d\theta,
\end{align*}
where Einstein summation convention is understood.
\begin{remark}
    Since this expression seems quite complicated, an explanation of the terms on the right hand side seems appropriate. The commutator of the magnetic part $H_{\mathrm{mag}}$ with the Toeplitz operator with symbol $\Phi$ can be decomposed as follows:
    \begin{align}
    \label{eq:MagneticPart}
    \begin{split}
\tfrac{i}{\hbar}[H_{\mathrm{mag}},\Phi(q,p)\ket{q,p}\bra{q,p}] &= \tfrac{1}{2}\tfrac{i}{\hbar}[-\alpha_k,\Phi(q,p)]_+ [A_k,\ket{q,p}\bra{q,p}] \nonumber\\ &\qquad + \tfrac{1}{2}\tfrac{i}{\hbar}[-\alpha_k,\Phi(q,p)][A_k,\ket{q,p}\bra{q,p}]_+. 
\end{split}
\end{align}
    
    The first term on the right hand side of \eqref{eq:MagneticPart} corresponds to the first term in the decomposition above (called $(C1)$ in the proof below). 
    This term will vanish in the limit with a rate of convergence in ${\hbar}$.

    The remaining terms correspond the second term in the decomposition (called $(C2)$ in the proof below). 
    The anticommutator  $[A_k,\ket{q,p}\bra{q,p}]_+$ is then split into two terms:
    
    The first is a second order finite difference, which corresponds to the second term on the right hand side of \eqref{eq:MagneticPart} (the one involving $\Delta[A_k]$). This term scales with $\hbar$.
    
    The second term is a multiplication with $A_k(x)$. This has to be yet again split into the sum of $A_k(x)-A_k(q)$ and $A_k(q)$, corresponding to the last tree terms on the right hand side of \eqref{eq:MagneticPart}. The first two (i.e. the second and third to last lines) correspond to a Taylor expansion of the difference $A_k(x)-A_k(q)$. One might try to estimate this difference by $\mathrm{Lip}(A)|x-q|$, however this yields an $O(1)$ expression which cannot be written as a commutator with $P$ and hence cannot enter the Lagrange multiplier. We thus have to resort to using Taylors theorem to express the $O(1)$ bit as a commutator with $P$ (using Lemma \ref{lemma:HusimiExtraTerm}) and estimate the remainder $R^{ij}$ of the Taylor expansion by some power of $\hbar$.

    The last line corresponds to the multiplication with $A_k(q)$ and thus enters the Lagrange multiplier directly.
\end{remark}
\begin{proof}
In the following we denote by $K_{q,p}[\Phi]$ the kernel of $\Phi(q,p)\ket{q,p}\bra{q,p}$ and by $\tilde{K}_{q,p}(x,y) = K_{q,p}(X,Y)$ we denote the kernel in Weyl variables.

Recall $H_{\mathrm{mag}} := -\alpha \cdot A = -\alpha_kA_k$. Similarly to the free kinetic part, we decompose the commutator as follows:
\begin{align}
\tfrac{i}{\hbar}[H_{\mathrm{mag}},\Phi(q,p)\ket{q,p}\bra{q,p}] &= \tfrac{1}{2}\tfrac{i}{\hbar}[-\alpha_k,\Phi(q,p)]_+ [A_k,\ket{q,p}\bra{q,p}] \nonumber\\ &\qquad + \tfrac{1}{2}\tfrac{i}{\hbar}[-\alpha_k,\Phi(q,p)][A_k,\ket{q,p}\bra{q,p}]_+. 
\end{align}
Therefore
\begin{align}
    &\Tr(R(t)\tfrac{i}{\hbar}[H_{\mathrm{mag}},\mathrm{Op}^T_{\hbar}[\Phi(t)]]) \nonumber \\ &\qquad = \Tr(R(t) \iint \tfrac{1}{2}\tfrac{i}{\hbar}[-\alpha_k,\Phi(q,p)]_+ [A_k,\ket{q,p}\bra{q,p}]) \tfrac{dqdp}{(2\pi\hbar)^d}\nonumber \\
    &\qquad \qquad + \Tr(R(t) \iint  \tfrac{1}{2}\tfrac{i}{\hbar}[-\alpha_k,\Phi(q,p)][A_k,\ket{q,p}\bra{q,p}]_+)\tfrac{dqdp}{(2\pi\hbar)^d} \nonumber \\ &\qquad =: (C1) + (C2). \label{eq:split commutator magnetic}
\end{align}

\underline{The term $(C1)$:} The commutator of the multiplication operator $A_k$ with the Toeplitz kernel $\ket{q,p}\bra{q,p}$ is given by
\begin{align*}
 \tfrac{i}{\hbar}[A_k,\ket{q,p}\bra{q,p}](X,Y) = \tfrac{i}{\hbar}(A_k(X)-A_k(Y))K_{q,p}(X,Y).
\end{align*}
Then we write $(C1)$ in Weyl variables as follows.
\begin{align*}
&\Tr(R(t) \iint \tfrac{1}{2}\tfrac{i}{\hbar}[-\alpha_k,\Phi(q,p)]_+ [A_k,\ket{q,p}\bra{q,p}]\tfrac{dqdp}{(2\pi\hbar)^d})  \\
& \qquad \qquad = \int  {R}(t,X,Y) \tfrac{i}{\hbar}(A_k(X)-A_k(Y))K_{q,p}(X,Y)\tfrac{1}{2}[-\alpha_k,\Phi(q,p)]_+  dX dY \tfrac{dqdp}{(2\pi\hbar)^d}\\ & \qquad \qquad =\int   \tilde{R}(t,x,y) \tfrac{i}{\hbar} \left[{A_k\left(x+\tfrac{\hbar y}{2}\right)-A_k\left(x-\tfrac{\hbar y}{2}\right)} \right]\tfrac{1}{2}[-\alpha_k,\tilde{K}_{q,p}[\Phi](x,y)]_+ \hbar^d dx dy \tfrac{dqdp}{(2\pi\hbar)^d} \\
& \qquad \qquad = \int   \tilde{R}(t,x,y) \frac{A_k\left(x+\tfrac{\hbar y}{2}\right)-A_k\left(x-\tfrac{\hbar y}{2}\right)}{\hbar|y|} \frac{y_j}{|y|}(iy_j)\tfrac{1}{2}[-\alpha_k,\tilde{K}_{q,p}[\Phi](x,y)]_+ \hbar^d dx dy \tfrac{dqdp}{(2\pi\hbar)^d} \\
& \qquad \qquad = -\int \tilde{R}(t,x,y) \frac{A_k\left(x+\tfrac{\hbar y}{2}\right)-A_k\left(x-\tfrac{\hbar y}{2}\right)}{\hbar|y|} \frac{y_j}{|y|} \tilde{K}_{q,p}[\tfrac{1}{2}[-\alpha_k,\partial_{p_j}\Phi]_+](x,y) \hbar^d dx dy \tfrac{dqdp}{(2\pi\hbar)^d}
\\
& \qquad \qquad = -\int \tilde{R}(t,x,y) \delta[A_k]_j(x,y)  \tilde{K}_{q,p}[\tfrac{1}{2}[-\alpha_k,\partial_{p_j}\Phi]_+](x,y) \hbar^d dx dy \tfrac{dqdp}{(2\pi\hbar)^d}\\
& \qquad \qquad =-\Tr(R(t)\delta[A_k]_j\mathrm{Op}^T_{\hbar}\left[\tfrac{1}{2}[-\alpha_k,\partial_{p_j}\Phi]_+\right]).
\end{align*}
Here we used $iy_j \tilde{K}_{q,p} = \partial_{p_j} \tilde{K}_{q,p}$ and integration by parts, together with the notation \eqref{eq:FirstOrderFinDiff}, i.e. $\delta[\,\cdot\,]_j$, for the first order finite difference. \\

\underline{The term $(C2)$:} The anticommutator of the mulitplication operator $A_k$ with the Toeplitz kernel $\ket{q,p}\bra{q,p}$ is given by 
\begin{align*}
 \tfrac{1}{2}\tfrac{i}{\hbar}[A_k,\ket{q,p}\bra{q,p}]_+(X,Y) &=  \tfrac{1}{2}\tfrac{i}{\hbar}\left(A_k(X)+A_k(Y)\right) K_{q,p}(X,Y)
 \\  &=\tfrac{1}{2}\tfrac{i}{\hbar}\left(A_k(X)-2A_k\left(\frac{X+Y}{2}\right)+A_k(Y)\right)K_{q,p}(X,Y) \\ &\qquad + \tfrac{i}{\hbar}A_k\left(\frac{X+Y}{2}\right) K_{q,p}(X,Y).
\end{align*}
That is, we write $[A_k,\ket{q,p}\bra{q,p}]_+$ as a second order difference plus an oscillating term. The oscillating term will enter the Lagrange multiplier while the second order difference will converge to zero.

As with $(C1)$, we write $(C2)$ in Weyl variables as follows
\begin{align*}
&\Tr(R(t) \iint  \tfrac{1}{2}\tfrac{i}{\hbar}[-\alpha_k,\Phi(q,p)][A_k,\ket{q,p}\bra{q,p}]_+\tfrac{dqdp}{(2\pi\hbar)^d}) = \\
&\qquad \qquad = \int_{\mathbb{R}^{4d}}   {R}(t,X,Y) \tfrac{1}{2}\tfrac{i}{\hbar}{\left(A_k\left(X\right)-2A_k\left(\tfrac{X+Y}{2}\right)+A_k\left(Y\right)\right)}[-\alpha_k,{K}_{q,p}[\Phi](X,Y)]dXdY\tfrac{dqdp}{(2\pi\hbar)^d}
 \\ &\qquad \qquad \qquad +  \int_{\mathbb{R}^{4d}} {R}(t,X,Y) \tfrac{i}{\hbar}A_k\left(\tfrac{X+Y}{2}\right)[-\alpha_k,{K}_{q,p}[\Phi](X,Y)] dX dY \tfrac{dqdp}{(2\pi\hbar)^d}\\
&\qquad \qquad =\hbar\int_{\mathbb{R}^{4d}}   \tilde{R}(t,x,y) i\frac{A_k\left(x+\frac{\hbar y}{2}\right)-2A_k(x)+A_k\left(x-\frac{\hbar y}{2}\right)}{2\hbar^2|y|^2} |y|^2[-\alpha_k,\tilde{K}_{q,p}[\Phi](x,y)] \hbar^d dx dy \tfrac{dqdp}{(2\pi\hbar)^d}
 \\ &\qquad \qquad \qquad +  \int_{\mathbb{R}^{4d}}   \tilde{R}(t,x,y) \tfrac{i}{\hbar}A_k\left(x\right)[-\alpha_k,\tilde{K}_{q,p}[\Phi](x,y)]\hbar^d dx dy \tfrac{dqdp}{(2\pi\hbar)^d}
  \\&\qquad \qquad= -\hbar\int_{\mathbb{R}^{4d}}  \tilde{R}(t,x,y)\Delta[A_k](x,y) [-\alpha_k,\tilde{K}_{q,p}[\Delta_p\Phi](x,y)] dx dy \tfrac{dq dp}{(2\pi)^d}
 \\ &\qquad \qquad \qquad + \int_{\mathbb{R}^{4d}}   \tilde{R}(t,x,y)\tfrac{i}{\hbar}[-\alpha_kA_k(x),\tilde{K}_{q,p}[{\Phi}](x,y)] dx dy \tfrac{dq dp}{(2\pi)^d},
 \end{align*}
where we used $|y|^2 \tilde{K}_{q,p} =-\Delta_p \tilde{K}_{q,p}$ and the notation \eqref{eq:SecondOrderFinDiff}, i.e. $\Delta[\,\cdot\,]$, for the second order finite difference.


Now we treat the second term on the right hand side. Note that
\begin{align*}
\tfrac{i}{\hbar}A_k(x)\tilde{K}_{q,p}[\Phi](x,y) &= \tfrac{i}{\hbar}\iint \left({A_k(x)-A_k(q)}\right) G_{\hbar}(x-q)e^{ip\cdot y} e^{-\frac{\hbar}{4}|y|^2} \Phi(q,p)\tfrac{dp dq}{(2\pi \hbar)^d} \\ &\qquad + \tfrac{i}{\hbar}\iint A_k(q) G_{\hbar}(x-q)e^{ip\cdot y} e^{-\frac{\hbar}{4}|y|^2} \Phi(q,p)\tfrac{dp dq}{(2\pi \hbar)^d}.
\end{align*}
By Taylor's theorem,
\begin{align*}
    A_k(x)-A_k(q) = \partial_{q_j}A_k(q)(x_j-q_j) + R^{ij}_k(x,q) (x_i-q_i)(x_j-q_j),
\end{align*}
where the Einstein summation convention is understood and where
\begin{align}
\label{eq:TaylorRemainder}
    R^{ij}_k(x,q) = \int_0^1 (1-\theta)\partial_i\partial_j A_k(q+\theta(x-q))d\theta.
\end{align}
Then 
\begin{align*}
    &\tfrac{i}{\hbar}\iint \left({A_k(x)-A_k(q)}\right) G_{\hbar}(x-q)e^{ip\cdot y} e^{-\frac{\hbar}{4}|y|^2} \Phi(q,p)\tfrac{dp dq}{(2\pi \hbar)^d}
\\
&\qquad =\tfrac{i}{\hbar}\iint  \partial_{q_j}A_k(q)(x_j-q_j)  G_{\hbar}(x-q)e^{ip\cdot y} e^{-\frac{\hbar}{4}|y|^2} \Phi(q,p)\tfrac{dp dq}{(2\pi \hbar)^d} \\
&\qquad \qquad + \tfrac{i}{\hbar}\iint R^{ij}_k(x,q) (x_i-q_i)(x_j-q_j) G_{\hbar}(x-q)e^{ip\cdot y} e^{-\frac{\hbar}{4}|y|^2} \Phi(q,p)\tfrac{dp dq}{(2\pi \hbar)^d}.
\end{align*}
Now we write
\begin{align*}
   R_k^{ij}(x,q) (x_i-q_i)(x_j-q_j) G_{\hbar}(x-q) = R_k^{ij}(x,q) \left[ \tfrac{\hbar^2}{4}\partial_{q_i}\partial_{q_j}G_{\hbar}(x-q) + \tfrac{\hbar}{2}\delta_{ij}G_{\hbar}(x-q) \right]
\end{align*}
and
\begin{equation*}
    (x_j-q_j)G_{\hbar}(x-q) = \tfrac{\hbar}{2}\partial_{q_j}G_{\hbar}(x-q).
\end{equation*}
A partial integration in $q$ yields
\begin{align*}
&\tfrac{i}{\hbar}\iint \left({A_k(x)-A_k(q)}\right) G_{\hbar}(x-q)e^{ip\cdot y} e^{-\frac{\hbar}{4}|y|^2} \Phi(q,p)\tfrac{dp dq}{(2\pi \hbar)^d}
\\
    &\qquad = \tfrac{1}{2i}\iint  G_{\hbar}(x-q)e^{ip\cdot y} e^{-\frac{\hbar}{4}|y|^2} \partial_{q_j}\left(\partial_{q_j}A_k(q)\Phi(q,p)\right)\tfrac{dp dq}{(2\pi \hbar)^d} \\
&\qquad \qquad + \tfrac{i\hbar}{4}\iint  G_{\hbar}(x-q)e^{ip\cdot y} e^{-\frac{\hbar}{4}|y|^2} \partial_{q_i}\partial_{q_j}\left(R^{ij}_k(x,q)\Phi(q,p)\right)\tfrac{dp dq}{(2\pi \hbar)^d}
\\
&\qquad \qquad + \tfrac{i}{2}\iint  G_{\hbar}(x-q)e^{ip\cdot y} e^{-\frac{\hbar}{4}|y|^2} \left(R^{jj}_k(x,q)\Phi(q,p)\right)\tfrac{dp dq}{(2\pi \hbar)^d}
\\
    &\qquad = \tfrac{1}{2i}\iint  G_{\hbar}(x-q)e^{ip\cdot y} e^{-\frac{\hbar}{4}|y|^2} \left[\partial_{q_j}\left(\partial_{q_j}A_k(q)\Phi(q,p)\right)-R^{jj}_k(x,q)\Phi(q,p)\right]\tfrac{dp dq}{(2\pi \hbar)^d} \\
&\qquad \qquad + \tfrac{i\hbar}{4}\iint  G_{\hbar}(x-q)e^{ip\cdot y} e^{-\frac{\hbar}{4}|y|^2}  \partial_{q_i}\partial_{q_j}\left(R^{ij}_k(x,q)\Phi(q,p)\right)\tfrac{dp dq}{(2\pi \hbar)^d}.
\end{align*}
We conclude that
\begin{align*}
\tfrac{i}{\hbar}A_k(x)\tilde{K}_{q,p}[\Phi](x,y) &= \tfrac{1}{2i}\tilde{K}_{q,p}\left[\partial_{q_j}(\partial_{q_j}A_k\Phi)-R^{jj}_k(x,q)\Phi\right](x,y) + \tfrac{i\hbar}{4} \tilde{K}_{q,p}\left[\partial_{q_i}\partial_{q_j}\left(R^{ij}_k(x,q)\Phi\right)\right](x,y) \\
&\qquad + \tfrac{i}{\hbar}\tilde{K}_{q,p}\left[A_k\Phi\right](x,y)
\end{align*}
and therefore
\begin{align*}
\tfrac{i}{\hbar}[-\alpha_k A_k(x),\tilde{K}_{q,p}[\Phi](x,y)] &= \tfrac{1}{2i}\tilde{K}_{q,p}\left[\partial_{q_j}(\partial_{q_j}A_k[-\alpha_k,\Phi])-R^{jj}_k(x,q)[-\alpha_k,\Phi]\right](x,y) \\ &\qquad + \tfrac{i\hbar}{4} \tilde{K}_{q,p}\left[\partial_{q_j}\partial_{q_i}\left(R^{ij}_k(x,q)[-\alpha_k,\Phi]\right)\right](x,y) \\
&\qquad + \tfrac{i}{\hbar}\tilde{K}_{q,p}\left[[-\alpha_kA_k,\Phi]\right](x,y) \\
&= \tfrac{1}{2i}\tilde{K}_{q,p}\left[\partial_{q_j}(\partial_{q_j}A_k[-\alpha_k,\Phi])-R^{jj}_k(x,q)[-\alpha_k,\Phi]\right](x,y) \\ &\qquad + \tfrac{i\hbar}{4} \tilde{K}_{q,p}\left[\partial_{q_i}\partial_{q_j}\left(R^{ij}_k(x,q)[-\alpha_k,\Phi]\right)\right](x,y) \\
&\qquad + \tilde{K}_{q,p}\left[\tfrac{i}{\hbar}[P_{\mathrm{mag}},\Phi]\right](x,y).
\end{align*}
Thus,
\begin{align*}
&\Tr(R(t) \iint  \tfrac{1}{2}\tfrac{i}{\hbar}[-\alpha_k,\Phi(q,p)][A_k,\ket{q,p}\bra{q,p}]_+) \tfrac{dqdp}{(2\pi\hbar)^d} = \\
&\qquad \qquad=-\hbar\int   \tilde{R}(t,x,y)\Delta[A_k](x,y) \tilde{K}_{q,p}[[-\alpha_k,\Delta_p\Phi]](x,y) \hbar^d dx dy \tfrac{dqdp}{(2\pi\hbar)^d}
 \\ &\qquad \qquad \qquad + \tfrac{1}{2i} \int  \tilde{R}(t,x,y)  \tilde{K}_{q,p}\left[\partial_{q_j}\left(\partial_{q_j}A_k[-\alpha_k,\Phi(q,p)]\right)-R^{jj}_k(x,q)[-\alpha_k,\Phi(q,p)]\right](x,y) \hbar^d dx dy \tfrac{dqdp}{(2\pi\hbar)^d} \\
 &\qquad \qquad \qquad + \tfrac{i\hbar}{4} \int  \tilde{R}(t,x,y)  \tilde{K}_{q,p}\left[\partial_{q_i}\partial_{q_j}\left(R^{ij}_k(x,q)[-\alpha_k,\Phi(q,p)]\right)\right](x,y) \hbar^d dx dy \tfrac{dqdp}{(2\pi\hbar)^d}\\
&\qquad \qquad \qquad  + \int  \tilde{R}(t,x,y) \tilde{K}_{q,p}[ \tfrac{i}{\hbar} [P_{\mathrm{mag}},\Phi]](x,y) \hbar^d dx dy \tfrac{dqdp}{(2\pi\hbar)^d}\\
&\qquad \qquad=\hbar\Tr({R}(t)\Delta[A_k] \mathrm{Op}^T_{\hbar}\left[[\alpha_k,\Delta_p\Phi]\right])
 \\ &\qquad \qquad \qquad +  \Tr({R}(t) \mathrm{Op}^T_{\hbar}\left[\tfrac{1}{2i} \partial_{q_j} \left(\partial_{q_j}A_k\left[-\alpha_k,\Phi\right]\right)\right]) - \Tr({R}(t) \mathrm{Op}^T_{\hbar}\left[\tfrac{1}{2i} R^{jj}_k(x,q)[-\alpha_k,\Phi]\right]) \\
 &\qquad \qquad \qquad +  \tfrac{i\hbar}{4} \Tr({R}(t) \mathrm{Op}^T_{\hbar}\left[\partial_{q_j}\partial_{q_i}\left(R^{ij}_k(x,q)\left[-\alpha_k,\Phi\right]\right)\right]) \\
&\qquad \qquad \qquad  + \Tr({R}(t) \mathrm{Op}^T_{\hbar}\left[ \tfrac{i}{\hbar} [P_{\mathrm{mag}},\Phi]\right]). \end{align*}
Collecting the calculations we have
\begin{align}
\label{eq:formula magnetic part quantum} &\Tr(R(t)\tfrac{i}{\hbar}[H_{\mathrm{mag}},\mathrm{Op}^T_{\hbar}[\Phi(t)]]) \\
&\qquad \qquad = -\Tr(R(t)\delta[A_k]_j\mathrm{Op}^T_{\hbar}\left[\tfrac{1}{2}[-\alpha_k,\partial_{p_j}\Phi]_+\right]) \nonumber\\
&\qquad \qquad \qquad + \hbar\Tr({R}(t)\Delta[A_k] \mathrm{Op}^T_{\hbar}\left[[\alpha_k,\Delta_p\Phi]\right])
 \nonumber\\ &\qquad \qquad \qquad +  \Tr({R}(t) \mathrm{Op}^T_{\hbar}\left[\tfrac{1}{2i} \partial_{q_j} \left(\partial_{q_j}A_k\left[-\alpha_k,\Phi\right]\right)\right]) -\Tr({R}(t) \mathrm{Op}^T_{\hbar}\left[\tfrac{1}{2i} R^{jj}_k(x,q)[-\alpha_k,\Phi]\right])\nonumber \\
 &\qquad \qquad \qquad +  \tfrac{i\hbar}{4} \Tr({R}(t) \mathrm{Op}^T_{\hbar}\left[ \partial_{q_j}\partial_{q_i} \left(R^{ij}_k(x,q)\left[-\alpha_k,\Phi\right]\right)\right]) \nonumber \\
&\qquad \qquad \qquad  + \Tr({R}(t) \mathrm{Op}^T_{\hbar}\left[ \tfrac{i}{\hbar} [P_{\mathrm{mag}},\Phi]\right]).\nonumber
\end{align}

Now the Poisson brackets are given by
\begin{align*}
\mathrm{Op}^T_{\hbar}\left[\tfrac{1}{2}(\{P_{\mathrm{mag}},\Phi(t)\}-\{\Phi(t),P_{\mathrm{mag}}\})\right] &= -\mathrm{Op}^T_{\hbar}\left[\tfrac{1}{2} \partial_{q_j}A_k(q) \partial_{p_j}(\alpha_k \Phi+ \Phi\alpha_k)\right] \\ &= \mathrm{Op}^T_{\hbar}\left[\tfrac{1}{2} \partial_{q_j}A_k(q) [-\alpha_k, \partial_{p_j}\Phi]_+\right] 
\end{align*}
Thus, together with \eqref{eq:formula magnetic part quantum},
\begin{align*}
&\Tr(R(t)\tfrac{i}{\hbar}\left[H_{\mathrm{mag}},\mathrm{Op}^T_{\hbar}[\Phi(t)]\right]) - \Tr(R(t)\mathrm{Op}^T_{\hbar}\left[\tfrac{1}{2}(\{P_{\mathrm{mag}},\Phi(t)\}-\{\Phi(t), P_{\mathrm{mag}}\})\right]) \\
& \qquad  = -\Tr(R(t)\delta[A_k]_j\mathrm{Op}^T_{\hbar}\left[\tfrac{1}{2}[-\alpha_k,\partial_{p_j}\Phi]_+\right]-\mathrm{Op}^T_{\hbar}\left[ \partial_{q_j}A_k(q)\tfrac{1}{2}[-\alpha_k,\partial_{p_j}\Phi]_+\right]) \nonumber\\
&\qquad \qquad \qquad + \hbar\Tr({R}(t)\Delta[A_k] \mathrm{Op}^T_{\hbar}\left[[\alpha_k,\Delta_p\Phi]\right])
 \nonumber\\ &\qquad \qquad \qquad +  \Tr({R}(t) \mathrm{Op}^T_{\hbar}\left[\tfrac{1}{2i} \partial_{q_j} \left(\partial_{q_j}A_k\left[-\alpha_k,\Phi\right]\right)\right])-\Tr({R}(t) \mathrm{Op}^T_{\hbar}\left[\tfrac{1}{2i} R^{jj}_k(x,q)[-\alpha_k,\Phi]\right]) \\
 &\qquad \qquad \qquad + \tfrac{i\hbar}{4}  \Tr({R}(t) \mathrm{Op}^T_{\hbar}\left[ \partial_{q_i} \partial_{q_j} \left(R^{ij}_k(x,q)\left[-\alpha_k,\Phi\right]\right)\right]) \\
&\qquad \qquad \qquad  + \Tr({R}(t) \mathrm{Op}^T_{\hbar}\left[ \tfrac{i}{\hbar} [P_{\mathrm{mag}},\Phi]\right]).\nonumber
\end{align*}
This completes the proof.
\end{proof}
\end{lemma}

\subsection{Electric potential part}
Next, we treat the electric part
\begin{equation*}
    H_{\mathrm{el}}(t) = V(t), \quad P_{\mathrm{el}}(t,q) = V(t,q)I_4.
\end{equation*}
Since the electric part is proportional to the identity $I_4$, all matrix commutators vanish and thus only the commutator of the multiplication operator $V$ with the Toeplitz quantization remains. This produces the first order finite difference.
\begin{lemma}\label{lemma:ElectricPart} We have
    \begin{align*} 
&\Tr(R(t)\tfrac{i}{\hbar}[H_{\mathrm{el}}(t),\mathrm{Op}^T_{\hbar}[\Phi(t)]]) + \Tr(R(t)\mathrm{Op}^T_{\hbar}[\tfrac{1}{2}(\{P_{\mathrm{el}}(t),\Phi(t)\}-\{\Phi(t), P_{\mathrm{el}}(t)\})]) \\
&\quad = -\Tr(R(t)\left(\delta[V]_j\mathrm{Op}^T_{\hbar}\left[\partial_{p_j}\Phi\right]-\mathrm{Op}^T_{\hbar}\left[\partial_{q_j}V(q)\partial_{p_j}\Phi\right]\right)).
\end{align*}
\begin{proof}
    The electric part is proportional to the identity and thus,
\begin{align*}
\tfrac{i}{\hbar}[VI_4,\Phi(q,p)\ket{q,p}\bra{q,p}] &= \tfrac{1}{2}\tfrac{i}{\hbar}[I_4,\Phi(q,p)]_+ [V,\ket{q,p}\bra{q,p}] \nonumber\\ &\qquad + \tfrac{1}{2}\tfrac{i}{\hbar}[I_4,\Phi(q,p)][V,\ket{q,p}\bra{q,p}]_+ \\ 
&= \tfrac{i}{\hbar}\Phi(q,p) [V,\ket{q,p}\bra{q,p}]. 
\end{align*}
Thus,
\begin{align*}
&\Tr(R(t)\tfrac{i}{\hbar}[H_{\mathrm{el}},\mathrm{Op}^T_{\hbar}[\Phi(t)]])  \\
&\qquad = \Tr(R(t) \iint \tfrac{i}{\hbar}[VI_4,\Phi(q,p)\ket{q,p}\bra{q,p}]\tfrac{dqdp}{(2\pi\hbar)^d})  \\
&\qquad = \Tr(R(t) \iint \tfrac{i}{\hbar}\Phi(q,p) [V,\ket{q,p}\bra{q,p}]\tfrac{dqdp}{(2\pi\hbar)^d})  \\
& \qquad \qquad = \int  {R}(t,X,Y) \tfrac{i}{\hbar}(V(X)-V(Y))K_{q,p}(X,Y)\Phi(q,p)dX dY \tfrac{dqdp}{(2\pi\hbar)^d} \\ & \qquad \qquad =\int   \tilde{R}(t,x,y) \tfrac{i}{\hbar} \left[{V\left(x+\tfrac{\hbar y}{2}\right)-V\left(x-\tfrac{\hbar y}{2}\right)} \right]\tilde{K}_{q,p}[\Phi](x,y)\hbar^d dx dy \tfrac{dqdp}{(2\pi\hbar)^d} \\
& \qquad \qquad = \int  \tilde{R}(t,x,y) \frac{V\left(x+\tfrac{\hbar y}{2}\right)-V\left(x-\tfrac{\hbar y}{2}\right)}{\hbar|y|} \frac{y_j}{|y|}(iy_j) \tilde{K}_{q,p}[\Phi](x,y)\hbar^d dx dy \tfrac{dqdp}{(2\pi\hbar)^d} \\
& \qquad \qquad = -\int  \tilde{R}(t,x,y) \frac{V\left(x+\tfrac{\hbar y}{2}\right)-V\left(x-\tfrac{\hbar y}{2}\right)}{\hbar|y|} \frac{y_j}{|y|} \tilde{K}_{q,p}[\partial_{p_j}\Phi](x,y) \hbar^d dx dy \tfrac{dqdp}{(2\pi\hbar)^d}
\\
& \qquad \qquad = -\int \tilde{R}(t,x,y) \delta[V]_j(x,y)  \tilde{K}_{q,p}[\partial_{p_j}\Phi](x,y) \hbar^d dx dy \tfrac{dqdp}{(2\pi\hbar)^d}\\
& \qquad \qquad =-\Tr(R(t)\delta[V]_j\mathrm{Op}^T_{\hbar}\left[\partial_{p_j}\Phi]\right]).
\end{align*}
Here we used again notation \eqref{eq:FirstOrderFinDiff} for the first order finite difference.
\end{proof}
\end{lemma}

\subsection{Husimi Extra Term}
The terms
\begin{align*}
     \Tr(R(t)\tfrac{1}{2i}\mathrm{Op}^T_{\hbar}\left[\left[\alpha_k,\partial_{p_k}\Phi(t)\right]\right]) \quad \text{and} \quad \Tr(R(t)\tfrac{1}{2i}\mathrm{Op}^T_{\hbar}\left[\partial_{q_j}\left(\partial_{q_j}A_k\left[-\alpha_k,\Phi(t)\right]\right)\right])
\end{align*}
as well as
\begin{align*}
    -\Tr({R}(t) \mathrm{Op}^T_{\hbar}\left[\tfrac{1}{2i} R^{jj}_k(x,q)[-\alpha_k,\Phi]\right]),
\end{align*}
where $R_k$ is the remainder of the Taylor expansion \eqref{eq:TaylorRemainder},
from Lemmas \ref{lemma:KineticPart} and \ref{lemma:MagneticPart} are $O(1)$ and do not disappear in the semiclassical limit. However they can be written – at least in the case of the Dirac equation – as a commutator with $P$. Therefore they enter the Lagrange multiplier. This is the content of the following lemma.
\begin{lemma}
\label{lemma:HusimiExtraTerm}
Let $P = \alpha \cdot (p-A) +\beta +V$ be the Dirac symbol. We have
        \begin{align}
        &\Tr(R(t)\tfrac{1}{2i}\mathrm{Op}^T_{\hbar}\left[\left[\alpha_k,\partial_{p_k}\Phi(t)\right]+\partial_{q_j}\left(\partial_{q_j}A_k\left[-\alpha_k,\Phi(t)\right]\right)-R^{jj}_k(x,q)[-\alpha_k,\Phi(t)]\right]) \nonumber \\
        &\qquad  = \Tr(R(t)\mathrm{Op}^T_{\hbar}\left[\tfrac{1}{i}[P(t),\mathcal{V}(t)]\right]) + O(\hbar)\left(\|\nabla^4 A(t)\|_{L^{\infty}}\|\Phi(t)\|_{L^2} + \|\nabla^3 A(t)\|_{L^{\infty}}\|\nabla \Phi(t)\|_{L^2}\right).
        \label{eq:HusimiExtraFormula}
    \end{align}
    where
    \begin{equation}
    \label{eq:HusimiExtraTermFinal}
        \mathcal{V} := - \tfrac{1}{4}\Delta_p \Phi + \tfrac14\partial_{q_j}A_k\partial_{q_j}A_m\partial_{p_m}\partial_{p_k} \Phi +\tfrac12\partial_{q_j}A_k\partial_{p_k}\partial_{q_j} \Phi + \tfrac{1}{4} \Delta A_k \partial_{p_k} \Phi
    \end{equation}
    is the \emph{Husimi extra term}.
\end{lemma}

To prove this Lemma suppose that $P$ is a polynomial such that the second order derivatives in $p$ and the mixed second order derivatives in $p$ and $q$ are proportional to the identity $I_4$, i.e.
\begin{align}
    \partial_{p_k}\partial_{p_m}P \propto I_4, && \partial_{p_k}\partial_{q_m}P \propto I_4.\label{eq:SecondOrderDerivatives}
\end{align}
This is clearly the case for the Dirac equation.

We have the following auxilliary lemma.
\begin{lemma} If $P$ is a matrix-valued polynomial such that \eqref{eq:SecondOrderDerivatives} holds and $[P,W]=0$, then
    \begin{equation}
    \label{eq:HusimiExtraTerm1}
        [\partial_{p_k}P,\partial_{q_j}W]=-[\partial_{q_j}P,\partial_{p_k}W]-[P,\partial_{p_k}\partial_{q_j}W].
    \end{equation}
    and
    \begin{equation}
        \label{eq:HusimiExtraTerm3}
        [\partial_{p_k}P,\partial_{p_j}W]+[\partial_{p_j}P,\partial_{p_k}W]=-[P,\partial_{p_j}\partial_{p_k}W].
    \end{equation}
\begin{proof}
First
\[
[P,W]=0\implies\partial_{p_k}[P,W]=[\partial_{p_k}P,W]+[P,\partial_{p_k}W]=0.
\]
Thus we have proved that
\[
[\partial_{p_k}P,W]=-[P,\partial_{p_k}W].
\]
It follows that
\[
\partial_{q_j}\partial_{p_k}P \propto I_4 \implies\partial_{q_j}[\partial_{p_k}P,W]=[\partial_{p_k}P,\partial_{q_j}W]=-[\partial_{q_j}P,\partial_{p_k}W]-[P,\partial_{p_k}\partial_{q_j}W].
\]
Therefore
\[
[\partial_{p_k}P,\partial_{q_j}W]=-[\partial_{q_j}P,\partial_{p_k}W]-[P,\partial_{p_k}\partial_{q_j}W].
\]
This proves \eqref{eq:HusimiExtraTerm1}.
Moreover
\[
\begin{aligned}
0=[P,W]\implies 0=&\partial_{p_j}\partial_{p_k}[P,W]=\partial_{p_j}\left([\partial_{p_k}P,W]+[P,\partial_{p_k}W]\right)
\\
=&[\partial_{p_k}P,\partial_{p_j}W]+[\partial_{p_j}P,\partial_{p_k}W]+[P,\partial_{p_j}\partial_{p_k}W],
\end{aligned}
\]
since
\[
\partial_{p_j}\partial_{p_k}P \propto I_4.
\]
This proves \eqref{eq:HusimiExtraTerm3}.
\end{proof}
\end{lemma}
Now we consider $P = \alpha_k(p_k-A_k(q))+\beta +V(q)I_4$. Then
\begin{align}
\label{eq:HusimiExtraTerm2}
    \partial_{q_j}P=-\partial_{q_j}A_m\partial_{p_m}P + \partial_{q_j}VI_4, && \partial_{p_k}P = \alpha_k.
\end{align}
Moreover we have
\begin{lemma}
    Let \begin{equation*}
        P(p,q) = \alpha_k(p_k-A_k(q))+\beta +V(q)I_4
    \end{equation*}
    and $W$ such that $[P,W]=0$. Then
    \begin{enumerate}
    \item \begin{equation}
\tfrac{1}{2}[\alpha_k,\partial_{p_k}W]=[P,U_1], \quad U_1 := -\tfrac{1}{4}\Delta_p W.
\end{equation}
\item \begin{equation}
\tfrac{1}{2}\partial_{q_j}A_k\partial_{q_j}[-\alpha_k,W]=[P,U_2], \quad
U_2:=\tfrac14\partial_{q_j}A_k\partial_{q_j}A_m\partial_{p_m}\partial_{p_k}W+ \tfrac{1}{2}\partial_{q_j}A_k\partial_{p_k}\partial_{q_j}W.
\end{equation}
\item \begin{align}
\label{eq:HusimiExtraTerm4}
    \tfrac{1}{2}\partial_{q_j}\left(\partial_{q_j}A_k[-\alpha_k,W]\right)=[P,U_3],
\end{align}
with 
\begin{align}
    U_3 &:= U_2 + \tfrac{1}{2}\Delta A_k \partial_{p_k}W \nonumber\\
    &= \tfrac14\partial_{q_j}A_k\partial_{q_j}A_m\partial_{p_m}\partial_{p_k}W+\tfrac{1}{2}\partial_{q_j}A_k\partial_{p_k}\partial_{q_j}W+ \tfrac{1}{2}\Delta A_k \partial_{p_k}W .
\end{align}
\end{enumerate}
\begin{proof}
For the first claim we have 
\[
[P,W]=0\implies 0=\Delta_p[P,W]=[\Delta_pP,W]+2[\partial_{p_k}P,\partial_{p_k}W]+[P,\Delta_pW].
\]
Therefore
\[
\Delta_pP=0\implies[\alpha_k,\partial_{p_k}W]=[\partial_{p_k}P,\partial_{p_k}W]=-\tfrac12[P,\Delta_pW],
\]
so that
\[
\tfrac{1}{2}[\alpha_k,\partial_{p_k}W]=[P,U_1]\quad\text{ with }U_1:=-\tfrac14\Delta_pW.
\]

For the second claim, using \eqref{eq:HusimiExtraTerm1}, \eqref{eq:HusimiExtraTerm3} and \eqref{eq:HusimiExtraTerm2} we find
\[
\begin{aligned}
\partial_{q_j}A_k\partial_{q_j}[-\alpha_k,W]=&-\partial_{q_j}A_k[\partial_{p_k}P,\partial_{q_j}W]
\\
=&\partial_{q_j}A_k[\partial_{q_j}P,\partial_{p_k}W]+\partial_{q_j}A_k[P,\partial_{p_k}\partial_{q_j}W]
\\
=&-\partial_{q_j}A_k\partial_{q_j}A_m[\partial_{p_m}P,\partial_{p_k}W]+\partial_{q_j}A_k[P,\partial_{p_k}\partial_{q_j}W]\\
=&-\partial_{q_j}A_k\partial_{q_j}A_m\tfrac12([\partial_{p_m}P,\partial_{p_k}W]+[\partial_{p_k}P,\partial_{p_m}W])
\\
&+\partial_{q_j}A_k[P,\partial_{p_k}\partial_{q_j}W]
\\
=&\partial_{q_j}A_k\partial_{q_j}A_m\tfrac12[P,\partial_{p_m}\partial_{p_k}W]+\partial_{q_j}A_k[P,\partial_{p_k}\partial_{q_j}W],
\end{aligned}
\]
where in the second to last line we used the symmetry of $\partial_{q_j}A_k\partial_{q_j}A_m$ in $k,m$.

For the last claim note that
\begin{align*}
   \tfrac{1}{2}[-\alpha_k,\partial_{q_j}\left(\partial_{q_j}A_kW\right)] &= -\tfrac{1}{2}\partial_{q_j}A_k\partial_{q_j}[\alpha_k,W] - \tfrac{1}{2}\Delta A_k[\alpha_k,W] \\&=[P,U_2]-\tfrac{1}{2}\Delta A_k[\partial_{p_k}P,W] \\ &=[P,U_2] -\tfrac{1}{2}\Delta A_k \left(\partial_{p_k}[P,W]-[P,\partial_{p_k}W]\right) \\
   &=[P,U_2] + \tfrac{1}{2}[P,\Delta A_k \partial_{p_k}W].
\end{align*}
This concludes the proof.
\end{proof}
\end{lemma}

\begin{proof}[Proof of Lemma \ref{lemma:HusimiExtraTerm}]
    The first two terms combine to
    \begin{equation*}
        - \tfrac{1}{4}\Delta_p \Phi + \tfrac14\partial_{q_j}A_k\partial_{q_j}A_m\partial_{p_m}\partial_{p_k} \Phi +\tfrac{1}{2}\partial_{q_j}A_k\partial_{p_k}\partial_{q_j} \Phi + \tfrac{1}{2}\Delta A_k \partial_{p_k} \Phi.
    \end{equation*}
    For the third term we have the expansion
    \begin{align*}
        R^{jj}_k(x,q) = \tfrac{1}{2}\Delta A_k(q) + \tfrac{1}{6}\nabla \Delta A_k(q) \cdot (x-q) + O(|x-q|^2\|\nabla^4 A\|_{L^{\infty}}).
    \end{align*}
    Then
    \begin{align*}
        \Tr({R}(t) \mathrm{Op}^T_{\hbar}\left[\tfrac{1}{2i} R^{jj}_k(x,q)[-\alpha_k,\Phi]\right]) &= \Tr({R}(t) \mathrm{Op}^T_{\hbar}\left[\tfrac{1}{2i} [-\alpha_k, \tfrac{1}{2}\Delta A_k\Phi]\right]) \\ &\qquad \qquad  + \Tr({R}(t) \mathrm{Op}^T_{\hbar}\left[\tfrac{1}{2i} \nabla \tfrac{1}{6}\Delta A_k\cdot (x-q) [-\alpha_k, \Phi]\right])) \\
        &\qquad \qquad  + \Tr({R}(t) \mathrm{Op}^T_{\hbar}\left[\tfrac{1}{2i} O(|x-q|^2\|\nabla^4 A\|_{L^{\infty}}) [-\alpha_k, \Phi]\right]))
    \end{align*}
    The first term on the right-hand side becomes
    \begin{align*}
        \Tr({R}(t) \mathrm{Op}^T_{\hbar}\left[\tfrac{1}{2i} [-\partial_{p_k}P, \tfrac{1}{2}\Delta A_k\Phi]\right]) &= \Tr({R}(t) \mathrm{Op}^T_{\hbar}\left[\tfrac{1}{i} [P, \tfrac{1}{4}\Delta A_k\partial_{p_k}\Phi]\right]).
    \end{align*}
    Again, by using $(x_j-q_j)G_{\hbar}(x-q) = \tfrac{\hbar}{2}\partial_{q_j}G_{\hbar}(x-q)$, the remaining terms are controlled by $\hbar$ and the norms on the right-hand side of \eqref{eq:HusimiExtraFormula}. Thus the final expression for $\mathcal{V}$ is
    \begin{equation*}
        \mathcal{V} = - \tfrac{1}{4}\Delta_p \Phi + \tfrac14\partial_{q_j}A_k\partial_{q_j}A_m\partial_{p_m}\partial_{p_k} \Phi +\tfrac{1}{2}\partial_{q_j}A_k\partial_{p_k}\partial_{q_j} \Phi + \tfrac{1}{4}\Delta A_k \partial_{p_k} \Phi,
    \end{equation*}
    as desired.
\end{proof}


\subsection{Final result}

\begin{lemma}\label{lemma:ExpressionDistance}
Let
\begin{equation*}
    H(t) = \alpha \cdot (-i\hbar \nabla - A(t)) + \beta +V(t)
\end{equation*}
be the Dirac Hamiltonian and
\begin{equation*}P(t,q,p) = \alpha \cdot(p-A(t,q))+\beta +V(t,q)I_4.
\end{equation*}
be its matrix-valued Dirac symbol. Let $\Phi$ be the solution  of the matrix Liouville equation \eqref{eq:test equation}. Then we have
\begin{align}
\label{eq:ExpressionDistance}
\begin{split}
    &\Tr(R(t)\tfrac{i}{\hbar}[ H(t),\mathrm{Op}^T_{\hbar}[\Phi(t)]]) + \Tr(R(t)\mathrm{Op}^T_{\hbar}\left[\tfrac{1}{2}\left(\{P(t),\Phi(t)\}-\{\Phi(t),P(t)\}\right)+\tfrac{1}{i}[P(t),\mathcal{Z}(t)]\right]) \\
    &\qquad = \underbrace{\Tr(R(t)\mathrm{Op}^T_{\hbar}\left[\tfrac{1}{i}[P,\mathcal{V}+\mathcal{Z}]\right])}_{\text{\emph{Lagrange multiplier and Husimi extra term}}} + \underbrace{O(\hbar)\left(\|\nabla^4 A(t)\|_{L^{\infty}}\|\Phi(t)\|_{L^2} + \|\nabla^3 A(t)\|_{L^{\infty}}\|\nabla \Phi(t)\|_{L^2}\right)}_{\text{\emph{Remainder Husimi extra term}}}\\
 &\qquad \qquad +  \hbar \left[\underbrace{\Tr({R}(t) \mathrm{Op}^T_{\hbar}\left[\tfrac{i}{4} \partial_{q_j}\partial_{q_i}\left(R^{ij}_k(x,q)\left[-\alpha_k,\Phi\right]\right)\right])}_{\text{\emph{Remainder Taylor expansion}}} +\underbrace{\Tr(R(t)\Delta[A_k] \mathrm{Op}^T_{\hbar}\left[[\alpha_k,\Delta_p\Phi]\right])}_{\text{\emph{Second order finite difference}}}\right] \\
    &\qquad \qquad - \underbrace{\Tr(R(t)\left(\delta[A_k]_j\mathrm{Op}^T_{\hbar}\left[\tfrac{1}{2}[-\alpha_k,\partial_{p_j}\Phi]_+\right]-\mathrm{Op}^T_{\hbar}\left[\partial_{q_j}A_k(q)\tfrac{1}{2}[-\alpha_k,\partial_{p_j}\Phi]_+\right]\right))}_{\text{\emph{First order finite difference (magnetic potential)}}} \\
    &\qquad \qquad - \underbrace{\Tr(R(t)\left(\delta[V]_j\mathrm{Op}^T_{\hbar}\left[\partial_{p_j}\Phi\right]-\mathrm{Op}^T_{\hbar}\left[\partial_{q_j}V(q)\partial_{p_j}\Phi\right]\right))}_{\text{\emph{First order finite difference (electric potential)}}}
    \end{split}
\end{align}
\begin{proof}
    The formula is the result of Lemmas \ref{lemma:KineticPart}-\ref{lemma:HusimiExtraTerm} together with the observation that
$\Phi$ is a solution of the test equation \eqref{eq:test equation} and therefore satisfies $[P(t),\Phi(t)] = 0$ for all $t\in [0,T]$. Moreover, we note that $\mathcal{L} = \tfrac{1}{i}[P,\mathcal{Z}]$. This proves the lemma.
\end{proof}
\end{lemma}

\section{\texorpdfstring{Gronwall estimate for the pseudo-distance $\mathfrak{D}(t)$}{Gronwall estimate for the pseudo-distance D(t)}}
\label{sec:GronwallDirac}

In this section we use expression \eqref{eq:ExpressionDistance} found in Lemma \ref{lemma:ExpressionDistance} to prove the Gronwall estimate. We treat (in this order) the first order finite difference terms, which yield a $\sqrt{h}$ convergence rate; the second order finite difference and Taylor remainder term, which yield $\hbar$ convergence rate and finally the Lagrange multiplier, which is the most delicate part, as we have to control the commutator $[P,W_{\hbar}]$ of the Dirac symbol and the Wigner transform, for which we have to propagate derivatives of $W_{\hbar}$ in time.

\subsection{First-order finite difference}

Here, we treat the two first order finite difference terms on the right hand side of \eqref{eq:ExpressionDistance}. The case of the electric potential is completely analogous to the scalar case since $V$ acts as a multiple of the identity matrix $I_4$. The magnetic field $A$ acts on the anticommutator $\tfrac{1}{2}[\alpha_k,\partial_{p_j}\Phi]_+$. However, since $\alpha_k$ is a constant matrix, the proof is similar to the electric potential. 
\begin{lemma}\label{lemma:GronwallFirstOrder}
Assume that $\Phi \in C([0,T],H^s)$ is a solution of the matrix Liouville equation \eqref{eq:test equation} with $s$ large enough. Moreover, suppose that $\nabla A$ and $\nabla V$ are Lipschitz-continuous and that $R$ satisfies the mixed state condition
\begin{equation}
    \|R^{\mathrm{in}}\|_{\mathrm{HS}} \leq C(2\pi\hbar)^{d/2} \Leftrightarrow \|\tilde{R}\|_{L^2_{x,y}}\leq C.
\end{equation}
Then
    \begin{align*}
        &\left|\Tr(R(t)\left(\delta[V]_j\mathrm{Op}^T_{\hbar}\left[\partial_{p_j}\Phi\right]-\mathrm{Op}^T_{\hbar}\left[\partial_{q_j}V(q)\partial_{p_j}\Phi\right]\right))\right| \\
        &\qquad \leq C\sqrt{\hbar} \|\nabla^2V\|_{L^{\infty}}\|\nabla_p\Phi\|_{L^2_{q,p}}.
    \end{align*}
    and
     \begin{align*}
        &\left|\Tr(R(t)\left(\delta[A_k]_j\mathrm{Op}^T_{\hbar}\left[\tfrac{1}{2}[\alpha_k,\partial_{p_j}\Phi]_+\right]-\mathrm{Op}^T_{\hbar}\left[\partial_{q_j}A_k(q)\tfrac{1}{2}[\alpha_k,\partial_{p_j}\Phi]_+\right]\right))\right| \\
        &\qquad \qquad  \leq C\sqrt{\hbar}\|\nabla^2 A\|_{L^{\infty}}\|\nabla_p\Phi\|_{L^2_{q,p}}.
    \end{align*}
    \begin{proof}
    \underline{The electric potential $V$:} From the proof of Lemma \ref{lemma:ElectricPart} we have
        \begin{align*}
            &\Tr(R(t)\left(\delta[V]_j\mathrm{Op}^T_{\hbar}\left[\partial_{p_j}\Phi\right]-\mathrm{Op}^T_{\hbar}\left[\partial_{q_j}V(q)\partial_{p_j}\Phi\right]\right)) \\&\qquad \qquad = \int \hbar^d dx dy  \tilde{R}(t,x,y) \left[\frac{V\left(x+\tfrac{\hbar y}{2}\right)-V\left(x-\tfrac{\hbar y}{2}\right)}{\hbar|y|} \frac{y_j}{|y|}\tilde{K}_{q,p}[\partial_{p_j}\Phi](x,y)- \tilde{K}_{q,p}[\partial_{q_j}V(q)\partial_{p_j}\Phi](x,y)\right] \\
            &= \int \tfrac{dx dy dq}{(2\pi)^d}  \tilde{R}(t,x,y) \left[\frac{V\left(x+\tfrac{\hbar y}{2}\right)-V\left(x-\tfrac{\hbar y}{2}\right)}{\hbar|y|} \frac{y_j}{|y|}-\partial_{q_j}V(q)\right]\mathcal{F}^{-1}_{p\rightarrow y}\left[\partial_{p_j}\Phi(q,p)\right] G_{\hbar}(x-q)e^{-\frac{\hbar}{4}|y|^2} \\
            &=i \int \tfrac{dx dy dq}{(2\pi)^d}   \tilde{R}(t,x,y) \left[\frac{V\left(x+\tfrac{\hbar y}{2}\right)-V\left(x-\tfrac{\hbar y}{2}\right)}{\hbar} -y_j\partial_{x_j}V(x)\right]\mathcal{F}^{-1}_{p\rightarrow y}\left[\Phi(q,p)\right] G_{\hbar}(x-q)e^{-\frac{\hbar}{4}|y|^2} \\
            &\qquad + \int \tfrac{dx dy dq}{(2\pi)^d}   \tilde{R}(t,x,y) \left[\partial_{x_j}V(x)-\partial_{q_j}V(q)\right]\mathcal{F}^{-1}_{p\rightarrow y}\left[\partial_{p_j}\Phi(q,p)\right] G_{\hbar}(x-q)e^{-\frac{\hbar}{4}|y|^2} \\
            &=: I_1 + I_2.
        \end{align*}
        We now estimate $|I_1|$ and $|I_2|$ individually. By assumption, $\nabla V$ is Lipschitz and thus we have
\begin{align}
\label{eq:VLipschitz1_Dirac}
\left|\tfrac{1}{\hbar}\left(V(x+\tfrac{\hbar y}{2})-V(x-\tfrac{\hbar y}{2})\right)- y_j\partial_{x_j} V(x)  \right| \leq C\hbar |y|^2 \|\nabla^2 V\|_{\infty},
\end{align}
and
\begin{align}
\label{eq:VLipschitz2_Dirac}
\left|\partial_{x_j}V(x) - \partial_{q_j}V(q)  \right| \leq C |x-q| \|\nabla^2 V\|_{\infty}.
\end{align}
Then we use \eqref{eq:VLipschitz1_Dirac}, the Cauchy-Schwarz and Young inequalities as well as Plancherel's theorem to estimate
\begin{align*}
|I_1| &\leq C\|\tilde{R}\|_{L^2_{x,y}}\\&\quad\cdot\left\|\int_{\mathbb{R}^{d}}\left|\tfrac{1}{\hbar}\left(V(x+\tfrac{\hbar y}{2})-V(x-\tfrac{\hbar y}{2})\right)- y_j\partial_{x_j} V(x)  \right| \left| \mathcal{F}^{-1}_{p\rightarrow y}[\Phi](q,y)\right| G_{\hbar}(x-q)e^{-\frac{\hbar}{4}|y|^2} dq \right\|_{L^2_{x,y}} \\
&\leq C \left\|\sup_{x,y} \left|\tfrac{1}{\hbar}\left(V(x+\tfrac{\hbar y}{2})-V(x-\tfrac{\hbar y}{2})\right)- y_j\partial_{x_j} V(x) \right|\int_{\mathbb{R}^{d}} \left| \mathcal{F}^{-1}_{p\rightarrow y}[\Phi](q,y)\right| G_{\hbar}(x-q)e^{-\frac{\hbar}{4}|y|^2} dq \right\|_{L^2_{x,y}}
\\
&\leq C\sqrt{\hbar}\|\nabla^2 V \|_{L^{\infty}} \left\|\int_{\mathbb{R}^{d}} \left| y\mathcal{F}^{-1}_{p\rightarrow y}[\Phi](q,y)\right| G_{\hbar}(x-q)\sqrt{\hbar}|y|e^{-\frac{\hbar}{4}|y|^2} dq \right\|_{L^2_{x,y}} \\
&\leq C\sqrt{\hbar}\|\nabla^2 V \|_{L^{\infty}} \left\| \left(\left| \mathcal{F}^{-1}_{p\rightarrow y}[\partial_{p_j}\Phi](\cdot,y)\right| \ast_x G_{\hbar}(\cdot)\right)(x)\right\|_{L^2_{x,y}} 
\\
&\leq C\sqrt{\hbar}\|\nabla^2 V \|_{L^{\infty}} \left\| \left\| \mathcal{F}^{-1}_{p\rightarrow y}[\nabla_{p_j} \Phi](x,y) \right\|_{L^2_x}\right\|_{L^2_{y}} 
\\
&= C\sqrt{\hbar}\|\nabla^2 V \|_{L^{\infty}} \left\| \left\| \nabla_y \Phi \right\|_{L^2_x}\right\|_{L^2_{y}} 
\\
&= C\sqrt{\hbar}\|\nabla^2 V \|_{L^{\infty}} \left\| \nabla_y \Phi \right\|_{L_{x,y}}. 
\end{align*} 
For $I_2$ we proceed as follows
\begin{align*}
|I_2| &\leq C\left\|\int_{\mathbb{R}^{d}}\left|\partial_{x_j}V(x)- \partial_{q_j} V(q) \right| \left| \mathcal{F}^{-1}_{p\rightarrow y}[\partial_{p_j}\Phi](q,y)\right| G_{\hbar}(x-q)e^{-\frac{\hbar}{4}|y|^2} dq \right\|_{L^2_{x,y}} \\
&\leq C\sqrt{\hbar}\|\nabla^2 V \|_{L^{\infty}} \left\|\int_{\mathbb{R}^{d}} \left| \mathcal{F}^{-1}_{p\rightarrow y}[\partial_{p_j}\Phi](q,y)\right| \frac{|x-q|}{\sqrt{\hbar}}G_{\hbar}(x-q)e^{-\frac{\hbar}{4}|y|^2} dq \right\|_{L^2_{x,y}} \\
&\leq C\sqrt{\hbar}\|\nabla^2 V \|_{L^{\infty}} \left\| \left(\left| \mathcal{F}^{-1}_{p\rightarrow y}[\partial_{p_j}\Phi](\cdot,y)\right| \ast_x \left(\frac{|\cdot|}{\sqrt{\hbar}}G_{\hbar}(\cdot)\right)\right)(x)\right\|_{L^2_{x,y}}
\\
&\leq C\sqrt{\hbar}\|\nabla^2 V \|_{L^{\infty}} \left\| \left\| \mathcal{F}^{-1}_{p\rightarrow y}[\partial_{p_j}\Phi](x,y)\right\|_{L^2_x}\right\|_{L^2_{y}} 
\\
&= C\sqrt{\hbar}\|\nabla^2 V \|_{L^{\infty}} \left\| \nabla_y \Phi \right\|_{L^2_{x,y}}.
\end{align*} 
This concludes the proof for the electric potential. \\

\underline{The magnetic potential $A$:} Similarly, we rewrite
        \begin{align*}
            &\Tr(R(t)\left(\delta[A_k]_j\mathrm{Op}^T_{\hbar}\left[\tfrac{1}{2}[\alpha_k,\partial_{p_j}\Phi]_+\right]-\mathrm{Op}^T_{\hbar}\left[\partial_{q_j}A_k(q)\tfrac{1}{2}[\alpha_k,\partial_{p_j}\Phi]_+\right]\right)) \\&  = \int \hbar^d dx dy  \tilde{R}(t,x,y) \Bigg[\frac{A_k\left(x+\tfrac{\hbar y}{2}\right)-A_k\left(x-\tfrac{\hbar y}{2}\right)}{\hbar|y|} \frac{y_j}{|y|}\tilde{K}_{q,p}[\tfrac{1}{2}[\alpha_k,\partial_{p_j}\Phi]_+](x,y) \\ &\qquad - \tilde{K}_{q,p}[\partial_{q_j}A_k(q)\tfrac{1}{2}[\alpha_k,\partial_{p_j}\Phi]_+](x,y)\Bigg] \\
            &  = i\int \tfrac{dx dy dq}{(2\pi)^d}   \tilde{R}(t,x,y) \Bigg[\frac{A_k\left(x+\tfrac{\hbar y}{2}\right)-A_k\left(x-\tfrac{\hbar y}{2}\right)}{\hbar}-y_j\partial_{x_j}A_k(x)\Bigg]\mathcal{F}_{p\rightarrow y}\left[\tfrac{1}{2}[\alpha_k,\Phi]_+\right] (q,y)\\&\qquad \cdot G_{\hbar}(x-q)e^{-\frac{\hbar}{4}|y|^2} \\
            &\qquad + \int \tfrac{dx dy dq}{(2\pi)^d}  \tilde{R}(t,x,y) \left[\partial_{x_j}A_k(x)-\partial_{q_j}A_k(q)\right]\mathcal{F}_{p\rightarrow y}\left[\tfrac{1}{2}[\alpha_k,\partial_{p_j}\Phi]_+\right] (q,y)\\&\qquad \qquad \cdot G_{\hbar}(x-q)e^{-\frac{\hbar}{4}|y|^2}.
        \end{align*}
        From here, the estimate continues exactly as above, since
        \begin{align*}
            \left|\tfrac{1}{2}[\alpha_k,\partial_{p_j}\Phi(q,p)]_+\right| \leq |\partial_{p_j}\Phi(q,p)|,
        \end{align*}
        as $\alpha_k$ is a constant matrix.
    \end{proof}
\end{lemma}

\subsection{Second-order finite difference and Taylor remainder term}
In this section we study the $O(\hbar)$ terms on the right hand side of \eqref{eq:ExpressionDistance}.
\begin{lemma}\label{lemma:GronwallSecondOrder}
Assume that $\Phi \in C([0,T],H^s)$ is a solution of the matrix Liouville equation \eqref{eq:test equation} with $s$ large enough. Moreover, suppose that $\nabla^s A$ is Lipschitz-continuous for $s$ large enough and that $R$ satisfies the mixed state condition
\begin{equation}
    \|R^{\mathrm{in}}\|_{\mathrm{HS}} \leq C(2\pi\hbar)^{d/2} \Leftrightarrow \|\tilde{R}\|_{L^2_{x,y}}\leq C.
\end{equation}
Then
    \begin{align}
        &\left|\tfrac{i\hbar}{4} \Tr({R}(t) \mathrm{Op}^T_{\hbar}\left[\partial_{q_j}\partial_{q_i}\left(R^{ij}_k(x,q)\left[\alpha_k,\Phi\right]\right)\right])\right| \leq C\hbar\|A\|_{W^{4,\infty}} \| \Phi\|_{H^2_{q,p}},
        \label{eq:EstimateTaylorRemainder}
        \end{align}
        and
        \begin{align}
        \label{eq:EstimateSecondFiniteDiff}
        \left|\hbar \Tr(R(t)\Delta[A_k] \mathrm{Op}^T_{\hbar}\left[[\alpha_k,\Delta_p\Phi]\right]) \right| \leq C\hbar\|A\|_{W^{2,\infty}}  \|\Phi\|_{H^2_{q,p}}.
    \end{align}
    \begin{proof}
        For the Taylor remainder term \eqref{eq:EstimateTaylorRemainder} we have, together with $\|\tilde{R}\|_{L^2}\leq C$, 
        \begin{align*}
            &\left|\tfrac{i\hbar}{4}\Tr({R}(t) \mathrm{Op}^T_{\hbar}\left[ \partial_{q_j}\partial_{q_i}\left(R^{ij}_k(x,q)\left[\alpha_k,\Phi\right]\right)\right])\right| \\&\qquad \qquad \leq C\hbar \|\tilde{R}\|_{L^2_{x,y}}\left\|\int \partial_{q_j}\partial_{q_i}\left(R^{ij}_k(x,q)\mathcal{F}^{-1}_{p\rightarrow y}\left[[\alpha_k,\Phi]\right](q,y)\right)G_{\hbar}(x-q)e^{-\frac{\hbar}{4}|y|^2}dq\right\|_{L^2_{x,y}}
            \\&\qquad \qquad \leq C\hbar\left\|\int \partial_{q_j}\partial_{q_i}\left(\int_0^1 (1-\theta)\partial_i \partial_j A_k(q+\theta(x-q))d\theta\mathcal{F}^{-1}_{p\rightarrow y}\left[[\alpha_k,\Phi]\right](q,y)\right)\right.\\ &\qquad \qquad \qquad \left.\cdot G_{\hbar}(x-q)e^{-\frac{\hbar}{4}|y|^2}dq\right\|_{L^2_{x,y}} \\
            &\qquad \qquad \leq C\hbar\int_0^1 (1-\theta) \left\|\partial_{q_j}\partial_{q_i}\left(\partial_i \partial_j A_k(q+\theta(x-q))\mathcal{F}^{-1}_{p\rightarrow y}\left[[\alpha_k,\Phi]\right](q,y)\right) G_{\hbar}(x-q)e^{-\frac{\hbar}{4}|y|^2}dq \right\|_{L^2_{x,y}}d\theta \\&\qquad \qquad \leq \tfrac{1}{2}C\hbar\sup_{0\leq\theta\leq 1}\left\|\partial_{q_j}\partial_{q_i}\left(\partial_i \partial_j A_k(q+\theta(x-q))\mathcal{F}^{-1}_{p\rightarrow y}\left[[\alpha_k,\Phi]\right](q,y)\right) G_{\hbar}(x-q)e^{-\frac{\hbar}{4}|y|^2}dq \right\|_{L^2_{x,y}} \\
            &\qquad \qquad \leq \tfrac{1}{2}C\hbar \left(\|\nabla^{4} A_k\|_{L^{\infty}} \|\Phi\|_{L^2_{x,y}}+ \|\nabla^{3} A_k\|_{L^{\infty}} \|\nabla_q \Phi\|_{L^2_{x,y}} + \|\nabla^{2} A_k\|_{L^{\infty}} \|\nabla_q^2\Phi\|_{L^2_{q,p}}\right).
        \end{align*}

        For the second order finite difference term \eqref{eq:EstimateSecondFiniteDiff} we have
        \begin{align*}
            &\left|\hbar \Tr(R(t)\Delta[A_k] \mathrm{Op}^T_{\hbar}\left[[\alpha_k,\Delta_p\Phi]\right]) \right| \\&\qquad \qquad \leq C\hbar \|\tilde{R}\|_{L^2_{x,y}} \left\| \int_{\mathbb{R}^d_q} \Delta[A_k](x,y)  \mathcal{F}^{-1}_{p\rightarrow y}[[\alpha_k,\Delta_p\Phi]](q,y)G_{\hbar}(x-q)e^{-\frac{\hbar}{4}|y|^2} dq\right\|_{L^2_{x,y}}
        \end{align*}
        Now $\Delta[A_k]$ is given by, cf. \eqref{eq:SecondOrderFinDiff}
         \begin{equation}
     \Delta[A_k](x,y) := i\frac{A_k\left(x+\frac{\hbar y}{2}\right)-2A_k(x)+A_k\left(x-\frac{\hbar y}{2}\right)}{2\hbar^2|y|^2} 
 \end{equation}
 and so
 \begin{equation}
     \sup_{x,y}\left|\Delta[A_k](x,y)\right| \leq \|\nabla^2 A_k\|_{L^{\infty}}.
 \end{equation}
 Hence, together with $\|\tilde{R}\|_{L^2}\leq C$,
 \begin{align*}
            &\left|{\hbar} \Tr(R(t)\Delta[A_k] \mathrm{Op}^T_{\hbar}\left[[\alpha_k,\Delta_p\Phi]\right]) \right| \\&\qquad \qquad \leq C\hbar \|\tilde{R}\|_{L^2_{x,y}} \|\nabla^2 A_k\|_{L^{\infty}} \left\|\int_{\mathbb{R}^d_q} \mathcal{F}^{-1}_{p\rightarrow y}[[\alpha_k,\Delta_p\Phi]](q,y)G_{\hbar}(x-q)e^{-\frac{\hbar}{4}|y|^2} dq\right\|_{L^2_{x,y}} \\
            &\qquad \qquad \leq C\hbar  \|\nabla^2 A_k\|_{L^{\infty}} \left\| \mathcal{F}^{-1}_{p\rightarrow y}[[\alpha_k,\Delta_p\Phi]](\cdot,y) \ast G_{\hbar}(\cdot) e^{-\frac{\hbar}{4}|y|^2} \right\|_{L^2_{x,y}} \\
            &\qquad \qquad \leq C\hbar \|\nabla^2 A_k\|_{L^{\infty}} \left\|\Delta_p \Phi\right\|_{L^2_{q,p}}.
        \end{align*}
        This completes the proof.
    \end{proof}
\end{lemma}
\subsection{Lagrange multiplier and Gronwall estimate for the commutator}
\label{sec:Lagrange multiplier}
We seek to estimate the remaining term on the right hand side of \eqref{eq:ExpressionDistance}, i.e.
\begin{align*}
    \Tr(R(t)\mathrm{Op}^T_{\hbar}\left[\tfrac{1}{i}[P(t),\mathcal{Z}(t)+\mathcal{V}(t)]\right]),
\end{align*}
where
\begin{equation*}
    \mathcal{L} = \tfrac{1}{i}[P,\mathcal{Z}]
\end{equation*}
is the Lagrange multiplier of the dual matrix Liouville equation and where the Husimi extra term $\mathcal{V}$ \eqref{eq:HusimiExtraTermFinal} is given by (cf. Lemma \ref{lemma:HusimiExtraTerm})
  \begin{equation*}
        \mathcal{V} = - \tfrac{1}{4}\Delta_p \Phi + \tfrac14\partial_{q_j}A_k\partial_{q_j}A_m\partial_{p_m}\partial_{p_k} \Phi +\tfrac12\partial_{q_j}A_k\partial_{p_k}\partial_{q_j} \Phi+\tfrac14 \Delta A_k \partial_{p_k} \Phi,
    \end{equation*}
We have the following estimate for the Lagrange multiplier including the Husimi extra term.
    \begin{lemma}
\label{thm:EstimateLagrangeMultiplier}
    Let $\Phi \in C([0,T],H^s)$ be a solution of the matrix Liouville equation \eqref{eq:test equation} with $s$ large enough and suppose that $A,V \in H^s_{t,x}$ with $s$ large enough. Moreover, suppose that $R$ satisfies the hypotheses of Theorem \ref{thm:main}, i.e. \eqref{eq:vonNeumannInitialDataMainThm}. Then
    \begin{align*}
    &\left| \Tr(R(t)\mathrm{Op}^T_{\hbar}\left[\tfrac{1}{i}[P(t),\mathcal{Z}(t)+\mathcal{V}(t)]\right])\right|  \leq C\left[\hbar (1+F)\exp(Ct(1+F))\|\Phi(t)\|_{H^2}\right],
\end{align*}
where
\begin{equation}
    F \equiv F(\|A\|_{W^{2,\infty}_{t,q}},\|V\|_{W^{2,\infty}_{t,q}}).
\end{equation}

\end{lemma}
In order to prove the lemma observe first that
\begin{equation}
    \Tr(R(t)\mathrm{Op}^T_{\hbar}[\Psi(t)]) = \iint \tilde{W}_{\hbar} (t,q,p)\Psi(t,q,p)dqdp,
\end{equation}
where $$\tilde{W}_{\hbar}(t,q,p) := (G_{\hbar} \ast_{q,p} W_{\hbar})(t,q,p)$$ is the Husimi transform of $W_{\hbar}$ (here, $G_{\hbar}$ is the Gaussian, cf. \eqref{eq:Gaussian}). Then we have (by the cyclicity of the trace)
\begin{align}
\Tr(R\mathrm{Op}^T_{\hbar}\left[\tfrac{1}{i}[P,\mathcal{Z}+\mathcal{V}]\right]) 
&= \iint \tr(\tilde{W}_{\hbar}\tfrac{1}{i}[P,\mathcal{Z}+\mathcal{V}])dqdp \nonumber\\
&= -\iint \tr(\tfrac{1}{i}[P,\tilde{W}_{\hbar}](\mathcal{Z}+\mathcal{V}))dqdp \nonumber\\
&= -\iint \tr(G_{\hbar}\ast\tfrac{1}{i}[P,W_{\hbar}](\mathcal{Z}+\mathcal{V}))dqdp \nonumber\\ &\qquad - \iint \tr((\tfrac{1}{i}[P,G_{\hbar}\ast W_{\hbar}]-G_{\hbar}\ast\tfrac{1}{i}[P,W_{\hbar}])(\mathcal{Z}+\mathcal{V}))dqdp\nonumber\\
&=: H_1 + H_2, \nonumber
\end{align}
where
\begin{align}
\label{eq:LagrangeMultiplierH1}
    H_1 &:= -\iint \tr(G_{\hbar}\ast\tfrac{1}{i}[P,W_{\hbar}](\mathcal{Z}+\mathcal{V}))dqdp, \\
    \label{eq:LagrangeMultiplierH2}
    H_2 &:=- \iint \tr((\tfrac{1}{i}[P,G_{\hbar}\ast W_{\hbar}]-G_{\hbar}\ast\tfrac{1}{i}[P,W_{\hbar}])(\mathcal{Z}+\mathcal{V}))dqdp.
\end{align}

At this point, in order to gauge the convergence rate of $H_1$ and $H_2$ we have to understand the convergence of $\tfrac{1}{i}[P,W_{\hbar}]$. Note that a simple estimate of $\tfrac{1}{i}[P,W_{\hbar}]$ in $L^2$ by $2\|P\|_{L^{\infty}}\|W_{\hbar}\|_{L^2}$ is not enough for two reasons. First, this estimate does not yield a rate of convergence in $\hbar$. Second, $P\in L^{\infty}$ is a strong assumption that is not satisfied by the Dirac Hamiltonian.

By using the Wigner equation we have
\begin{align*}
\tfrac{1}{i}[P,W_{\hbar}] = \hbar\left(\partial_t W_{\hbar} - \tfrac{1}{2}\left( \{P,W_{\hbar}\} - \{W_{\hbar},P\}\right) - r_{\hbar}\right),
\end{align*}
where $r_{\hbar}$ is the remainder given by \eqref{eq:Remainder}. This suggests that if one controls the quantities
\begin{align*}
\partial_t W_{\hbar}, \qquad  \nabla_q W_{\hbar}, \qquad  \nabla_p W_{\hbar}, \qquad \text{and} \qquad r_{\hbar},
\end{align*}
in $L^2$, one obtains an estimate of $\tfrac{1}{i}[P,W_{\hbar}]$ such that $\tfrac{1}{i}[P,W_{\hbar}] = O(\hbar)_{L^2}$.
\begin{proposition}
\label{lemma:GronwallLagrangeWigner} Let $A,V \in W^{2,\infty}(\mathbb{R}_t\times \mathbb{R}^d_{q})$ be external electromagnetic potentials.
    Suppose that the initial Wigner transform satisfies
    \begin{align}
    \label{eq:AssumptionInitialDerivativesWigner}
        \|W^{\mathrm{in}}_{\hbar}\|_{H^1_{q,p}} + \|\partial_t W|_{t=0}\|_{L^2_{q,p}} \leq C.
    \end{align}
    Then we have the following propagation of the derivatives of the Wigner transform $W_{\hbar}$:
    \begin{equation}\label{eq:SpaceTimeGradientWigner}\|\nabla_{t,q,p}W_{\hbar}(t)\|_{L^2_{q,p}} \leq C\exp(Ct(1+F(\|A\|_{W^{2,\infty}_{t,q}},\|V\|_{W^{2,\infty}_{t,q}}))).
    \end{equation}
    where $F$ is a continuous function vanishing at the origin.
    
    In particular, the commutator $[P,W_{\hbar}]$ is propagated as follows:
    \begin{align}
    \|[P(t),W_{\hbar}(t)]\|_{L^2} &\leq \hbar\left[C(1+F(\|A\|_{W^{2,\infty}_{t,q}},\|V\|_{W^{2,\infty}_{t,q}})) \exp(Ct(1+F(\|A\|_{W^{2,\infty}_{t,q}},\|V\|_{W^{2,\infty}_{t,q}})))\right]. \label{eq:PropagationCommutatorLemma}
\end{align}
\begin{remark}\label{Remark:Preconditioning}
    Let us explain the main idea behind the proof. The problem has the following structure:
    \begin{equation}
    \label{eq:ModelProblem}
        \partial_t U = LU := \tfrac{1}{\varepsilon}f(U) +BU
    \end{equation}
    where $B$ and $f$ are skew-adjoint, e.g. $\frac{1}{i}[P,\cdot]$). Then one clearly has an energy estimate for $U$,
    \begin{align*}
        \tfrac{1}{2}\partial_t \|U\|^2 &=0.
    \end{align*}
On the other hand, if $d$ is another derivation (such as $\partial_t$, $\partial_{q_j}$ and $\partial_{p_j}$) then the energy estimate fails for $dU$, since
\begin{align*}
        \tfrac{1}{2}\partial_t \|dU\|^2 &= \tfrac{1}{\varepsilon}\langle dU, [d,f](U) \rangle + \tfrac{1}{\varepsilon}\langle dU, f(dU) \rangle + \langle dU,BdU \rangle + \langle dU, [d,B] U\rangle \\
        &= \tfrac{1}{\varepsilon}\langle dU, [d,f](U) \rangle + \langle dU, [d,B] U\rangle.
    \end{align*}
    The second term on the right-hand side is not problematic as it involves no oscillations in $\varepsilon$ and only first order derivatives of $U$. On the other hand, the first term carries a factor $\frac{1}{\varepsilon}$ and thus it is not possible to obtain a uniform estimate in $\varepsilon$. To circumvent this obstacle, we precondition the equation by introducing another derivation $\tilde{d}$ such that
    \begin{equation}
        \|dU\| \leq C \|\tilde{d}U\|.
    \end{equation}
    Moreover,
    \begin{align*}
        \partial_t \tilde{d} U &= \tilde{d} \partial_t U + [\partial_t,\tilde{d}]U \\
        &= \tilde{d}\left(\tfrac{1}{\varepsilon}f(U) +BU\right) + [\partial_t,\tilde{d}]U \\
        &= \tfrac{1}{\varepsilon}f(\tilde{d} U) + \tfrac{1}{\varepsilon} [\tilde{d},f](U)+B\tilde{d}U + [\tilde{d},B]U + [\partial_t,\tilde{d}]U. 
    \end{align*}
    Thus, if we impose that $\tilde{d}$ is such that
    \begin{equation}
    \label{eq:ConditionDerivationGeneral}
        [\tilde{d},f] = k_df
    \end{equation}
    – where $k_d$ is some function that depends on the choice of the derivation $d$ –  we can use equation \eqref{eq:ModelProblem} to obtain
    \begin{align*}
        \partial_t \tilde{d} U
        &= \tfrac{1}{\varepsilon}f(\tilde{d} U) + \tfrac{1}{\varepsilon} k_df(U)+B\tilde{d}U + [\tilde{d},B]U + [\partial_t,\tilde{d}]U \\
         &= \tfrac{1}{\varepsilon}f(\tilde{d} U) + k_d\left(\partial_t U - BU\right)+B\tilde{d}U + [\tilde{d},B]U + [\partial_t,\tilde{d}]U \\ 
         &= L\tilde{d} U + k_d\left(\partial_t U - BU\right) + [\tilde{d},B-\partial_t]U.
    \end{align*}
    This means that we removed the oscillations in $\varepsilon$. Since $L$ is skew-adjoint and the remaining terms depend linearly on derivatives of $U$ we can now perform an energy estimate for $\tilde{d}U$ that is independent of $\varepsilon$. Indeed,
    \begin{align}
    \label{eq:EnergyEstimateCovariant}
        \tfrac{1}{2}\partial_t \|\tilde{d} U\|^2
        \leq  \langle \tilde{d}U, k_d\left(\partial_t U - BU\right) + [\tilde{d},B-\partial_t]U\rangle . 
    \end{align}
It therefore remains to determine $\tilde{d}$.

For the Dirac equation we have $f = \tfrac{1}{i}[P,\cdot]$. Since $\tilde{d}$ is a derivation we have, for some matrix-valued function $M$, 
\begin{align*}
    \left[\tilde{d}, \tfrac{1}{i}[P,\cdot]\right] M= \tilde{d} \tfrac{1}{i}[P,M] - \tfrac{1}{i}[P,\tilde{d}M] = \tfrac{1}{i}[\tilde{d}P,M].
\end{align*}
Since this should be equal to $k_d\frac{1}{i}[P,M]$, condition \eqref{eq:ConditionDerivationGeneral} becomes
\begin{equation}
    \label{eq:ConditionDerivationDirac}\tilde{d}P = k_dP.
\end{equation}
Now for the Dirac equation $P = \lambda_0 Q + VI_4$ with 
\begin{equation}
\label{eq:Lambda0}
\lambda_0 := \sqrt{1+|\xi-A|^2}    
\end{equation} and $Q:= \Pi_+-\Pi_-$ where $Q^2 = I_4$. In fact, it suffices to consider only $P_0 := \lambda_0 Q$ since the diagonal part $VI_4$ doesn't play any role. Since we impose $\|dU\| \leq C \|\tilde{d}U\|$, a possible choice for $\tilde{d}$ is
\begin{equation}
    \tilde{d} = d + [C_d,\cdot],
\end{equation}where $C_d$ is some matrix depending on the choice of the derivation $d$. Using this ansatz we obtain, using that $\lambda_0$ is scalar,
\begin{align*}
   \tilde{d}P_0 &= (d\lambda_0)Q +\lambda_0 dQ + [C_d,P_0] \\
   &= k_d P_0 +\lambda_0 dQ + \lambda_0 [C_d,Q]
\end{align*}
with
\begin{equation}
    k_d:= \frac{{d}\lambda_0}{\lambda_0} 
\end{equation}
Since $\lambda_0 \neq 0$ and in order to satisfy \eqref{eq:ConditionDerivationDirac}, the matrix $C_d$ must satisfy
\begin{equation*}
    dQ = [Q,C_d].
\end{equation*}
Using $Q^2 = I_4$ we find $dQ = - Q(dQ)Q$ and thus
\begin{align*}
    dQ = QC_d -C_d Q = \tfrac{1}{2}(dQ- Q(dQ)Q).
\end{align*}
This clearly implies
\begin{equation*}
    C_d = \tfrac{1}{2}QdQ.
\end{equation*}
We finally obtain
\begin{equation}
    \tilde{d} = d + [\tfrac{1}{2}QdQ,\cdot].
\end{equation}
Moreover,
\begin{equation}
    \|d U\| \leq \|\tilde{d}U\| + |Q||dQ| \|U\|.
\end{equation}
By the energy estimate for $\|U\|$ and the fact that $Q$ and its derivatives are bounded (this will involve derivatives of the fields) we obtain that $\|dU\|$ remains bounded if $\|\tilde{d}U\|$ is. The latter is obtained from the energy estimate for $\tilde{d}U$ that does not involve any oscillations in $\varepsilon$.
\end{remark}
\begin{proof}[Proof of Proposition \ref{lemma:GronwallLagrangeWigner}]
Following Remark \ref{Remark:Preconditioning}, we consider the covariant derivatives
\begin{align}
    \begin{cases}
        \tilde{\partial}_t := \partial_t + [\tfrac{1}{2}Q \partial_t Q,\cdot], \\
        \tilde{\partial}_{q_j} := \partial_{q_j} + [\tfrac{1}{2}Q \partial_{q_j} Q,\cdot], \\
        \tilde{\partial}_{p_j} := \partial_{p_j} + [\tfrac{1}{2}Q \partial_{p_j} Q,\cdot],
    \end{cases}
\end{align}
where 
\begin{equation*} 
Q := \Pi_+-\Pi_- = \frac{P_0}{\langle p-A \rangle}, \quad P_0 := \alpha \cdot(p-A) + \beta.
\end{equation*}
Throughout the proof, 
\begin{equation*}
    d \quad \text{designates either} \quad \partial_t, \, \partial_{q_j}, \, \text{ or } \, \partial_{p_j}
\end{equation*} and 
\begin{equation*}
    \tilde{d} \quad \text{designates either} \quad \tilde{\partial}_t, \, \tilde{\partial}_{q_j}, \, \text{ or } \, \tilde{\partial}_{p_j}.
\end{equation*}
Next, define
\begin{equation}
    \label{eq:OperatorB}B W_{\hbar} := \tfrac{1}{2}(\{P,W_{\hbar}\}-\{W_{\hbar},P\}) +r_{\hbar}[W_{\hbar}],
\end{equation}
where $r_{\hbar}$ is given by
\begin{align}
r_{\hbar}[W_{\hbar}] &= \hbar \Theta[A_k][\alpha_k,W^{\hbar}] +(-\nabla_x A_k\cdot \nabla_{\xi}-\theta[A_k])\tfrac{1}{2}[\alpha_k, W^{\hbar}]_+   +(-\nabla V \cdot \nabla_{\xi} -\theta [V])W^{\hbar} \\
&=: r_{\hbar}^1+r_{\hbar}^2+r_{\hbar}^3, \nonumber
\end{align}
cf. Section \ref{sec:Remainder}, equation \eqref{eq:Remainder}.
Then, following \eqref{eq:EnergyEstimateCovariant} with 
\begin{equation*}
    U = W_{\hbar}, \quad k_d = \frac{d \lambda_0}{\lambda_0} \quad \text{and} \quad d = \partial_t,\partial_{q_j},\partial_{p_j},
\end{equation*}
the energy estimates for $\tilde{\partial}_t W_{\hbar}$,  $\tilde{\partial}_{q_j} W_{\hbar}$ and $\tilde{\partial}_{p_j} W_{\hbar}$ read
 \begin{align*}
        \tfrac{1}{2}\partial_t \|\tilde{\partial}_t W_{\hbar}\|_{L^2_{q,p}}^2
        &\leq  
        \iint \tr(\tilde{\partial}_t W_{\hbar}\left[ \tfrac{\partial_t \lambda_0}{\lambda_0}\left(\tilde{\partial}_t W_{\hbar} - [\tfrac{1}{2}Q\partial_t Q,W_{\hbar}] - B W_{\hbar}\right) + [\tilde{\partial}_t ,B-\partial_t]W_{\hbar} \right]) dq dp \\
        &\leq \|\tfrac{\partial_t \lambda_0}{\lambda_0}\|_{L^{\infty}} \|\tilde{\partial}_t W_{\hbar}\|_{L^2_{q,p}}^2 \\
        &\qquad + \|\tilde{\partial}_t W_{\hbar}\|_{L^2_{q,p}}\left[ \|\tfrac{\partial_t \lambda_0}{\lambda_0}\|_{L^{\infty}}\left(\|Q\partial_t Q\|_{L^{\infty}} \|W_{\hbar}\|_{L^2_{q,p}}+ \|B W_{\hbar}\|_{L^2_{q,p}}\right) + \|[\tilde{\partial}_t ,B-\partial_t]W_{\hbar}\|_{L^2_{q,p}}\right],
    \end{align*}
while
 \begin{align*}
        \tfrac{1}{2}\partial_t \|\tilde{\partial}_{q_j} W_{\hbar}\|_{L^2_{q,p}}^2
        &\leq 
        \iint \tr(\tilde{\partial}_{q_j} W_{\hbar}\left[ \tfrac{\partial_{q_j} \lambda_0}{\lambda_0}\left(\tilde{\partial}_t W_{\hbar} - [\tfrac{1}{2}Q\partial_t Q,W_{\hbar}] - B W_{\hbar}\right) + [\tilde{\partial}_{q_j} ,B-\partial_t]W_{\hbar}\right]) dq dp \\
        &\leq \|\tilde{\partial}_{q_j} W_{\hbar}\|_{L^2_{q,p}} \Big[ \|\tfrac{\partial_{q_j} \lambda_0}{\lambda_0}\|_{L^{\infty}}\left(   \|\tilde{\partial}_t W_{\hbar}\|_{L^2_{q,p}}+ \|Q\partial_t Q\|_{L^{\infty}}\|W_{\hbar}\|_{L^2_{q,p}} + \|B W_{\hbar}\|_{L^2_{q,p}}\right) \\ &\qquad \qquad \qquad + \|[\tilde{\partial}_{q_j} ,B-\partial_t]W_{\hbar}\|_{L^2_{q,p}}\Big],
    \end{align*}
and
 \begin{align*}
        \tfrac{1}{2}\partial_t \|\tilde{\partial}_{p_j} W_{\hbar}\|_{L^2_{q,p}}^2
        &\leq 
        \iint \tr(\tilde{\partial}_{p_j} W_{\hbar}\left[ \tfrac{\partial_{p_j} \lambda_0}{\lambda_0}\left(\tilde{\partial}_t W_{\hbar} - [\tfrac{1}{2}Q\partial_t Q,W_{\hbar}] - B W_{\hbar}\right) + [\tilde{\partial}_{p_j} ,B-\partial_t]W_{\hbar}\right]) dq dp \\
        &\leq \|\tilde{\partial}_{p_j} W_{\hbar}\|_{L^2_{q,p}} \Big[ \|\tfrac{\partial_{p_j} \lambda_0}{\lambda_0}\|_{L^{\infty}}\left(   \|\tilde{\partial}_t W_{\hbar}\|_{L^2_{q,p}}+ \|Q\partial_t Q\|_{L^{\infty}}\|W_{\hbar}\|_{L^2_{q,p}} + \|B W_{\hbar}\|_{L^2_{q,p}}\right) \\ &\qquad \qquad \qquad + \|[\tilde{\partial}_{p_j} ,B-\partial_t]W_{\hbar}\|_{L^2_{q,p}}\Big].
    \end{align*}
    We now estimate the terms on the right-hand sides of the energy estimates one-by-one.\\
    
    \underline{The terms $\frac{d\lambda_0}{\lambda_0}$:}
    From the definition \eqref{eq:Lambda0} of $\lambda_0$ we observe that
    \begin{equation*}
        \|\tfrac{d\lambda_0}{\lambda_0} \|_{L^{\infty}_{t,q}} \leq C\|dA\|_{L^{\infty}_{t,q}} \quad \text{for }d = \partial_t,\partial_{q_j} \quad \text{and} \quad \|\tfrac{d\lambda_0}{\lambda_0} \|_{L^{\infty}_{t,q}} \leq C \quad \text{for }d = \partial_{p_j}.
    \end{equation*} \\

    \underline{The terms $QdQ$:}
    Moreover,
    \begin{equation*}
        Q d Q = \frac{P_0 dP_0}{\lambda_0^2} + \frac{P_0^2}{\lambda_0}d\left(\frac{1}{\lambda_0}\right) = \frac{P_0 dP_0}{\lambda_0^2} - \frac{d\lambda_0}{\lambda_0}.
    \end{equation*}
    Therefore,
    \begin{align*}
        \|Q d Q\|_{L^{\infty}_{t,q}} &\leq C (1+\|dA\|_{L^{\infty}_{t,q}}) \\
        &\leq C(1+F(\|A\|_{W^{1,\infty}_{t,q}})) \qquad \text{ for }d = \partial_t,\partial_{q_j}
    \end{align*}
    and
    \begin{equation*}
        \|Q d Q\|_{L^{\infty}_{t,q}} \leq C
        \qquad \text{ for }d = \partial_{p_j}.
    \end{equation*} \\

    \underline{The term $B W_{\hbar}$:}
Recall the definition \eqref{eq:OperatorB} of  $B W_{\hbar}$.
We have the bound on the Poisson bracket
\begin{align}
    \|\tfrac{1}{2}\{P,W_{\hbar}\}-\{W_{\hbar},P\}\|_{L^2_{q,p}} &\leq \|\nabla_{q,p} P\|_{L^{\infty}_{q,p}} \|\nabla_{q,p}W_{\hbar}\|_{L^2_{q,p}} \nonumber \\ &\leq \left(1+\|\nabla_{q}A\|_{L^{\infty}_{q}} + \|\nabla_{q}V\|_{L^{\infty}_{q}}\right)  \|\nabla_{q,p}W_{\hbar}\|_{L^2_{q,p}} \nonumber \\ &\leq \left(1+F(\|A\|_{W^{1,\infty}_{q}}, \|V\|_{W^{1,\infty}_{q}})\right)  \|W_{\hbar}\|_{H^1_{q,p}} \label{eq:EstimatePoissonL2}
\end{align}
Next we estimate the remainder $r_{\hbar}$ in $L^2$. \\

\underline{The term $r_{\hbar}^1$:} We have  \begin{align*}
    (\Theta[F]f)(x,\xi) &= -\frac{1}{(2\pi)^d}\int e^{i(\xi-\eta)\cdot y} \Delta[F](x,y)|y|^2 f(x,\eta)d y d\eta,
\end{align*}
where
\begin{equation*}
     \Delta[F](x,y) = i\frac{F\left(x+\frac{\hbar y}{2}\right)-2F(x)+F\left(x-\frac{\hbar y}{2}\right)}{2\hbar^2|y|^2}. 
\end{equation*}
Now,
\begin{align*}
     \frac{F(x+\frac{\hbar y}{2})-2F(x)+F(x-\frac{\hbar y}{2})}{\hbar|y|} = \frac{F(x+\frac{\hbar y}{2})-F(x)}{\hbar|y|} + \frac{F(x-\frac{\hbar y}{2})-F(x)}{\hbar|y|}. 
\end{align*}
Thus,
\begin{align*}
    \hbar(\Theta[F]f)(x,\xi) &= \frac{1}{(2\pi)^d}\int e^{i(\xi-\eta)\cdot y} \frac{F(x+\frac{\hbar y}{2})-F(x)}{2i\hbar|y|}\frac{y\cdot y}{|y|} f(x,\eta)d y d\eta \\
    &\qquad +\frac{1}{(2\pi)^d}\int e^{i(\xi-\eta)\cdot y} \frac{F(x-\frac{\hbar y}{2})-F(x)}{2i\hbar|y|}\frac{y\cdot y}{|y|} f(x,\eta)d y d\eta \\
    &= \frac{1}{(2\pi)^d}\int e^{i(\xi-\eta)\cdot y} \frac{F(x+\frac{\hbar y}{2})-F(x)}{2i\hbar|y|}\frac{y}{|y|}\cdot (i\nabla_{\eta} f(x,\eta)d y d\eta \\
    &\qquad +\frac{1}{(2\pi)^d}\int e^{i(\xi-\eta)\cdot y} \frac{F(x-\frac{\hbar y}{2})-F(x)}{2i\hbar|y|}\frac{y}{|y|}\cdot (i\nabla_{\eta} f(x,\eta)d y d\eta.
\end{align*}
Now use the inequality
\begin{equation}
\label{eq:ConvolutionHölder}
    \|f\ast g\|_{L^2} \leq \|\mathcal{F}f\|_{L^\infty}\|g\|_{L^2},
\end{equation}
and 
\begin{equation}
    \left|\frac{F(x\pm\frac{\hbar y}{2})-F(x)}{2i\hbar|y|}\right| \leq \mathrm{Lip}(F),
\end{equation}
to conclude that
\begin{equation}
    \|r_{\hbar}^1\|_{L^2} \leq \left\|\hbar \Theta[A_k][\alpha_k,W_{\hbar}]\right\|_{L^2} \lesssim 2 \mathrm{Lip}(A) \|\nabla_{\xi}W_{\hbar}\|_{L^2}.
\end{equation}

\underline{The term $r_2^{\hbar}$:} We have
$$r_2^{\hbar} = \frac{1}{2}(-\nabla_x A_k\cdot \nabla_{\xi}-\theta[A_k])[\alpha_k,W_{\hbar}]_+.$$
Again, since $r_{\hbar}$ needs to be only $O(1)_{L^2}$ we do not have to extract any power of $\hbar$ here. Therefore we can simply estimate
\begin{align*}
    \|r_2^{\hbar}\|_{L^2} &\leq \tfrac{1}{2}\left(\|\nabla_x A_k\cdot \nabla_{\xi}W_{\hbar}\|_{L^2} + \|\theta[A_k][\alpha_k,W_{\hbar}]_+\|_{L^2}\right) \\ &\leq
    \tfrac{1}{2}\left(\|\nabla A\|_{L^{\infty}} + \mathrm{Lip}(A)\right)\|\nabla_{\xi}W_{\hbar}\|_{L^2}.
\end{align*}
In the last step we used \eqref{eq:ConvolutionHölder}. \\

\underline{The term $r_3^{\hbar}$:} This term is completely analogous to $r_2^{\hbar}$. We obtain
\begin{align*}
    \|r_3^{\hbar}\|_{L^2} &\leq
    \tfrac{1}{2}\left(\|\nabla V\|_{L^{\infty}} + \mathrm{Lip}(V)\right)\|\nabla_{\xi}W_{\hbar}\|_{L^2}.
\end{align*}

Therefore,
\begin{equation}
\label{eq:EstimateRemainderl2}
    \|r_{\hbar}\|_{L^2} \leq C\left(\mathrm{Lip}(A)+\mathrm{Lip}(V)\right)\|\nabla_{\xi}W_{\hbar}\|_{L^2}.
\end{equation} 

Combining \eqref{eq:EstimatePoissonL2} and \eqref{eq:EstimateRemainderl2} yields
\begin{align}
    \|B W_{\hbar}\|_{L^2_{q,p}} \leq C (1+\|\nabla A\|_{L^{\infty}_q}+\|\nabla V\|_{L^{\infty}_q})\|\nabla_{q,p}W_{\hbar}\|_{L^2_{q,p}} \nonumber \\
    \leq C (1+F( \|A\|_{W^{1,\infty}_q},\| V\|_{W^{1,\infty}_q}))\|W_{\hbar}\|_{H^1_{q,p}} \label{eq:EstimateBWL2}
\end{align}

\underline{The terms $[\tilde{d},B-\partial_t]W_{\hbar}$:}
We have
\begin{align*}
    [\tilde{d},B-\partial_t]W_{\hbar} &= [{d}+[\tfrac{1}{2}QdQ,\cdot],B-\partial_t]W_{\hbar} \\
    &= [d,B]W_{\hbar} + [[\tfrac{1}{2}QdQ,\cdot],B]W_{\hbar} - [[\tfrac{1}{2}QdQ,\cdot],\partial_t]W_{\hbar} \\
    &= B_dW_{\hbar} + [[\tfrac{1}{2}QdQ,\cdot],B]W_{\hbar} - [[\tfrac{1}{2}QdQ,\cdot],\partial_t]W_{\hbar} \\
    &= B_dW_{\hbar} + [\tfrac{1}{2}QdQ,BW_{\hbar}] - B[\tfrac{1}{2}QdQ,W_{\hbar}] \\ &\qquad \qquad - [\tfrac{1}{2}QdQ,\partial_t W_{\hbar}] + \partial_t[\tfrac{1}{2}QdQ,W_{\hbar}] \\
    &= B_dW_{\hbar} + [\tfrac{1}{2}QdQ,BW_{\hbar}] - B[\tfrac{1}{2}QdQ,W_{\hbar}] \\ &\qquad \qquad + [\tfrac{1}{2}\partial_t(QdQ),W_{\hbar}].
\end{align*}
Here,
\begin{equation}
    B_{{d}}W_{\hbar} = \tfrac{1}{2}\left(\{{d}P,W_{\hbar}\}-\{W_{\hbar},{d}P\}\right) + {r}^{\hbar}_d,
\end{equation}
where ${r}^{\hbar}_d$ is defined as
\begin{align}
{r}^{\hbar}_d &:= \hbar \Theta[{d}A_k][\alpha_k,W^{\hbar}] +(-\nabla_x {d}A_k\cdot \nabla_{\xi}-\theta[{d}A_k])\tfrac{1}{2}[\alpha_k, W^{\hbar}]_+   +(-\nabla {d}V \cdot \nabla_{\xi} -\theta [{d}V])W^{\hbar}.
\end{align}

Replacing $B$ by $B_{{\partial}_t}$, $B_{{\partial}_{q_j}}$ and $B_{{\partial}_{p_j}}$ thus only adds derivatives in $t$ and $q$ of the potentials. 
There are no additional derivatives of $W_{\hbar}$ involved. Therefore, similar to \eqref{eq:EstimateBWL2} with $A,V$ replaced by ${d}A,{d}V$ we have
\begin{align}
    \|B_{{d}}W_{\hbar}\|_{L^2_{q,p}} &\leq C\left(1+\|\nabla {d}A\|_{L^{\infty}_{q}} + \|\nabla {d}V\|_{L^{\infty}_{q}}\right) \|\nabla_{q,p}W_{\hbar}\|_{L^2_{q,p}} \nonumber \\
    &\leq C(1+F(\|A\|_{W^{2,\infty}_{q}}))\|W^{\hbar}\|_{H^1_{q,p}}
\end{align}
For the term $[\tfrac{1}{2}QdQ,BW_{\hbar}]$ we have 
\begin{align*}
    \|[\tfrac{1}{2}QdQ,BW_{\hbar}]\|_{L^2_{q,p}} &\leq \|QdQ\|_{L^{\infty}_{q,p}}\|BW_{\hbar}\|_{L^2_{q,p}} \\
    &\leq C (1+F(\|A\|_{W^{1,\infty}_{q}}))\|W^{\hbar}\|_{H^1_{q,p}}
\end{align*}
For the term $B[\tfrac{1}{2}QdQ,W_{\hbar}]$ we have
\begin{align*}
    \|B[\tfrac{1}{2}QdQ,W_{\hbar}]\|_{L^2_{q,p}} &\leq \|\tfrac{1}{2}(\{P,[\tfrac{1}{2}QdQ,W_{\hbar}]\}-\{[\tfrac{1}{2}QdQ,W_{\hbar}],P\})\|_{L^2_{q,p}} + \|r_{\hbar}[[\tfrac{1}{2}QdQ,W_{\hbar}]]\|_{L^2_{q,p}}\\
    &\leq C (1+F(\|A\|_{W^{2,\infty}_{t,q}},\|V\|_{W^{2,\infty}_{q}}))\|W^{\hbar}\|_{H^1_{q,p}}
\end{align*}
Finally,
\begin{align*}
    \|\partial_t Q dQ + Q\partial_t dQ\|_{L^{\infty}_{t,q}} &\leq C\left(1+ \|\partial_t Q\|_{L^{\infty}_{t,q}} \|dQ \|_{L^{\infty}_{t,q}} + \|Q\|_{L^{\infty}_{t,q}} \|\partial_t d Q\|_{L^{\infty}_{t,q}}\right) \\
    &\leq C \left(1+\|\partial_t A\|_{L^{\infty}_{t,q}}\|d A\|_{L^{\infty}_{t,q}} + \| A\|_{L^{\infty}_{t,q}} \|\partial_t dA\|_{L^{\infty}_{t,q}}\right) \\
    &= 1+F(\|A\|_{W^{2,\infty}_{t,q}}). 
\end{align*}
Thus, the energy estimates become
\begin{align*}
    \tfrac{1}{2}\partial_t\|\tilde{d}W_{\hbar}\|^2_{L^2_{q,p}} \leq C(1+F(\|A\|_{W^{2,\infty}_{t,q,p}},\|V\|_{W^{2,\infty}_{t,q,p}})) \|\tilde{d}W_{\hbar}\|_{L^2_{q,p}}^2
\end{align*}
and so
\begin{equation*}
    \|\tilde{d}W_{\hbar}(t)\|^2_{L^2_{q,p}} \leq \|\tilde{d}W_{\hbar}^{\mathrm{in}}\|^2_{L^2_{q,p}} \exp(Ct(1+F(\|A\|_{W^{2,\infty}_{t,q,p}},\|V\|_{W^{2,\infty}_{t,q,p}}))). 
\end{equation*}
This implies
\begin{equation*}
    \|dW_{\hbar}(t)\|_{L^2_{q,p}} \leq C \exp(Ct(1+F(\|A\|_{W^{2,\infty}_{t,q,p}},\|V\|_{W^{2,\infty}_{t,q,p}}))). 
\end{equation*}
since $\|dW_{\hbar}^{\mathrm{in}}\|^2_{L^2_{q,p}}\leq C$ by assumption \eqref{eq:AssumptionInitialDerivativesWigner}. This proves the propagation of derivatives of $W_{\hbar}$. 

For the bound on the commutator we use the Wigner equation and obtain
\begin{align*}
\tfrac{1}{i}[P,W_{\hbar}] = \hbar\left(\partial_t W_{\hbar} - \tfrac{1}{2}\left( \{P,W_{\hbar}\} - \{W_{\hbar},P\}\right) - r_{\hbar}\right),
\end{align*}
From this and the propagation of the derivatives we obtain the desired estimate.
\end{proof}
\end{proposition}


Now we are able to prove Lemma \ref{thm:EstimateLagrangeMultiplier}.
\begin{proof}[Proof of Lemma \ref{thm:EstimateLagrangeMultiplier}]
Recall $H_1$ and $H_2$, cf. \eqref{eq:LagrangeMultiplierH1} and \eqref{eq:LagrangeMultiplierH2}. 

\underline{The term $H_1$:} For $H_1$ we observe
\begin{align}
\label{eq:EstimateH1}
\begin{split}
    |H_1| &= \left|\iint \tr(G_{\hbar}\ast\tfrac{1}{i}[P,W_{\hbar}](\mathcal{Z}+\mathcal{V}))dqdp\right| \\
    &\leq \left\|G_{\hbar}\ast\tfrac{1}{i}[P,W_{\hbar}]\right\|_{L^2} \|\mathcal{Z}+\mathcal{V}\|_{L^2} \\
    &\leq \|G_{\hbar}\|_{L^1}\left\|[P,W_{\hbar}]\right\|_{L^2} \left(\|\mathcal{Z}\|_{L^2}+\|\mathcal{V}\|_{L^2}\right)  \\&\lesssim \left\|[P,W_{\hbar}]\right\|_{L^2} \left(\|\mathcal{Z}\|_{L^2}+\|\mathcal{V}\|_{L^2}\right). 
    \end{split}
\end{align}
\underline{The term $H_2$:} For $H_2$ we observe
\begin{align*}
    H_2 &= - \iint \tr((\tfrac{1}{i}[P,G_{\hbar}\ast W_{\hbar}]-G_{\hbar}\ast\tfrac{1}{i}[P,W_{\hbar}])(\mathcal{Z}+\mathcal{V}))dqdp \\ &= - \iiiint \tr((\tfrac{1}{i}[P(q,p)-P(q-q',p-p'), W_{\hbar}(q-q',p-p')]G_{\hbar}(q',p')(\mathcal{Z}+\mathcal{V})(q,p)))dqdpdq'dp'.
\end{align*}
Now
\begin{equation*}
  P(q,p)-P(q-q',p-p') = \alpha\cdot p' - \alpha \cdot (\nabla A(q) \cdot q')) + \nabla V(q) \cdot q' + O((\|\nabla^2A\|_{L^{\infty}}+\|\nabla^2V\|_{L^{\infty}})|q'|^2)  
\end{equation*}
In the commutator, the contribution of $V$ vanishes. Now let
\begin{equation*}
    a(q,p) := (\alpha,\alpha \cdot \nabla A(q)), \quad |a| \leq 1+\|\nabla A\|_{L^{\infty}}.
\end{equation*}
and denote $z =(q,p)$. Then
\begin{align*}
    H_2 &= -\iint \tr(\tfrac{1}{i}[a(z)\cdot z',W_{\hbar}(z-z')]G_{\hbar}(z')(\mathcal{Z}+\mathcal{V})(z))dzdz' \\ &\qquad \qquad  + \iint \tr(O(\|\nabla^2 A\|_{L^{\infty}}|z'|^2) W_{\hbar}(z-z') G_{\hbar}(z')(\mathcal{Z}+\mathcal{V})(z))dzdz' \\
    &= -\int \tr(\tfrac{1}{i}[a(z),W_{\hbar}]\ast_{z} (-\tfrac{\hbar}{2} \nabla G_{\hbar}(z'))(\mathcal{Z}+\mathcal{V})(z))dz \\ &\qquad \qquad  + \iint \tr(O(\|\nabla^2 A\|_{L^{\infty}}|z'|^2) W_{\hbar}(z-z') G_{\hbar}(z')(\mathcal{Z}+\mathcal{V})(z))dzdz' \\
    &= \frac{\hbar}{2}\int \tr(\tfrac{1}{i}[a(z),\nabla W_{\hbar}]\ast_{z} G_{\hbar}(z')(\mathcal{Z}+\mathcal{V})(z))dz \\ &\qquad \qquad  + \iint \tr(O(\|\nabla^2 A\|_{L^{\infty}}|z'|^2) W_{\hbar}(z-z') G_{\hbar}(z')(\mathcal{Z}+\mathcal{V})(z))dzdz'.
\end{align*}
Now we can estimate $H_2$ as follows
\begin{align}
    |H_2(t)|&\leq \hbar \left[1+\|\nabla A\|_{L^{\infty}}\|\nabla_{q,p}W_{\hbar}\|_{L^2_{q,p}} + \|\nabla^2 A\|_{L^{\infty}}\|W_{\hbar}\|_{L^2_{q,p}} \right]\left\|(\mathcal{Z}(t)+\mathcal{V}(t))\right\|_{L^2_{q,p}}.\label{eq:EstimateH2}
\end{align}
Combining the estimates for $H_1$ and $H_2$, i.e. \eqref{eq:EstimateH1} and \eqref{eq:EstimateH2} we obtain
\begin{align*}
    \left| \Tr(R(t)\mathrm{Op}^T_{\hbar}\left[\tfrac{1}{i}[P(t),\mathcal{Z}(t)+\mathcal{V}(t)]\right])\right| &\leq |H_1(t)|+|H_2(t)| \\ &\leq C\hbar (1+F(\|A\|_{W^{1,\infty}})\|W_{\hbar}\|_{H^1_{q,p}} \left\|(\mathcal{Z}(t)+\mathcal{V}(t))\right\|_{L^2} \\
    & \qquad + \left\|[P,W_{\hbar}]\right\|_{L^2}\left\|(\mathcal{Z}(t)+\mathcal{V}(t))\right\|_{L^2} .
\end{align*}
By Lemma \ref{eq:PropertiesEstimatesZdelta} below, we can control $\|\mathcal{Z}(t)\|_{L^2}$ as follows:
\begin{equation}
    \|\mathcal{Z}(t)\|_{L^2} \leq C\left(\|\tfrac{1}{\sigma^2} \partial_t P(t)\|_{L^6} + \|\tfrac{1}{\sigma}P(t)\|_{W^{1,\infty}}\right) \|\Phi(t)\|_{H^1_{q,p}}.
\end{equation}
By the formula \eqref{eq:HusimiExtraTermFinal} for the Husimi extra term we have
\begin{equation}
    \|\mathcal{V}(t)\|_{L^2} \leq C \left(1+\|\nabla A(t)\|_{L^{\infty}}^2 + \|\nabla^2 A(t)\|_{L^{\infty}}\right) \|\Phi(t)\|_{H^2_{q,p}}.
\end{equation}
Then we have the combined estimate
\begin{equation*}
    \|\mathcal{Z}(t)\|_{L^2}+\|\mathcal{V}(t)\|_{L^2} \leq C \left(1+ F\left(\| A(t)\|_{W^{2,\infty}_{t,q}},\| V(t)\|_{W^{1,\infty}_{t,q}}\right)\right)\|\Phi(t)\|_{H^2},
\end{equation*}
where $F$ is a continuous function of vanishing at the origin. 
Therefore we obtain, since $\|W_{\hbar}(t)\|_{L^2}\leq C$ by charge conservation and by the propagation estimate for the derivatives \eqref{eq:SpaceTimeGradientWigner} and the commutator \eqref{eq:PropagationCommutatorLemma},
\begin{align*}
    \left| \Tr(R(t)\mathrm{Op}^T_{\hbar}\left[\tfrac{1}{i}[P(t),\mathcal{Z}(t)+\mathcal{V}(t)]\right])\right| \leq C\left[\hbar (1+F)\exp(Ct(1+F))\|\Phi(t)\|_{H^2}\right]
\end{align*}
where
\begin{equation}
    F \equiv F(\|A\|_{W^{2,\infty}_{t,q}},\|V\|_{W^{2,\infty}_{t,q}}).
\end{equation}
This concludes the proof.
\end{proof}

\subsection{Gronwall estimate} Now we finalize the proof of the Gronwall estimate \eqref{eq:MainGronwallEstimate} from Theorem \ref{thm:main}. Recall the estimate for the pseudo-distance $\mathfrak{D}(t)$ \eqref{eq:GrönwallDistance}
\begin{align*}
    \mathfrak{D}(t) \leq \mathfrak{D}^{\mathrm{in}} &+ \sup_{\mathcal{A}} \int_0^t \left|\Tr(R(s)\left(\tfrac{i}{\hbar}[H(s),\mathrm{Op}^T_{\hbar}[\Phi(s)]]+\mathrm{Op}^T_{\hbar}\left[\tfrac{1}{2}\left(\{P(s),\Phi(s)\} - \{\Phi(s),P(s)\}\right)\right]\right))\right| ds \\ &\qquad + \sup_{\mathcal{A}}\int_0^t\left|\Tr(R(s)\mathrm{Op}^T_{\hbar}[\tfrac{1}{i}[P(s),\mathcal{Z}(s)]) \right|ds .
\end{align*}
Using Lemma \ref{lemma:ExpressionDistance} we obtain 
\begin{align*}
    \mathfrak{D}(t) &\leq \mathfrak{D}^{\mathrm{in}} +\sup_{\mathcal{A}}\int_0^t \left| \Tr(R(t)\mathrm{Op}^T_{\hbar}\left[\tfrac{1}{i}[P(t),\mathcal{Z}(t)+\mathcal{V}(t)]\right])\right|ds \\
 &\qquad \qquad  +  \sup_{\mathcal{A}}\hbar \int_0^t\left|\Tr({R}(s) \mathrm{Op}^T_{\hbar}\left[  \partial_{q_i}\partial_{q_j}\left(R^{ij}_k(s,x,q)\left[\alpha_k,\Phi(s)\right]\right)\right])\right|ds \\ &\qquad \qquad +\sup_{\mathcal{A}}\hbar \int_0^t\left|\Tr(R(s)\Delta[A_k(s)] \mathrm{Op}^T_{\hbar}\left[[\alpha_k,\Delta_p\Phi(s)]\right])\right|ds \\
    &\qquad \qquad + \sup_{\mathcal{A}}\int_0^t\left|\Tr(R(s)\left(\delta[A_k(s)]_j\mathrm{Op}^T_{\hbar}\left[\tfrac{1}{2}[\alpha_k,\partial_{p_j}\Phi(s)]_+\right]-\mathrm{Op}^T_{\hbar}\left[\partial_{q_j}A_k(s,q)\tfrac{1}{2}[\alpha_k,\partial_{p_j}\Phi(s)]_+\right]\right)) \right|ds\\
    &\qquad \qquad + \sup_{\mathcal{A}}\int_0^t \left|\Tr(R(s)\left(\delta[V(s)]_j\mathrm{Op}^T_{\hbar}\left[\partial_{p_j}\Phi(s)\right]-\mathrm{Op}^T_{\hbar}\left[\partial_{q_j}V(s,q)\partial_{p_j}\Phi(s)\right]\right))\right|ds. \\ & \qquad \qquad 
    +O(\hbar)\left(\|\nabla^4 A(t)\|_{L^{\infty}}\|\Phi(t)\|_{L^2} + \|\nabla^3 A(t)\|_{L^{\infty}}\|\nabla \Phi(t)\|_{L^2}\right).
  \end{align*}
Using Lemmas \ref{lemma:GronwallFirstOrder}, \ref{lemma:GronwallSecondOrder} and \ref{thm:EstimateLagrangeMultiplier} we then obtain the following estimate for $\mathfrak{D}(t)$,
\begin{align*} 
    \mathfrak{D}(t) & \leq \mathfrak{D}^{\mathrm{in}} +\sup_{\mathcal{A}} C\hbar (1+F) \int_0^t\left[e^{C(1+F)s}+ \|A(s)\|_{W^{4,\infty}_{t,q}}\right]\|\Phi(s)\|_{H^2_{q,p}} ds \\
    &\qquad \qquad + \sup_{\mathcal{A}} C\sqrt{\hbar} F\int_0^t  \|\Phi(s)\|_{H^1_{q,p}}ds 
\end{align*}
where $F\equiv F(\|A\|_{W^{2,\infty}_{t,q}},\|V\|_{W^{2,\infty}_{t,q}})$.
At this point we invoke the growth bound on the backward solution $\Phi$, i.e. \eqref{eq:GronwallTestBackward}:
\begin{equation*}
    \|\Phi(t)\|_{H^k}^2 \leq \|\Phi^{\mathrm{fin}}\|_{H^k}^2 e^{C_{k} (T-t)}, \quad \text{for } t\in [0,T],
\end{equation*}
where
\begin{equation*}
    C_{k} = 1+F(\|\nabla A\|_{W^{k,\infty}_{q}},\|\partial_t A\|_{W^{k-1,\infty}_{q}},\|\nabla V\|_{W^{k,\infty}_{q}}),
\end{equation*}
and where $F$ is a continuous function vanishing at the origin. Now we define
\begin{align*}
    \begin{cases}
    F_2 := F(\|A\|_{W^{2,\infty}_{t,q}},\|V\|_{W^{2,\infty}_{t,q}}) \\
        F_4 := F(\|A\|_{W^{4,\infty}_{t,q}},\|V\|_{W^{3,\infty}_{t,q}})
    \end{cases}
\end{align*}
so that, by taking the supremum over $\|\Phi^{\mathrm{fin}}\|_{H^k} \leq 1$,
\begin{align*} 
    \mathfrak{D}(t) &\leq \mathfrak{D}^{\mathrm{in}} + C\hbar (1+F_4) \left[\int_0^te^{C(1+F_4)s}e^{C_2(T-s)}ds + \int_0^t e^{C_2(T-s)}ds\right] \\
    &\qquad \qquad +  C\sqrt{\hbar} F_2\int_0^t  e^{C_1(T-s)}ds \\
    &= \mathfrak{D}^{\mathrm{in}} + C\hbar (1+F_4)e^{C_2T}\left[\int_0^te^{(C(1+F_4)-C_2)s}ds + \int_0^t e^{-C_2s}ds\right] \\
    &\qquad \qquad +  C\sqrt{\hbar} F_2e^{C_1T}\int_0^t  e^{-C_1s}ds \\
    &= \mathfrak{D}^{\mathrm{in}} + C\hbar {(1+F_4)e^{C_2T}}\left[\frac{1}{C(1+F_4)-C_2}\left(e^{(C(1+F_4)-C_2)t}-1\right) +\frac{1}{C_2}\left(1-e^{-C_2t}\right)\right] \\
    &\qquad \qquad +  C\sqrt{\hbar} \frac{F_2e^{C_1T}}{C_1}\left(1-e^{-C_1t}\right)
\end{align*}
Note that if the rates $C(1+F_4)$ and $C_2$ are equal, the first term inside the square brackets reduces to $t$.

Now we introduce constants $\Xi_2$ depending on $T$, $F_2$ and $C_1$ and $\Xi_4$ depending on $T$, $F_4$ and $C_2$ such that
\begin{align*}
    \mathfrak{D}(t) &\leq \mathfrak{D}^{\mathrm{in}} + \sqrt{\hbar}\, \Xi_2 e^{\Xi_2 t} + \hbar\, \Xi_4 e^{\Xi_4 t}, \quad t\in [0,T]. 
\end{align*}

This completes the proof of \eqref{eq:MainGronwallEstimate}. 

The $L^2$ estimate for the commutator $[P,W_{\hbar}]$ was already proved in Proposition \ref{lemma:GronwallLagrangeWigner}. Therefore the proof of Theorem \ref{thm:main} is complete.
\section{Wellposedness of the matrix Liouville equation} 
\label{sec:LWPMatrixLiouville}
In this section we prove Theorem \ref{thm:LWPLiouville}.
Recall the matrix Liouville equation
    \begin{align}
    \begin{cases}
\partial_t \Phi = \tfrac{1}{i}[P,\mathcal{Z}]+ \tfrac{1}{2} \left( \{P, \Phi \} - \{ \Phi, P \} \right), \qquad \Phi|_{t=0}  = \Phi^{\mathrm{in}}, \\
[P,\Phi] = 0,
\end{cases}
\end{align}
where
\begin{align}
    P(t,q,p) = \alpha \cdot(p-A(t,q))+ \beta +V(t,q).
\end{align}
We will use parabolic regularization (also known as vanishing viscosity) by addition of a parabolic term $\delta \Delta \Phi$ with $\delta>0$.

Formally,
\begin{equation}
    \partial_t \iint \tr(\Phi^{\ast}\Phi)dqdp = \partial_t \|\Phi\|^2_{L^2(\mathbb{R}^d\times \mathbb{R}^d,\mathbb{M}_m(\mathbb{C}))} = \text{const.}
    \label{eq:L2estimatePhi}
\end{equation}
This can be seen as follows. Assume that $\Phi=\Phi^{\ast}$ is self-adjoint. Then
\begin{align*}
    \iint \tr(\Phi \partial_t \Phi) dqp &= \iint \tr(\Phi(\tfrac{1}{2}(\{P,\Phi\}-\{\Phi,P\})+\tfrac{1}{i}[P,\mathcal{Z}]))dqdp \\&= \tfrac{1}{2}\iint \tr(\Phi(\{P,\Phi\}-\{\Phi,P\})) dqdp + \int \tr(\tfrac{1}{i}[P,\Phi]\mathcal{Z}) \\
    &= \tfrac{1}{2}\iint \tr(\Phi(\{P,\Phi\}-\{\Phi,P\})) dqdp \\
    &=0, 
\end{align*}
since $[P,\Phi]=0$. The remaining term on the right hand side is just the matrix-valued analogue of the scalar Poisson bracket $\{H,f\}$ and we omit the classic calculation that $\iint f\{H,f\}dx\xi = 0$ for the matrix case. 

The propagation of self-adjointness of $\Phi$ is easy to see once one has uniqueness of the solution as the adjoint equation is the same as the original equation, i.e.
\begin{align*}
\begin{cases}
    \partial_t\Phi^{\ast} = \tfrac{1}{2}\left(\{P,\Phi^{\ast}\}-\{\Phi^{\ast},P\}\right) + \tfrac{1}{i}[P,\mathcal{Z}[\Phi^{\ast}]], \qquad \Phi^{\ast}|_{t=0} = \Phi_{\mathrm{in}}^{\ast} \\
    [P,\Phi^{\ast}] = 0,
\end{cases}
\end{align*}
since $P = P^*$ and $\mathcal{Z}[\Phi^{\ast}] = (\mathcal{Z}[\Phi])^*$. Take self-adjoint initial data $\Phi_{\mathrm{in}}^{\ast} = \Phi_{\mathrm{in}}$ and conclude by uniqueness.

\subsection{Wellposedness for the regularized equation}
The regularized matrix Liouville equation (here $\delta >0$ is a real number) is given by:
\begin{align}
    \begin{cases}
\partial_t \Phi^{\delta} = \tfrac{1}{i}[P,\mathcal{Z}^{\delta}] + \tfrac{1}{2} \left( \{P, \Phi^{\delta} \} - \{ \Phi^{\delta}, P \} \right) +  \delta\Delta \Phi^{\delta}, \qquad \Phi^{\delta}|_{t=0}  = \Phi^{\mathrm{in}}, \\
[P,\Phi^{\delta}] = 0,
\end{cases}
\label{eq:regularizedLiouville}
\end{align}
where
\begin{align}
    P(t,q,p) = \alpha \cdot(p-A(t,q))+ \beta +V(t,q).
\end{align}
Note that $\Delta=\Delta_{q,p}$ should be understood as the Laplacian on $\mathbb{R}^d_q\times \mathbb{R}^d_{p}$, i.e.
\begin{equation}
    \Delta := \sum_{j=1}^d (\partial_{q_k}^2+\partial_{p_k}^2).
\end{equation}

Here, $$\mathcal{L}^{\delta} := \tfrac{1}{i}[P,\mathcal{Z}^{\delta}]$$
is the Lagrange multiplier associated to the regularized constraint $[P,\Phi^{\delta}]=0$. Compared to the unregularized Lagrange multiplier $\mathcal{Z}$, the regularized Lagrange multiplier $\mathcal{Z}^{\delta}$ contains terms corresponding to $\delta\Delta\Phi^{\delta}$, cf. Lemma \ref{eq:PropertiesEstimatesZdelta}.

We have the following local wellposedness result for the regularized equation \eqref{eq:regularizedLiouville}.
\begin{proposition}
    Let $\delta>0$ and $\Phi^{\mathrm{in}} \in H^s(\mathbb{R}^d_q\times \mathbb{R}^d_p,\mathbb{M}_m(\mathbb{C}))$ with $s\geq 2$. Let $A,V \in H^s(\mathbb{R}_t\times \mathbb{R}^d_q)$ with $s$ large enough. Then there exists a $T^{\delta}$, depending on $\|\Phi^{\mathrm{in}}\|_{H^s}$ and $\delta$, such that the regularized matrix Liouville equation \eqref{eq:regularizedLiouville} is locally well-posed in the space $C([0,T^{\delta}),H^s)$ with initial data $\Phi^{\mathrm{in}}$.
\end{proposition}

The a priori estimate in $L^2$ \eqref{eq:L2estimatePhi} becomes
\begin{equation}
    \partial_t \|\Phi^{\delta}\|^2_{L^2(\mathbb{R}^d_q\times \mathbb{R}^d_p,\mathbb{M}_m(\mathbb{C}))} = -\delta \|\nabla \Phi^{\delta}\|^2_{L^2(\mathbb{R}^d_q\times \mathbb{R}^d_p,\mathbb{M}_m(\mathbb{C}))} \leq 0,
\end{equation}
i.e. $\|\Phi^{\delta}(t)\|^2_{L^2(\mathbb{R}^d_q\times \mathbb{R}^d_p,\mathbb{M}_m(\mathbb{C}))}$ is decreasing in time. We also need the decay properties of the heat semigroup, i.e.
\begin{equation}
    \|e^{\delta t\Delta}f\|_{W^{s,p}} \lesssim (\delta t)^{-\frac{s}{2}-\frac{d}{2}\left(\frac{1}{r}-\frac{1}{p}\right)}\|f\|_{L^r}, \qquad 1<r\leq 2,
    \label{eq:HeatDecayDimdSobolevp}
\end{equation}
cf. \cite{wang_harmonic_2011}. Here, we take $d=6$ and $p=2$. Then, 
\begin{equation}
    \|e^{\delta t\Delta}f\|_{H^{s}} \lesssim (\delta t)^{1-\frac{3}{r}}\|\langle D\rangle^{s-1} f\|_{L^r}.
    \label{eq:HeatDecayDim6Sobolev2}
\end{equation}
Here $\langle D \rangle^{s}$ stands for the Fourier multiplier with symbol $(1+|\xi|^2)^{s/2}$ where $\xi \in \mathbb{R}^d \times \mathbb{R}^d$. Note that in order for $t^{1-\frac{3}{r}}$ to be integrable, one needs to take $\frac{3}{2}<r$.

We shall also make frequent use of the fractional Leibniz rule or Kato-Ponce inequality \cite[Theorem 1]{grafakos_kato-ponce_2014}.

\begin{theorem}
    Let $\frac{1}{2} < r < \infty$, $1 < p_1, p_2, q_1, q_2 \leq \infty$ satisfy $\frac{1}{r} = \frac{1}{p_1} + \frac{1}{q_1} = \frac{1}{p_2} + \frac{1}{q_2}$. Given $s > \max\left(0, \frac{d}{r} - d\right)$ or $s \in 2\mathbb{N}$, there exists $C = C(d, s, r, p_1, q_1, p_2, q_2) < \infty$ such that for all $f, g \in \mathcal{S}(\mathbb{R}^d)$ we have
\begin{align}
  \| D^s (f g) \|_{L^r}
  & \leq C \left[ \| D^s f \|_{L^{p_1}} \| g \|_{L^{q_1}}
    + \| f \|_{L^{p_2}} \| D^s g \|_{L^{q_2}} \right], \label{eq:D_ineq} \\
  \| \langle D \rangle^s (f g) \|_{L^r}
  & \leq C \left[ \| f \|_{L^{p_1}} \| \langle D \rangle^s g \|_{L^{q_1}}
    + \| \langle D \rangle^s f \|_{L^{p_2}} \| g \|_{L^{q_2}} \right]. \label{eq:J_ineq}
\end{align}
\end{theorem}

We will show that the operator $G$, defined by
\begin{equation}
\label{eq:DuhamelOperator}
    G(\Phi^{\delta}(t)) = e^{\delta t\Delta}
    \Phi^{\mathrm{in}} + \int_0^t e^{\delta (t-s)\Delta} \left( \tfrac{1}{2} \left( \{P(s), \Phi^{\delta}(s) \} - \{ \Phi^{\delta}(s), P(s) \} \right) + \tfrac{1}{i}[P(s),\mathcal{Z}^{\delta}(s)]\right) ds,
\end{equation}
has a fixed point in 
\begin{equation}
\label{eq:SolutionSpaceLiouville}
    X^s_T := \{ f\in C([0,T],H^s(\mathbb{R}^6),\mathbb{M}_m(\mathbb{C})) \colon \sup_{t\in[0,T]}\|f(t)\|_{H^s} < R\}
\end{equation}
Note that the mild formulation in \eqref{eq:DuhamelOperator} preserves the constraint $[P,\Phi^{\delta}] = 0$. In fact, denote $f := \frac{1}{i}[P,\Phi]$. Then $f$ satisfies
\begin{equation*}
    \partial_t f = \delta \Delta f + \text{ terms linear in }f, \quad f|_{t=0} = 0.
\end{equation*}
We have
\begin{align*}
    \|G(\Phi^{\delta}(t))\|_{H^s} &\lesssim \|e^{\delta t\Delta}
    \Phi^{\mathrm{in}}\|_{H^s} + \int_0^t \left\|e^{\delta (t-s)\Delta} \left( \tfrac{1}{2} \left( \{P(s), \Phi^{\delta}(s) \} - \{ \Phi^{\delta}(s), P(s) \} \right) + \tfrac{1}{i}[P(s),\mathcal{Z}^{\delta}(s)]\right)\right\|_{H^s} ds \\ &\lesssim \|e^{\delta t\Delta}
    \Phi^{\mathrm{in}}\|_{H^s} \\ &\qquad + \int_0^t \left\|e^{\delta (t-s)\Delta} \left(-\partial_{q_k}\tfrac{1}{2} [\alpha_k,\Phi^{\delta}(s)]_+ - \partial_j A_k(s)\partial_{p_k}\tfrac{1}{2} [\alpha_k, \Phi^{\delta}(s)]_+ +\partial_kV(s)\partial_{p_k}\Phi^{\delta}(s)\right) \right\|_{H^s} ds\\ &\qquad  +\int_0^t \left\|e^{\delta (t-s)\Delta}[P(s),\mathcal{Z}^{\delta}(s)]\right\|_{H^s} ds.
\end{align*}
The two last terms on the right hand side are the Poisson bracket term and the Lagrange multiplier term. We treat them separately. \\

\textbf{The Poisson bracket term:}
Suppose now that $A,V\in H^s_{t,x}$ with $s$ large enough. Then by the heat decay estimate \eqref{eq:HeatDecayDimdSobolevp},
\begin{align*}
    &\left\|e^{\delta (t-s)\Delta} \left(-\partial_{q_k}\tfrac{1}{2} [\alpha_k,\Phi^{\delta}(s)]_+ - \partial_j A_k(s)\partial_{p_k}\tfrac{1}{2} [\alpha_k, \Phi^{\delta}(s)]_+ +\partial_kV(s)\partial_{p_k}\Phi^{\delta}(s)\right) \right\|_{H^s} \\
    &\qquad \lesssim (\delta (t-s))^{1-\frac{3}{r}}\|\langle D\rangle^{s-1}\left(-\partial_{q_k}\tfrac{1}{2} [\alpha_k,\Phi^{\delta}(s)]_+ - \partial_j A_k(s)\partial_{p_k}\tfrac{1}{2} [\alpha_k, \Phi^{\delta}(s)]_+ +\partial_kV(s)\partial_{p_k}\Phi^{\delta}(s)\right) \|_{L^r} \\
    &\qquad \lesssim (\delta (t-s))^{1-\frac{3}{r}}\left[\|\langle D\rangle^{s-1}\nabla_{q}\Phi^{\delta}(s)\|_{L^r} + \|\langle D \rangle^{s-1}(\nabla A(s)\nabla_{p}\Phi^{\delta}(s))\|_{L^r} +\|\langle D \rangle^{s-1}(\nabla V(s)\nabla_{p}\Phi^{\delta}(s))\|_{L^r}\right].
\end{align*}
Now set $r=2$. Bythe Kato-Ponce inequality (fractional Leibniz rule) \eqref{eq:J_ineq},
\begin{align*}
    \|\langle D \rangle^{s-1}(\nabla A\nabla_{p}\Phi^{\delta})\|_{L^2(\mathbb{R}^6)} &\lesssim \| A \|_{W^{s,\infty}_q} \|\nabla_p \Phi^{\delta}\|_{L^{2}} + \|\nabla A\|_{L^{\infty}_q}\| \Phi^{\delta} \|_{H^{s}} \\
    &\lesssim \| A \|_{W^{s,\infty}_q} \|\Phi^{\delta}\|_{H^{1}} + \|\nabla A\|_{L^{\infty}_q}\| \Phi^{\delta} \|_{H^{s}},
    \\
     \|\langle D \rangle^{s-1}(\nabla V\nabla_{p}\Phi^{\delta})\|_{L^2(\mathbb{R}^6)} &\lesssim \| V \|_{W^{s,\infty}_q} \|\nabla_p \Phi^{\delta}\|_{L^{2}} + \|\nabla V\|_{L^{\infty}_q}\| \Phi^{\delta} \|_{H^{s}} \\
    &\lesssim \| V \|_{W^{s,\infty}_q} \|\Phi^{\delta}\|_{H^{1}} + \|\nabla V\|_{L^{\infty}_q}\| \Phi^{\delta} \|_{H^{s}}.
\end{align*}
Thus, e.g. for $s=1$, Lipschitz continuous potentials would be enough, for $s=2$ it would be enough to take $\nabla V$, $\nabla A$ to be Lipschitz. 

Thus, we obtain
\begin{align*}
    &\left\|e^{\delta (t-s)\Delta} \left(-\partial_{q_k}\tfrac{1}{2} [\alpha_k,\Phi^{\delta}(s)]_+ - \partial_j A_k(s)\partial_{p_k}\tfrac{1}{2} [\alpha_k, \Phi^{\delta}(s)]_+ +\partial_kV(s)\partial_{p_k}\Phi^{\delta}(s)\right) \right\|_{H^s} \\
    &\qquad \lesssim (\delta (t-s))^{-\frac{1}{2}}\left(1+\|A(s)\|_{W^{s,\infty}}+\|V(s)\|_{W^{s,\infty}}\right)\|\Phi^{\delta}(s)\|_{H^s}.
\end{align*}\\

\textbf{The Lagrange multiplier term:}
We have to estimate
\begin{equation*}
    \left\|e^{\delta (t-s)\Delta}[P(s),\mathcal{Z}^{\delta}(s)]\right\|_{H^s}.
\end{equation*}
By the heat decay estimate \eqref{eq:HeatDecayDim6Sobolev2},
\begin{align}
\label{eq:HeatDecayLagrangeMultiplier}
    \left\|e^{\delta (t-s)\Delta}[P(s),\mathcal{Z}^{\delta}(s)]\right\|_{H^s} &\lesssim (\delta (t-s))^{1-\frac{3}{r}} \|\langle D\rangle^{s-1}[P(s),\mathcal{Z}^{\delta}(s)]\|_{L^r}. 
\end{align}
Again, we can take $r=2$. Thus we have to estimate
\begin{equation*}
    \|[P(s),\mathcal{Z}^{\delta}(s)]\|_{H^{s-1}}
\end{equation*}
in terms of $\|\Phi(s)\|_{H^{s}}$. This is the content of the following lemma.
\begin{lemma}
\label{eq:PropertiesEstimatesZdelta}
    Let $P$ be the Dirac Hamiltonian. Let $\delta \geq 0$ be a real number. We have the following formula for the Lagrange multiplier $\mathcal{Z}^{\delta}$:
    \begin{equation}
        \mathcal{Z}^{\delta} = \tfrac{i}{\sigma^2}\left(-[\partial_tP, \Phi^{\delta}] - [P,\tfrac{1}{2}(\{P,\Phi^{\delta}\}-\{\Phi^{\delta},P\})]+ \delta[\Delta P,\Phi^{\delta}]+2\delta[\nabla P,\nabla\Phi^{\delta}]\right),
        \label{eq:IdentityZdelta}
    \end{equation}
    where 
    \begin{equation*}
        \sigma \equiv \sigma(t,x,\xi) := |\lambda_+(t,x,\xi)-\lambda_-(t,x,\xi)| = 2\sqrt{1+|\xi-A(t,x)|^2},
    \end{equation*}
    is the distance between the eigenvalues of the Dirac Hamiltonian and where $\nabla=\nabla_{q,p}$ and $\Delta=\Delta_{q,p}$ denote the gradient and Laplacian in $(q,p)$, respectively.
    In particular, we have the following estimates
    \begin{align}
        \|\mathcal{Z}^{\delta}\|_{L^2_{q,p}} &\lesssim \left(\|\tfrac{1}{\sigma^2}\partial_t P\|_{L^{6}_{q,p}}+ \|\tfrac{1}{\sigma}P\|_{W^{1,{\infty}}_{q,p}}^2 + \delta \|\tfrac{1}{\sigma^2}P\|_{W^{2,{\infty}}_{q,p}}\right)\|\Phi^{\delta}\|_{H^1_{q,p}}. 
        \label{eq:EstimateZdeltainL^2}\\
        \|[P,\mathcal{Z}^{\delta}]\|_{L^2_{q,p}} &\lesssim \|\tfrac{1}{\sigma} P\|_{L^{\infty}_{q,p}}\left(\|\tfrac{1}{\sigma}\partial_t P\|_{L^{6}_{q,p}}+ \|\tfrac{1}{\sigma}P\|_{L^{\infty}_{q,p}}\|P\|_{\dot{W}^{1,\infty}_{q,p}} + \delta \|\tfrac{1}{\sigma}D^2_{q,p}P\|_{L^{\infty}_{q,p}}\right)\|\Phi^{\delta}\|_{H^1_{q,p}}. 
        \label{eq:EstimatePZdeltainL^2}
    \end{align}
    More generally,
    \begin{align}
        \|\mathcal{Z}^{\delta}\|_{H^s_{q,p}} &\lesssim  \left( \|\tfrac{1}{\sigma^2}\partial_tP \|_{{W}^{s,6}_{q,p}}+ F\left(\|\tfrac{1}{\sigma}P\|_{W^{s+1,\infty}_{q,p}}\right)+ \delta\|\tfrac{1}{\sigma^2}P\|_{W^{s+2,\infty}_{q,p}}\right) \|\Phi^{\delta}\|_{H^{s+1}_{q,p}}
        \label{eq:EstimateZdeltainH^s} \\
        \|[P,\mathcal{Z}^{\delta}]\|_{H^s_{q,p}} &\lesssim \left(\|\tfrac{1}{\sigma^2}\partial_tP \|_{\dot{W}^{s,3}\cap L^{6}} + F\left(\|\tfrac{1}{\sigma}P\|_{W^{s+1,\infty}_{q,p}}\right) + \delta \left\|\tfrac{1}{\sigma^2} P\right\|_{W^{s+1,{3}}\cap \dot{W}^{1,\infty}}\right) \|\Phi^{\delta}\|_{H^{s+1}_{q,p}}
        \label{eq:EstimatePZdeltainH^s}
    \end{align}
    Here, $F$ denotes a continuous function vanishing at the origin.
    \begin{remark}
        Note that $\delta = 0$ is included in the statement, i.e. the Lemma holds for the unregularized Lagrange multiplier $\mathcal{Z}$ as well.
    \end{remark}
    \begin{remark}
        Note that due to factor $\tfrac{1}{\sigma^2} \in W^{s,\infty}(\mathbb{R}_t \times \mathbb{R}^d_q \times \mathbb{R}^d_{p})$ one does not have to control $P$ in $L^{\infty}$, only $\tfrac{1}{\sigma}P$. In fact for the Dirac equation, $P\notin L^{\infty}$, at least not in the $p$-variable, where it is only $L^{\infty}_{\mathrm{loc}}$ (since it grows linearly in $p$). Thus, by using formula \eqref{eq:IdentityZdelta} as opposed to estimating $[P,\mathcal{Z}^{\delta}]$ by $2P\mathcal{Z}^{\delta}$ and applying the fractional Leibniz rule, we do not have to control low frequencies of $P$ in $L^{\infty}$. However, we have to control derivatives of $P$ in $L^{\infty}$ for the term $[P,\mathcal{Z}^{\delta}]$, e.g. the term $\|P\|_{\dot{W}^{1,\infty}_{q,p}}$ in \eqref{eq:EstimatePZdeltainL^2}. This is due to the Poisson bracket term where we have to control
        \begin{equation*}
            \tfrac{1}{\sigma^2}[P,[P,\tfrac{1}{2}(\{P,\Phi\}-\{\Phi^{\delta},P\})]] = [\tfrac{1}{\sigma}P,[\tfrac{1}{\sigma}P,\tfrac{1}{2}(\{P,\Phi\}-\{\Phi^{\delta},P\})]].
        \end{equation*}
        Hence at some point the derivatives of $P$ coming from the Poisson bracket have to be controlled in $L^{\infty}$, but we never have to control $P$ in $L^{\infty}$.
    \end{remark}
    \begin{proof}
        Applying Lemma \ref{lemma:CharacterizationSubspace} to $\mathcal{L}$ yields
        \begin{equation*}
            \mathcal{Z}^{\delta} = \tfrac{1}{\sigma^2}[P,[P,\mathcal{Z}^{\delta}]] = \tfrac{i}{\sigma^2}[P,\mathcal{L}^{\delta}].
        \end{equation*}
        From the equation we obtain
        \begin{equation*}
            \mathcal{L}^{\delta} = \partial_t \Phi^{\delta} - \tfrac{1}{2}(\{P,\Phi^{\delta}\}-\{\Phi,P\}) - \delta \Delta \Phi^{\delta}.
        \end{equation*}
        Applying $\frac{i}{\sigma^2}[P,\cdot]$ to this identity yields
        \begin{equation*}
            \mathcal{Z}^{\delta} = \tfrac{i}{\sigma^2}[P,\partial_t \Phi^{\delta} - \tfrac{1}{2}(\{P,\Phi^{\delta}\}-\{\Phi^{\delta},P\})- \delta \Delta \Phi^{\delta}].
        \end{equation*}
        Since $[P,\Phi] = 0$ we have
        \begin{equation*}
            [P,\partial_t \Phi^{\delta}] = -[\partial_tP, \Phi^{\delta}] \quad \text{and} \quad [P,\Delta \Phi^{\delta}]= -[\Delta P,\Phi^{\delta}]-2[\nabla P,\nabla \Phi^{\delta}].
        \end{equation*}
        This proves \eqref{eq:IdentityZdelta}. Estimates \eqref{eq:EstimateZdeltainL^2} and \eqref{eq:EstimatePZdeltainL^2} are  direct consequences of this formula and the Sobolev embedding $H^1(\mathbb{R}^6)\hookrightarrow L^3(\mathbb{R}^6)$.

        For the $H^s$ estimate \eqref{eq:EstimateZdeltainH^s}, $s\geq 1$ we consider
        \begin{equation}
            \|D^s \mathcal{Z}\|_{L^2_{q,p}} \leq 2\|D^s\left(\tfrac{1}{\sigma^2}\partial_t P \Phi\right) \|_{L^2_{q,p}} +\|D^s\left( \tfrac{1}{\sigma^2}P(\{P,\Phi\}-\{\Phi,P\})\right)\|_{L^2_{q,p}}.
        \end{equation}
        Now we apply the fractional Leibniz rule to both terms on the right hand side. First,
        \begin{align*}
            \|D^s\left(\tfrac{1}{\sigma^2}\partial_t P \Phi\right) \|_{L^2_{q,p}} &\leq \|D^s(\tfrac{1}{\sigma^2}\partial_t P)\|_{L^3}\|\Phi\|_{L^6} + \|\tfrac{1}{\sigma^2}\partial_t P\|_{L^{6}}\|D^s\Phi\|_{L^3} \\
            &\lesssim \|\tfrac{1}{\sigma^2}\partial_tP \|_{\dot{W}^{s,3}\cap L^{6}}\|\Phi\|_{H^{s+1}},
        \end{align*}
        where the second line follows from the  the Sobolev embedding $H^1(\mathbb{R}^6) \hookrightarrow L^{3}(\mathbb{R}^6)$. Second,
        \begin{align*}
            \|D^s\left( \tfrac{1}{\sigma^2}P(\{P,\Phi\}-\{\Phi,P\})\right)\|_{L^2_{q,p}} &\leq \|D^s\tfrac{1}{\sigma}P\|_{L^{\infty}} \|\tfrac{1}{\sigma}(\{P,\Phi\}-\{\Phi,P\})\|_{L^2}  \\
            &\qquad \qquad + \|\tfrac{1}{\sigma}P\|_{L^{\infty}}\|D^s\tfrac{1}{\sigma} (\{P,\Phi\}-\{\Phi,P\})\|_{L^2} \\ &\lesssim \|D^s\tfrac{1}{\sigma}P\|_{L^\infty}\|\tfrac{1}{\sigma}\nabla_{q,p}P\|_{L^{\infty}}\|\nabla_{q,p}\Phi\|_{L^2}\\&\qquad \qquad  + \|\tfrac{1}{\sigma}P\|_{L^{\infty}}\Big(\|\tfrac{1}{\sigma}\nabla_{q,p}P\|_{L^{\infty}}\|D^s\nabla_{q,p}\Phi\|_{L^2} \\ &\qquad \qquad \qquad +\|D^s\tfrac{1}{\sigma}\nabla_{q,p}P\|_{\infty}\|\nabla_{q,p}\Phi\|_{L^2} \Big) \\
            &\lesssim F\left(\|\tfrac{1}{\sigma}P\|_{W^{s+1,\infty}_{q,p}}\right)\|\Phi\|_{H^{s+1}_{q,p}},
        \end{align*}
        where $F$ is some continuous function vanishing at the origin. 

        Finally, the parabolic regularization term can easily be estimated by
        \begin{align*}
            \left\|D^s\left[\tfrac{1}{\sigma^2}\Delta P, \Phi\right]\right\|_{L^2_{q,p}} &\lesssim \left\|D^s\tfrac{1}{\sigma^2}\Delta P\right\|_{L^{3}} \|\Phi\|_{L^6} + \left\|\tfrac{1}{\sigma^2}\Delta P\right\|_{L^{6}}\|D^s\Phi\|_{L^3} \\
            &\lesssim \left\|\tfrac{1}{\sigma^2} P\right\|_{W^{s+2,{3}}\cap \dot{H}^2} \|\Phi\|_{H^{s+1}_{q,p}},
        \end{align*}
        and
        \begin{align*}
            \left\|D^s\left[\tfrac{1}{\sigma^2}\nabla P, \nabla \Phi\right]\right\|_{L^2_{q,p}} &\lesssim \left\|D^s\tfrac{1}{\sigma^2}\nabla P\right\|_{L^{3}} \|\nabla\Phi\|_{L^6} + \left\|\tfrac{1}{\sigma^2}\nabla P\right\|_{L^{\infty}}\|D^s\nabla\Phi\|_{L^2} \\
            &\lesssim \left\|\tfrac{1}{\sigma^2} P\right\|_{W^{s+1,{3}}\cap \dot{W}^{1,\infty}} \|\Phi\|_{H^{s+1}_{q,p}}.
        \end{align*}
        This concludes the proof.
    \end{proof}
\end{lemma}
Applying Lemma \ref{eq:PropertiesEstimatesZdelta} to \eqref{eq:HeatDecayLagrangeMultiplier} yields
\begin{align}
\label{eq:HeatDecayLagrangeMultiplierL2}
   \left\|e^{\delta (t-s)\Delta}[P(s),\mathcal{Z}^{\delta}(s)]\right\|_{H^s} &\lesssim (\delta (t-s))^{-\frac{1}{2}} \|\langle D \rangle^{s-1}P(s)\|_{\dot{W}^{1,\infty}}\|\Phi^{\delta}(s)\|_{H^s}
\end{align}

\textbf{Final estimate:}
Now we can finalize the estimate of \eqref{eq:DuhamelOperator} in $H^s$. Combining the estimates for Poisson bracket term and \eqref{eq:HeatDecayLagrangeMultiplierL2} yields
\begin{align*}
    \|G(\Phi^{\delta}(t))\|_{H^s} &\lesssim \|e^{\delta t\Delta}
    \Phi^{\mathrm{in}}\|_{H^s} + \int_0^t(\delta (t-s))^{-\frac{1}{2}}\left(\|A(s)\|_{W^{s,\infty}}+\|V(s)\|_{W^{s,\infty}}\right)\|\Phi^{\delta}(s)\|_{H^s}ds \\ &\qquad \qquad + \int_0^t (\delta (t-s))^{-\frac{1}{2}} \|\langle D \rangle^{s-1}P(s)\|_{\dot{W}^{1,\infty}}\|\Phi^{\delta}(s)\|_{H^s}ds \\
    &\lesssim \|e^{\delta t\Delta}
    \Phi^{\mathrm{in}}\|_{H^s} + \left(\|A(s)\|_{L^{\infty}_tW^{s,\infty}}+\|V(s)\|_{L^{\infty}_tW^{s,\infty}}\right)\|\Phi^{\delta}(s)\|_{L^{\infty}_tH^s}\int_0^t(\delta (t-s))^{-\frac{1}{2}}ds \\ &\qquad \qquad + \|\langle D \rangle^{s-1}P(s)\|_{L^{\infty}_t\dot{W}^{1,\infty}}\|\Phi^{\delta}(s)]\|_{L^{\infty}_tH^s}\int_0^t (\delta (t-s))^{-\frac{1}{2}} ds \\
    &\lesssim \|e^{\delta t\Delta}
    \Phi^{\mathrm{in}}\|_{H^s} \\ &\qquad + 2\delta^{-\frac{1}{2}}t^{\frac{1}{2}}\left(\|A(s)\|_{L^{\infty}_tW^{s,\infty}}+\|V(s)\|_{L^{\infty}_tW^{s,\infty}}+\|\langle D \rangle^{s-1}P(s)\|_{L^{\infty}_t\dot{W}^{1,\infty}}\right)\|\Phi^{\delta}(s)\|_{L^{\infty}_tH^s}. 
\end{align*}
Thus, if 
\begin{equation}
    \|A(s)\|_{L^{\infty}_tW^{s,\infty}}+\|V(s)\|_{L^{\infty}_tW^{s,\infty}}+\|\langle D \rangle^{s-1}P(s)\|_{L^{\infty}_t\dot{W}^{1,\infty}} \leq C,
\end{equation}
for some constant $C$, we have for $\Phi^{\delta}\in X^s_T$ (recall the definition of $X^s_T$ in \eqref{eq:SolutionSpaceLiouville}),
\begin{equation}
     \sup_{t\in [0,T]}\|G(\Phi^{\delta}(t))\|_{H^s} \lesssim \|\Phi^{\mathrm{in}}\|_{H^s} + 2C\delta^{-\frac{1}{2}}T^{\frac{1}{2}}R.
\end{equation}
An appropriate choice of $T^{\delta}=T^{\delta}(\delta,R)$ shows that $G$ maps $X^T_s$ into itself. A straightforward argument for the difference $G(\Phi^{\delta})-G(\Psi^{\delta})$ with identical initial data yields the desired contraction. Note that since $T^{\delta}$ depends on $\delta$ it is not immediately clear that one can pass to the limit. To this end we need an a priori estimate for the solution $\Phi^{\delta}$ that holds uniformly in $\delta$. This will be dealt with in the next section.

\subsection{A priori estimate in $H^k$}
In order to pass to the limit in the regularized equation, we need an a priori estimate on the $H^k$ norm of $\Phi$, where $k$ is an integer. 
\begin{proposition}
\label{thm:AprioriMatrixLiouville}
    Let $A,V\in H^k(\mathbb{R}\times \mathbb{R}^d)$ for $k$ large enough. Let $\Phi$ be a solution in $C^1_tH^k_{q,p}$ of the matrix Liouville equation \eqref{eq:MatrixLiouvilleTest} and let $k\geq 1$. Then we have the following differential inequality:
    \begin{equation}
    \label{eq:AprioriMatrixLiouville}
        \frac{d}{dt}\|\Phi(t)\|^2_{H^k_{q,p}} \lesssim (1+F\left(\|\nabla A(t)\|_{{W}^{k,\infty}},\|\partial_t A(t)\|_{{W}^{k-1,\infty}},\|\nabla V(t)\|_{W^{k,\infty}}\right)\|\Phi(t)\|_{H^k_{q,p}}^2,
    \end{equation}
    where $F$ is a continuous function vanishing at the origin. The a priori estimate also holds for the regularized solution $\Phi^{\delta}$.
    \begin{remark}
        In the proof of the a priori estimate it is crucial to estimate $\mathcal{Z}$ in $H^{k-1}$, since its $H^{k-1}$ norm is controlled by the $H^k$ norm of $\Phi$, cf. Lemma \ref{eq:PropertiesEstimatesZdelta}. In order to achieve that we need to exploit the commutativity of $P$ and $\Phi$ and partial integration in order to shift the $k$th derivative of $\mathcal{Z}$ on $P$ and $\Phi$.
    \end{remark}
\begin{proof}    
Let $\alpha$ be a multiindex with $|\alpha|=k$. We apply $D^{\alpha}=D^{\alpha}_{q,p}$ to \eqref{eq:MatrixLiouvilleTest}
and obtain
\begin{align}
    \begin{cases}
\partial_t D^{\alpha}\Phi = D^{\alpha}\left(\tfrac{1}{i}[P,\mathcal{Z}]\right) + \tfrac{1}{2} D^{\alpha}\left( \{P, \Phi \} - \{ \Phi, P \} \right),  \\
D^{\alpha}[P,\Phi] = 0.
\end{cases}
\end{align}
Multiplying by $D^{\alpha}\Phi$, taking the trace $\tr_{\mathbb{C}^m}$ and integrating over $\mathbb{R}^d_q\times \mathbb{R}^d_p$ yields
\begin{align*}
    \partial_t \|D^{\alpha}\Phi\|_{L^2(\mathbb{R}^d_q\times \mathbb{R}^d_p,\mathbb{M}_m(\mathbb{C}))} &= \iint \tr(D^{\alpha}\Phi\tfrac{1}{2}D^{\alpha} \left( \{P, \Phi \} - \{ \Phi, P \} \right)) dqdp \\
    &\qquad + \iint \tr(D^{\alpha}\Phi\tfrac{1}{i}D^{\alpha}[P,\mathcal{Z}])dqdp.
\end{align*}
The addition of the parabolic term $\delta\Delta \Phi$ only adds a negative term $-\delta\|D^{\alpha}\nabla\Phi^{\delta}\|_L^2$ to the right hand side and therefore all the calculations in this section carry over to the regularized equation.\\

\textbf{The Poisson bracket term:} The matrix-valued Poisson bracket $\tfrac{1}{2}(\{P,\Phi\}-\{\Phi,P\})$ is the matrix-valued equivalent of the scalar Poisson bracket $\{f,g\}$. In fact, it is a generalization in the sense that it obeys the same defining properties. It is anti-commutative:
\begin{equation}
    \tfrac{1}{2}(\{P,\Phi\}-\{\Phi,P\}) = -\tfrac{1}{2}(\{\Phi,P\}-\{P,\Phi\})
\end{equation}
(Note that $\{P,\Phi\} \neq -\{\Phi,P\}$ for matrix-valued functions). Moreover, it is bilinear, obeys the Leibniz rule and the Jacobi identity when under the trace $\tr(\cdot)$ (by cyclicity). Therefore the calculations for the scalar case, where one has to calculate
\begin{equation}
    \iint \nabla_{x,\xi}f\cdot \nabla_{x,\xi}\{H,f\}d x\xi,
\end{equation}
can be carried out analogously in the matrix-valued case.\\

\textbf{The Lagrange multiplier term:}

Contrary to the scalar Liouville equation, we have to deal with the Lagrange multiplier term in the matrix-valued case. Let $|\alpha|\leq k$. Then
\begin{align*}
    \iint \tr(D^{\alpha}\Phi\tfrac{1}{i}D^{\alpha}[P,\mathcal{Z}])dqdp &= \iint \tr(D^{\alpha}\Phi\sum_{|\beta|\leq|\alpha|}\begin{pmatrix}
        \alpha \\ \beta
    \end{pmatrix}\tfrac{1}{i}[D^{\alpha-\beta}P,D^{\beta}\mathcal{Z}])dqdp \\ 
    &= -\iint \tr(\sum_{|\beta|\leq|\alpha|}\begin{pmatrix}
        \alpha \\ \beta
    \end{pmatrix}\tfrac{1}{i}[D^{\alpha-\beta}P,D^{\alpha}\Phi] D^{\beta}\mathcal{Z})dqdp \\
    &= -\iint \tr(\sum_{|\beta|\leq|\alpha|-1}\begin{pmatrix}
        \alpha \\ \beta
    \end{pmatrix}\tfrac{1}{i}[D^{\alpha-\beta}P,D^{\alpha}\Phi] D^{\beta}\mathcal{Z})dqdp \\&\qquad \qquad  -\iint \tr(\tfrac{1}{i}[P,D^{\alpha}\Phi] D^{\alpha}\mathcal{Z})dqdp. \\
    \end{align*}
Since $D^{\alpha}[P,\Phi] = 0$ we obtain 
\begin{equation*}
    [P,D^{\alpha}\Phi] = -\sum_{|\beta|\leq|\alpha|-1}\begin{pmatrix}
        \alpha \\ \beta
    \end{pmatrix}[D^{\alpha-\beta}P,D^{\beta}\Phi],
\end{equation*}
and so
\begin{align*}
    D[P,D^{\alpha}\Phi] &= -D\sum_{|\beta|\leq|\alpha|-1}\begin{pmatrix}
        \alpha \\ \beta
    \end{pmatrix}[D^{\alpha-\beta}P,D^{\beta}\Phi] \\
    &= -\sum_{|\beta|\leq|\alpha|-1}\begin{pmatrix}
        \alpha \\ \beta
    \end{pmatrix}\left([D^{\alpha-\beta+1}P,D^{\beta}\Phi] +[D^{\alpha-\beta}P,D^{\beta+1}\Phi]\right).
\end{align*}
Therefore,
\begin{align*}
    -\iint \tr(\tfrac{1}{i}[P,D^{\alpha}\Phi] D^{\alpha}\mathcal{Z})dqdp = \iint \tr(\tfrac{1}{i}\sum_{|\beta|\leq|\alpha|-1}\begin{pmatrix}
        \alpha \\ \beta
    \end{pmatrix}[D^{\alpha-\beta}P,D^{\beta}\Phi]D^{\alpha}\mathcal{Z})dqdp \\ =
    -\iint \tr(\tfrac{1}{i}\sum_{|\beta|\leq|\alpha|-1}\begin{pmatrix}
        \alpha \\ \beta
    \end{pmatrix}\left([D^{\alpha-\beta+1}P,D^{\beta}\Phi] +[D^{\alpha-\beta}P,D^{\beta+1}\Phi]\right)D^{\alpha-1}\mathcal{Z})dqdp.
\end{align*}
It follows that
\begin{align*}
     \iint \tr(D^{\alpha}\Phi\tfrac{1}{i}D^{\alpha}[P,\mathcal{Z}])dqdp = -\iint \tr(\sum_{|\beta|\leq|\alpha|-1}\begin{pmatrix}
        \alpha \\ \beta
    \end{pmatrix}\tfrac{1}{i}[D^{\alpha-\beta}P,D^{\alpha}\Phi] D^{\beta}\mathcal{Z})dqdp  \\
    -\iint \tr(\tfrac{1}{i}\sum_{|\beta|\leq|\alpha|-1}\begin{pmatrix}
        \alpha \\ \beta
    \end{pmatrix}\left([D^{\alpha-\beta+1}P,D^{\beta}\Phi] +[D^{\alpha-\beta}P,D^{\beta+1}\Phi]\right)D^{\alpha-1}\mathcal{Z})dqdp.
\end{align*}
We observe that the highest derivative of $\mathcal{Z}$ is of order $|\alpha|-1 \leq k-1$ and the highest derivative of $\Phi$ is of order $|\alpha| \leq k$, while the highest derivative of $P$ is of order $|\alpha|+1 \leq k+1$. Therefore we obtain an estimate of the form
\begin{equation*}
    \left|\iint \tr(D^{\alpha}\Phi\tfrac{1}{i}D^{\alpha}[P,\mathcal{Z}])dqdp\right| \lesssim \|P\|_{\dot{W}^{\infty,k+1}}\|\Phi\|_{\dot{H}^k}\|\mathcal{Z}\|_{\dot{H}^{k-1}}.
\end{equation*}
By estimate \eqref{eq:EstimateZdeltainH^s}, we obtain
\begin{equation}
    \left|\iint \tr(D^{\alpha}\Phi\tfrac{1}{i}D^{\alpha}[P,\mathcal{Z}])dqdp\right| \lesssim (1+F\left(\|\nabla A(t)\|_{{W}^{k,\infty}},\|\partial_t A(t)\|_{{W}^{k-1,\infty}},\|\nabla V(t)\|_{W^{k,\infty}}\right)\|\Phi\|_{\dot{H}^k}^2.
\end{equation}
where $F$ denotes a continuous function vanishing at the origin.
Combined with the $L^2$ identity
\begin{equation}
    \iint \tr(\Phi\tfrac{1}{i}[P,\mathcal{Z}])dqdp = 0,
\end{equation}
this yields the claim. The proof is completed.
\end{proof}
\end{proposition}
\begin{remark}
Note that a more careful analysis would yield some more optimized estimates for the symbol $P$, by using e.g. Sobolev's inequality and other cancellation properties of the Dirac symbol (e.g. $D^k_pP= 0$ for $k\geq 2$ and $D^k_pD^j_qP=0$ for $j,k\geq 1$).
\end{remark}
    \subsection{Convergence of the regularized equation}

    As a first consequence of the a priori estimate \eqref{eq:AprioriMatrixLiouville}, we can extend the $\delta$-dependent interval $[0,T^{\delta})$ to an interval  $[0,T]$, independent of $\delta$.
    
    The convergence of solutions of the regularized equation to strong solutions of the original equation follows an argument due to Bona-Smith \cite{bona_initial-value_1975}, which boils down to showing that the solutions to the regularized equation are Cauchy in $C([0,T],L^2)$. To this end consider the difference $\chi$ between to solutions $\Phi^{\delta},\Phi^{\delta'}$:
    \begin{equation}
        \chi := \Phi^{\delta}-\Phi^{\delta'}.
    \end{equation}
    It is easy to see that $\chi$ solves
    \begin{align}
        \begin{cases}\partial_t \chi = \tfrac{1}{i}[P,\mathcal{Z}^{\delta}-\mathcal{Z}^{\delta'}] +\tfrac{1}{2}\left(\{P,\chi\}-\{\chi,P\}\right) + (\delta-\delta')\Delta \Phi^{\delta}+\delta'\Delta \chi,\\
        [P,\chi] = 0.
        \end{cases}
    \end{align}
    The energy estimate for $\chi$ yields
    \begin{align*}
        \tfrac{1}{2}\partial_t\|\chi\|^2_{L^2} &\lesssim -\iint \tr(\chi\tfrac{1}{i}[P,\mathcal{Z}^{\delta}-\mathcal{Z}^{\delta'}])dqdp + |\delta-\delta'|\iint \tr(\chi\Delta\Phi^{\delta})dqdp \\ &\qquad \qquad -\delta'\iint \tr(\chi\Delta \chi) dqdp \\ &= \iint \tr(\tfrac{1}{i}[P,\chi](\mathcal{Z}^{\delta}-\mathcal{Z}^{\delta'}))dqdp + |\delta-\delta'|\iint \tr(\chi\Delta\Phi^{\delta})dqdp - \delta'\|\nabla\chi\|^2_{L^2}
    \end{align*}
    Thus, since $[P,\chi] = 0$ and since the last term is negative,
    \begin{equation*}
        \tfrac{1}{2}\partial_t\|\chi\|^2_{L^2} \lesssim (\delta-\delta') \left|\iint \tr(\chi\Delta\Phi^{\delta})dqdp\right| \lesssim |\delta-\delta'| \|\chi\|_{L^2}\|\Phi^{\delta}\|_{H^2}.
    \end{equation*}
    Taking the supremum in $t\in[0,T]$ and applying Gronwall's inequality yields
    \begin{equation}
        \sup_{t\in [0,T]}\|\chi\|_{L^2}\lesssim |\delta-\delta'|,
    \end{equation}
    and so $\Phi^{\delta}$ is Cauchy in $C([0,T],L^2)$. By the a priori estimate \eqref{eq:AprioriMatrixLiouville}, we have that $\Phi^{\delta} \in L^{\infty}([0,T],H^s)$ where $s\geq 2$. By interpolation with this uniform-in-$\delta$ bound we obtain that $\Phi^{\delta}$ is Cauchy in $C([0,T],H^r)$ for $r<s$. 

    From the equation one has $\Phi^{\delta} \in C([0,T],H^{r-2})$. By taking pointwise limits, we obtain that $\Phi \in C([0,T],H^r)\cap C^1([0,T],H^{r-2})$ solves the unregularized equation. Moreover, since $\Phi^{\delta}$ is bounded in $H^s$ uniformly in $\delta$ and $\Phi^{\delta}$ converges to $\Phi$ in $L^2$, we obtain that $\Phi \in L^{\infty}([0,T],H^s)$.
    
    For the continuity in time, we define the approximate identity $\varphi^{\eta}$, $\eta > 0$ where $\varphi \in \mathcal{S}$ is such that $\widehat{\varphi}(0) = 1$, $\partial_{\zeta}^k \widehat{\varphi}(0) = 0$ for all $k>0$ and $\varphi^{\eta}(z) = \eta^{-2d} \varphi(z/\eta)$. Then define the mollified initial data $\Phi_{\mathrm{in}}^{\eta} = \varphi^{\eta} \ast \Phi^{\mathrm{in}}$. Let $\Phi^{\eta}$ be the solution $\Phi_{\mathrm{in}}^{\eta}$. Then one has
    \begin{equation}
        \label{eq:APrioriRegularizedPhi}\|\Phi^{\eta}\|_{H^{s+l}} \leq \eta^{-l}.
    \end{equation}
    Indeed, due to the a priori estimate it suffices to show the claim at $t=0$. Then use Fourier to obtain
    \begin{equation*}
        \|\Phi^{\eta}_{\mathrm{in}}\|_{H^{s+l}} = \|\langle \zeta \rangle^{l} \widehat{\varphi}(\eta\zeta) \langle \zeta \rangle^s \widehat{\Phi^{\eta}_{\mathrm{in}}}\|_{L^2} \lesssim \|\langle \zeta \rangle^{l} \widehat{\varphi}(\eta\zeta)\|_{L^{\infty}} \|\Phi^{\mathrm{in}}\|_{H^{s}} \lesssim \eta^{-l}.
    \end{equation*}
Combined with interpolation, estimate \eqref{eq:APrioriRegularizedPhi} shows that
    \begin{equation*}
        \Phi^{\eta} \in C([0,T],L^2)\cap L^{\infty}([0,T],H^{s+l}) \subset C([0,T],H^s).
    \end{equation*}
    It remains to show that $\Phi^{\eta}$ converges uniformly to $\Phi$. Again, we are reduced to showing the claim for the initial data by the a priori estimate. The claim follows by a similar analysis to the one above using the fact that the derivatives of $\widehat{\varphi}$ vanish at the origin. Thus we have shown that
    \begin{equation}
        \Phi \in C([0,T],H^s).
    \end{equation}
    
    Moreover, standard arguments lead to unicity and continuous dependence on the data (for reference cf. \cite{erdogan_dispersive_2016}).

\appendix
\section{Characterization of off-diagonal subspaces}
    \begin{lemma}
\label{lemma:CharacterizationSubspace}
Let $P$ be a matrix-valued symbol 
with distinct eigenvalues $\lambda_{j}$, $j=1,...,m$ of constant multiplicity (for all $(x,\xi) \in \mathbb{R}^d_x \times \mathbb{R}^d_{\xi}$) with smooth projections $\Pi_{j} \in C^{\infty}(\mathbb{R}_t \times \mathbb{R}^{d}_x \times \mathbb{R}^d_{\xi},\mathbb{M}_m(\mathbb{C}))$ with bounded derivatives. Let $A\in \mathcal{S}'(\mathbb{R}_t \times \mathbb{R}^{d}_x \times \mathbb{R}^d_{\xi},\mathbb{M}_m(\mathbb{C}))$. Then 
\begin{equation*}
    \Pi_{j}A\Pi_{j} = 0
\end{equation*}
if and only if there exists ${B}$ such that
\begin{equation*}
    A = \tfrac{1}{i}[P,B].
\end{equation*}    
Explicitly,
\begin{equation*}
    B = \sum_{j\neq k} \frac{i}{\lambda_j-\lambda_k}\Pi_jA\Pi_k.
\end{equation*}
Moreover, the following formula holds:
\begin{equation}
\label{eq:formulaA}
    A =  \sum_{j\neq k}\frac{1}{(\lambda_j-\lambda_k)^2}[P,[P,\Pi_jA\Pi_k]].
\end{equation}
\begin{proof}
    Suppose that there is a matrix $B$ such that $A = \tfrac{1}{i}[P,B]$. The spectral decomposition of $P$ is given by
    \begin{equation*}
        P = \sum_{j=1}^m \lambda_j \Pi_j.
    \end{equation*}
    Then
    \begin{align*}
        \Pi_jA\Pi_j &= \tfrac{1}{i}\Pi_j[P,B]\Pi_j \\ &= \tfrac{1}{i}\sum_{k=1}^m\Pi_j[\lambda_k\Pi_k,B]\Pi_j \\ &=\tfrac{1}{i}\sum_{k=1}^m\lambda_k\left(\Pi_j\Pi_k B\Pi_j - \Pi_jB\Pi_k\Pi_j\right) \\&=0,
    \end{align*}
    since $\Pi_j\Pi_k = 0$ for $j\neq k$ and $\Pi_j\Pi_k = \Pi_j$ if $j=k$.

    On the other hand, assume that $\Pi_jA\Pi_j =0$. Consider the commutator of $P$ with $\Pi_jA\Pi_k$ with $j\neq k$. We have
    \begin{align*}
        [P,\Pi_jA\Pi_k] &= \sum_{\ell=1}^m\left(\lambda_{\ell}\Pi_{\ell}\Pi_jA\Pi_k - \lambda_{\ell}\Pi_jA\Pi_k\Pi_{\ell}\right) \\ &= (\lambda_j-\lambda_k)\Pi_jA\Pi_k.
    \end{align*}
From this we surmise that
\begin{align*}
    \left[P,\sum_{j\neq k}\frac{1}{\lambda_j-\lambda_k}\Pi_jA\Pi_k\right] = \sum_{j\neq k} \Pi_jA\Pi_k = A,
\end{align*}
since $\Pi_jA\Pi_j = 0$ and $\sum_j \Pi_j = I$.
Thus, $B$ is given by
\begin{equation}
    B = i\sum_{j\neq k}\frac{1}{\lambda_j-\lambda_k}\Pi_jA\Pi_k
\end{equation}

Moreover, we readily observe that $\Pi_j B \Pi_j = 0$, so there exists a $C$ such that $B = \tfrac{1}{i}[P,C]$ with $C$ given by
\begin{align*}
    C &= i\sum_{j'\neq k'}\frac{1}{\lambda_{j'}-\lambda_{k'}}\Pi_{j'}B\Pi_{k'} \\
    &= i\sum_{j'\neq k'}\frac{1}{\lambda_{j'}-\lambda_{k'}}\Pi_{j'}\left(i\sum_{j\neq k}\frac{1}{\lambda_j-\lambda_k}\Pi_jA\Pi_k\right)\Pi_{k'} \\
    &= - \sum_{j'\neq k'}\sum_{j\neq k}\frac{1}{\lambda_{j'}-\lambda_{k'}}\frac{1}{\lambda_j-\lambda_k} \Pi_{j'}\Pi_jA\Pi_k\Pi_{k'} \\ &= - \sum_{j'\neq k'}\sum_{j\neq k}\frac{1}{\lambda_{j'}-\lambda_{k'}}\frac{1}{\lambda_j-\lambda_k} \delta_{j'j}\delta_{k'k}\Pi_jA\Pi_k \\ 
    &= - \sum_{j\neq k}\frac{1}{(\lambda_j-\lambda_k)^2} \Pi_jA\Pi_k.
\end{align*}
Thus we obtain
\begin{align*}
    A &= \tfrac{1}{i}[P,B] \\ &= -[P,[P,C]] \\
    &= \sum_{j\neq k}\frac{1}{(\lambda_j-\lambda_k)^2}[P,[P,\Pi_jA\Pi_k]].
\end{align*}
This concludes the proof.
\end{proof}
\begin{remark}
     If there are only two distinct eigenvalues $\lambda_1<\lambda_2$ with $\sigma := |\lambda_1-\lambda_2| $, then \eqref{eq:formulaA} has the following simplification
    \begin{equation}
        A = \frac{1}{\sigma^2}[P,[P,A]].
    \end{equation}
    Indeed, since $A=\sum_{j\neq k} \Pi_j A\Pi_k$, 
    \begin{equation*}
        A = \sum_{j\neq k}\frac{1}{(\lambda_j-\lambda_k)^2}[P,[P,\Pi_jA\Pi_k]] = \frac{1}{\sigma^2}[P,[P,\Pi_1A\Pi_2+\Pi_2A\Pi_1]] = \frac{1}{\sigma^2}[P,[P,A]].
    \end{equation*}
\end{remark}
\end{lemma}

\section{A counterexample to naive matrix-valued QOT}
\label{sec:Counterexample}
We start by defining couplings between matrix-valued density operators $R$ and classical matrix-valued phase space densities $F$ on $\mathbb{R}^d_x \times \mathbb{R}^d_{\xi}$. Recall that a matrix-valued density operator is an element $R\in \mathcal{L}(L^2(\mathbb{R}^d), \mathbb{M}_m(\mathbb{C}))$, i.e.
\begin{equation}
    R = \sum_{j,k} R_{jk} e_j \otimes e_k^{\ast}, \quad R_{jk} \in \mathcal{L}(L^2(\mathbb{R}^d)) 
\end{equation}
that is positive and trace class.
We represent $R$ by its kernel $R(x,y) \in L^2(\mathbb{R}^d \times \mathbb{R}^d, \mathbb{M}_m(\mathbb{C}))$. 
A classical matrix-valued phase space density is an element $F \in \mathcal{P}(\mathbb{R}^d_x \times \mathbb{R}^d_{\xi}, \mathbb{M}_m(\mathbb{C}))$, i.e.
\begin{equation}
    F(x,\xi) = \sum_{j,k} f_{jk}(x,\xi) e_j\otimes e_k^{\ast},
\end{equation}
such that 
\begin{equation}
    \int_{\mathbb{R}^d_x \times \mathbb{R}^d_{\xi}} \tr(F(x,\xi))dx d\xi = \int_{\mathbb{R}^d_x \times \mathbb{R}^d_{\xi}} \sum_j f_{jj}(x,\xi)dx d\xi = 1.
\end{equation}
We now define a coupling between $F$ and $R$  as an element $$Q\in \mathcal{P}(\mathbb{R}^d_x \times \mathbb{R}^d_{\xi},\mathbb{M}_m) \otimes \mathcal{L}(L^2(\mathbb{R}^d),\mathbb{M}_m),$$i.e.
\begin{equation}
    Q(x,\xi) := \sum_{j,k,\ell,m} q_{jk\ell m}(x,\xi) e_j\otimes e_k^{\ast} \otimes e_{\ell} \otimes e_m^{\ast}, \quad q_{jk\ell m}(x,\xi)\in \mathcal{L}(L^2(\mathbb{R}^d)),
\end{equation}
such that $Q$ has the marginals
\begin{equation}
    \int_{\mathbb{R}^d_x \times \mathbb{R}^d_{\xi}} \sum_{m}q_{jkmm}(x,\xi) dx d\xi = R_{jk},
\end{equation}
and
\begin{equation}
    \sum_{j} \tr_{L^2(\mathbb{R}^d)}( q_{jj\ell m}(x,\xi)) = f_{\ell m}(x,\xi).
\end{equation}
The positivity of $Q$ can be expressed as follows: For all $\Psi \in L^2(\mathbb{R}^d, \mathbb{M}_m(\mathbb{C})$,
\begin{equation}
    \langle \Psi_{j\ell}, q_{jk\ell m}(x,\xi) \Psi_{km} \rangle_{L^2} \geq 0 \quad \text{a.e. in } \mathbb{R}^d_x\times \mathbb{R}^d_{\xi} .
\end{equation}

\textbf{Transportation cost}: In scalar QOT, the cost is an operator $c_{\hbar}(x,\xi,y,D_y)$, e.g.
\begin{equation}
\label{eq:CostDiagional}
    c_{\hbar}(x,\xi,y,D_y) = |x-y|^2 + |\xi- D_y|^2,
\end{equation}
where $D_y = -i\hbar \partial_y$.
The scalar quantum Wasserstein distance $d_{\hbar}$ is defined as
\begin{equation}
    d_{\hbar}(f,R) := \inf_{Q \text{ coupling}} \int_{\mathbb{R}_x^d\times \mathbb{R}_{\xi}^d} \tr(Q(x,\xi)^{\frac{1}{2}}c_{\hbar}(x,\xi,y,D_y)Q(x,\xi)^{\frac{1}{2}}) dx d\xi.
\end{equation}
Again, a naive generalization to matrices would be (assuming everything is such that the trace is well-defined)
\begin{equation}
    \mathfrak{d}_{\hbar}(F,R) := \inf_{Q \text{ coupling}} \int_{\mathbb{R}_x^d\times \mathbb{R}_{\xi}^d} \Tr(Q(x,\xi)^{\frac{1}{2}}c_{\hbar}(x,\xi,y,D_y)Q(x,\xi)^{\frac{1}{2}}) dx d\xi.
    \label{eq:NaiveMatrixDistance}
\end{equation}
where $\Tr(\cdot)$ denotes the trace in $\mathcal{L}(L^2(\mathbb{R}^d,\mathbb{M}_m(\mathbb{C})))$.

It seems that applying the diagonal cost $c$ \eqref{eq:CostDiagional} here causes the same problem as for the classical matrix-valued optimal transport. In fact, we can construct a counterexample where the transportation cost does not take into account the differences in direction of the matrix part of the density operators and phase space densities. 

\textbf{Counterexample:}
Take the matrix-valued von Neumann equation
\begin{equation}
    i\hbar \partial_t R = [\mathcal{H}_A,R], \quad \mathcal{H_A} = -\frac{\hbar^2}{2}\Delta \otimes A = \begin{pmatrix}
        -\frac{\hbar^2}{2}\Delta & 0 \\ 0 & 0
    \end{pmatrix}
\end{equation}
that is,
\begin{equation}
    i\hbar \partial_t \begin{pmatrix}
        R_{11} & R_{12} \\ \overline{R_{12}} & R_{22}
    \end{pmatrix}(X,Y) = -\frac{\hbar^2}{2} \begin{pmatrix}
        (\Delta_X-\Delta_Y) R_{11} & \Delta_X R_{12} \\ -\Delta_Y \overline{R_{12}} & 0 
    \end{pmatrix}(X,Y)
\end{equation}
In the Weyl variables $x=(X+Y)/2$ and $y=(X-Y)/\hbar$ we obtain the equation for $\Tilde{R}(x,y)$
\begin{equation}
    i\hbar \partial_t \begin{pmatrix}
        \tilde{R}_{11} & \tilde{R}_{12} \\ \overline{\tilde{R}_{12}} & \tilde{R}_{22}
    \end{pmatrix} = -\frac{\hbar^2}{2} \begin{pmatrix}
        \frac{2}{\hbar} \nabla_y\cdot\nabla_x \tilde{R}_{11} & (\frac{1}{4}\Delta_x + \frac{1}{\hbar} \nabla_y\cdot\nabla_x + \frac{1}{\hbar^2}\Delta_y)\tilde{R}_{12}  \\ -(\frac{1}{4}\Delta_x - \frac{1}{\hbar} \nabla_y\cdot\nabla_x + \frac{1}{\hbar^2}\Delta_y)\overline{\tilde{R}_{12}} & 0 
    \end{pmatrix}
\end{equation}
and the Wigner equation for the Wigner transform $W_{\hbar}$ after taking a Fourier transform in $y$:
\begin{equation}
    \partial_t \begin{pmatrix}
        W_{11} & W_{12} \\ \overline{W_{12}} & W_{22}
    \end{pmatrix} +  \begin{pmatrix}
        \xi\cdot\nabla_x W_{11} & (\frac{\hbar}{8}\Delta_x + \frac{1}{2} \xi\cdot\nabla_x + \frac{1}{i\hbar}\frac{|\xi|^2}{2})W_{12}  \\ -(\frac{\hbar}{8}\Delta_x - \frac{1}{2} \xi\cdot\nabla_x + \frac{1}{i\hbar}\frac{|\xi|^2}{2})\overline{W_{12}} & 0 
    \end{pmatrix} = 0
\end{equation}
This can be rewritten as
\begin{equation}
    \partial_t \begin{pmatrix}
        W_{11} & W_{12} \\ \overline{W_{12}} & W_{22}
    \end{pmatrix} + \frac{1}{i\hbar}\frac{|\xi|^2}{2} \begin{pmatrix}
        0 & W_{12} \\ \overline{W}_{12} & 0
        \end{pmatrix} + \xi\cdot\nabla_x\begin{pmatrix}
         W_{11} & \frac{1}{2}W_{12}  \\ \frac{1}{2}\overline{W_{12}} & 0 
    \end{pmatrix} + r_{\hbar}= 0
\end{equation}
where 
\begin{equation*}
    r_{\hbar} = -\frac{\hbar}{8}\Delta_x\begin{pmatrix}
        0 & -W_{12} \\ \overline{W}_{12} & 0
        \end{pmatrix}.
\end{equation*}
Moreover, take the matrix-valued phase space density $F=F^{\ast}$ solving
\begin{align}
    \partial_t F + \mathfrak{X} +\frac{1}{2}(\{\mathcal{H}_B,F\}-\{F,\mathcal{H}_B\}) &= 0, \quad \mathcal{H}_B = \frac{1}{2}|\xi|^2 \otimes B = \begin{pmatrix}
        0 & 0 \\ 0 & \frac{|\xi|^2}{2}
    \end{pmatrix},\\ [\mathcal{H}_B,F] &=0
\end{align}
that is,
\begin{align}
    \partial_t \begin{pmatrix}
        f_{11} & f_{12} \\ \overline{f_{12}} & f_{22}
    \end{pmatrix}  + \begin{pmatrix}
        0 & x_{12} \\ \overline{x_{12}} & 0
    \end{pmatrix} + \xi\cdot \nabla_x \begin{pmatrix}
        0 & \tfrac{1}{2} f_{12} \\ \tfrac{1}{2} \overline{f_{12}} & f_{22}
    \end{pmatrix} &= 0, \\
    \frac{|\xi|^2}{2}\begin{pmatrix}
        0 & -f_{12} \\ \overline{f_{12}} & 0
    \end{pmatrix} &= 0.
\end{align}
On the other hand, $f_{22}$ and $R_{11}$ both solve the free evolution:
\begin{equation}
\label{eq:FreeEvoluionScalar}
    \partial_t f_{22} + \xi \cdot \nabla_x f_{22} = 0, \quad i \hbar \partial_t R_{11} = -\frac{\hbar^2}{2}(\Delta_X-\Delta_Y)R_{11}.
\end{equation}
and $f_{11}$ and $R_{22}$ solve
\begin{equation}
    \partial_t f_{11} = 0, \quad i\hbar \partial_t R_{22} = 0.
\end{equation}
Thus, by choosing appropriate initial data, we can take $f_{11}=R_{22}=0$.

Now take an optimal coupling $q$ of $f_{22}$ and $R_{11}$. We construct a matrix coupling $Q$ such that $\tr_{\mathbb{C}^m}(Q)$ 
is the optimal coupling $q$ of $R_{11}$ and $f_{22}$. Indeed, by defining
\begin{align}
    Q:= \begin{cases}
        Q_{1122} = q,\\
        Q_{jk\ell m} = 0 \quad \text{otherwise},
    \end{cases}
\end{align}
we obtain that the marginals of $Q$ are $R = \text{diag}(R_{11},0)$ and $F = \text{diag}(0,f_{22})$ and that $q = \tr(Q)$.
With the diagonal cost \eqref{eq:CostDiagional} it holds that
\begin{align*}
    \mathfrak{d}_{\hbar}(F,R) &\leq \int_{\mathbb{R}_x^d\times \mathbb{R}_{\xi}^d} \Tr(Q^{\frac{1}{2}}(|x-y|^2 + |\xi- D_y|^2)Q^{\frac{1}{2}}) dx d\xi \\
    &= \int_{\mathbb{R}_x^d\times \mathbb{R}_{\xi}^d} \tr_{L^2}(\tr_{\mathbb{C}^m}(Q)^{\frac{1}{2}}(|x-y|^2 + |\xi- D_y|^2)\tr_{\mathbb{C}^m}(Q)^{\frac{1}{2}}) dx d\xi
     \\
    &= \int_{\mathbb{R}_x^d\times \mathbb{R}_{\xi}^d} \tr_{L^2}(q^{1/2}(|x-y|^2 + |\xi- D_y|^2)q^{1/2}) dx d\xi \\
    &= d_{\hbar}(f_{22},R_{11}).
\end{align*}
If we take $R_{11}$ as the Toeplitz quantization of $f_{22}$, the scalar Wasserstein distance $d_{\hbar}(f_{22},R_{11})$ between $f_{22}$ and $R_{11}$ is bounded by $d \hbar$ where $d$ is the dimension and thus, the pseudo-distance $\mathfrak{d}(F,R)$ is small in the limit. But the matrix-valued density $F$ is clearly not the limit of $W_{\hbar}[R]$.

This ``counterexample'' to naive quantum optimal transport for matrix-valued measures and density operators is inspired by a similar counterexample for the classical case in \cite{ning_matrix-valued_2015}.

\section*{Acknowledgement}
J.M. acknowledges funding by the Austrian Science Fund (FWF) via grant 10.55776/J4840.

We thank Herbert Spohn for raising the question about the rate of decoupling of the eigenmodes of the Hamiltonian (cf. Remark  \ref{RemarkCommutator}).

We thank Nikolai Leopold for pointing out the connection between the diluted mixed state condition and the fermionic mean field scaling (cf. Remark \ref{remark:DilutedMixedStates}).

    \bibliographystyle{abbrv}
    \bibliography{references}

\end{document}